\documentclass{article}

\newcommand{\bA}{{\mathbf{A}}}

\newcommand{\bB}{{\mathbf{B}}}

\newcommand{\bC}{{\mathbf{C}}}

\newcommand{\bd}{{\mathbf{d}}}

\newcommand{\cD}{{\mathcal{D}}}

\newcommand{\bfe}{{\mathbf{e}}}

\newcommand{\cE}{{\mathcal{E}}}

\newcommand{\bF}{{\mathbf{F}}}
\newcommand{\cF}{{\mathcal{F}}}

\newcommand{\bG}{{\mathbf{G}}}
\newcommand{\cG}{{\mathcal{G}}}

\newcommand{\bH}{{\mathbf{H}}}

\newcommand{\bI}{{\mathbf{I}}}
\newcommand{\cI}{{\mathcal{I}}}

\newcommand{\bJ}{{\mathbf{J}}}

\newcommand{\bK}{{\mathbf{K}}}

\newcommand{\cL}{{\mathcal{L}}}

\newcommand{\bM}{{\mathbf{M}}}

\newcommand{\cN}{{\mathcal{N}}}

\newcommand{\cO}{{\mathcal{O}}}

\newcommand{\bP}{{\mathbf{P}}}

\newcommand{\bQ}{{\mathbf{Q}}}

\newcommand{\bR}{{\mathbf{R}}}

\newcommand{\bS}{{\mathbf{S}}}
\newcommand{\cS}{{\mathcal{S}}}

\newcommand{\cT}{{\mathcal{T}}}

\newcommand{\bU}{{\mathbf{U}}}

\newcommand{\cU}{{\mathcal{U}}}

\newcommand{\bV}{{\mathbf{V}}}

\newcommand{\bW}{{\mathbf{W}}}

\newcommand{\bX}{{\mathbf{X}}}

\newcommand{\by}{{\mathbf{y}}}

\newcommand{\bY}{{\mathbf{Y}}}

\newcommand{\bz}{{\mathbf{z}}}

\newcommand{\bpi}{{\boldsymbol{\pi}}}

\newcommand{\s}{\sigma}

\newcommand{\bTh}{{\boldsymbol{\Theta}}}

\newcommand{\bPi}{{\boldsymbol{\Pi}}}

\newcommand{\expect}{\mathbb{E}}
\newcommand{\reals}{\mathbb{R}}

\newcommand{\compls}{\mathbb{C}}
\newcommand{\nats}{\mathbb{N}}

\newcommand{\ip}[2]{\left\langle{#1},{#2}\right\rangle}

\newcommand{\ceil}[1]{\lceil{#1}\rceil}
\newcommand{\floor}[1]{\lfloor{#1}\rfloor}

\newcommand{\avg}{\text{avg}}

\usepackage{arxiv}
\providecommand{\keywordsname}{Keywords}
\usepackage[utf8]{inputenc}
\usepackage[T1]{fontenc}    		
\usepackage{url}            			
\usepackage{amsthm}
\usepackage{booktabs}       		
\usepackage{amsfonts}       		
\usepackage{amsmath}
\usepackage{nicefrac}       		
\usepackage{microtype}      		
\usepackage[square,numbers]{natbib}
\usepackage{mathtools}
\usepackage{algpseudocode}
\usepackage{graphicx}
\usepackage{indentfirst,latexsym,bm}
\usepackage{amssymb}
\usepackage{xcolor}
\usepackage{comment}
\usepackage{enumitem}
\usepackage{dsfont}

\usepackage{algorithm}
\usepackage{algorithmicx}
\usepackage{float}
\usepackage{placeins}
\usepackage{subfig}
\usepackage{caption}
\usepackage{makecell}
\usepackage[flushleft]{threeparttable}
\usepackage{tikz}
\usetikzlibrary{arrows.meta,positioning}
\usepackage{pgfplots}
\pgfplotsset{compat=1.18}
\usepgfplotslibrary{fillbetween}
\usetikzlibrary{shapes,decorations}
\usetikzlibrary{fit}

\colorlet{color1}{blue}
\colorlet{color2}{red!50!black}

\definecolor{ivory}{RGB}{218,215,203}

\definecolor{cuhkp}{RGB}{98,56,105} 	
\definecolor{cuhkpl}{RGB}{152,24,147} 	
\definecolor{cuhkb}{RGB}{219,160,1} 	
\definecolor{cuhkbd}{RGB}{178,129,0} 	
\definecolor{cuhkr}{RGB}{88,35,155}  	

\usepackage{hyperref}
\hypersetup{
  hidelinks,
  hypertexnames=false,
  breaklinks=true,
}
\usepackage{bookmark}

\RequirePackage[capitalize,nameinlink]{cleveref}[0.19]

\crefname{section}{section}{sections}
\crefname{subsection}{subsection}{subsections}

\Crefname{figure}{Figure}{Figures}

\crefformat{equation}{\textup{#2(#1)#3}}
\crefrangeformat{equation}{\textup{#3(#1)#4--#5(#2)#6}}
\crefmultiformat{equation}{\textup{#2(#1)#3}}{ and \textup{#2(#1)#3}}
{, \textup{#2(#1)#3}}{, and \textup{#2(#1)#3}}
\crefrangemultiformat{equation}{\textup{#3(#1)#4--#5(#2)#6}}%
{ and \textup{#3(#1)#4--#5(#2)#6}}{, \textup{#3(#1)#4--#5(#2)#6}}{, and \textup{#3(#1)#4--#5(#2)#6}}

\Crefformat{equation}{#2Equation~\textup{(#1)}#3}
\Crefrangeformat{equation}{Equations~\textup{#3(#1)#4--#5(#2)#6}}
\Crefmultiformat{equation}{Equations~\textup{#2(#1)#3}}{ and \textup{#2(#1)#3}}
{, \textup{#2(#1)#3}}{, and \textup{#2(#1)#3}}
\Crefrangemultiformat{equation}{Equations~\textup{#3(#1)#4--#5(#2)#6}}%
{ and \textup{#3(#1)#4--#5(#2)#6}}{, \textup{#3(#1)#4--#5(#2)#6}}{, and \textup{#3(#1)#4--#5(#2)#6}}

\crefdefaultlabelformat{#2\textup{#1}#3}

\theoremstyle{plain}
\newtheorem{theorem}{Theorem}[section]
\newtheorem{thm}{Theorem}[section]
\newtheorem{lemma}[thm]{Lemma}
\newtheorem{corollary}[thm]{Corollary}

\newtheorem{assumption}[thm]{Assumption}

\theoremstyle{remark}
\newtheorem{remark}{Remark}[section]

\theoremstyle{definition}

\newcommand{\prt}[1]{\left(#1\right)}
\newcommand{\brk}[1]{\left[#1\right]}
\newcommand{\crk}[1]{\left\{#1\right\}}

\newcommand{\norm}[1]{\left\|#1\right\|}

\usepackage{bbm}
\newcommand{\mone}{\mathbbm{1}}

\newcommand{\T}{\top}

\title{Optimal Network Dependence in Distributed Stochastic Optimization via Tree Routing}

\author{%
  Runze You\\
  School of Data Science\\
  The Chinese University of Hong Kong, Shenzhen \\
  Shenzhen, Guangdong, China \\
  \texttt{runzeyou@link.cuhk.edu.cn} \\
  \And
  Shi Pu\thanks{Corresponding author.} \\
  School of Data Science \\
  The Chinese University of Hong Kong, Shenzhen \\
  Shenzhen, Guangdong, China \\
  \texttt{pushi@cuhk.edu.cn} \\
}
\date{}

\hypersetup{
  pdftitle={Optimal Network Dependence in Distributed Stochastic Optimization via Tree Routing},
  pdfauthor={Runze You and Shi Pu}
}

\begin{document}
\maketitle

\begin{abstract}
Communication is a central bottleneck in distributed optimization, but its
effect is often summarized by the spectral gap of a chosen mixing matrix.
Since this gap depends on link weights as well as topology, it can obscure the
intrinsic effect of the network. We clarify this distinction by relating graph
diameter to the inverse spectral gap. We revisit a classical one-sided
relation between these quantities and use graph families and alternative
weightings to show why it does not imply a universal two-sided scaling law. 
This motivates a class-based view of network optimality:
diameter-dependent guarantees can be assessed uniformly over graph classes,
whereas spectral-gap optimality requires both the graph--matrix pair class and
the admissible algorithm class to be specified. Based on the resulting
diameter-based minimax benchmark, we introduce Tree-Routed Gradient Tracking
(Tree-RGT), a distributed method for stochastic nonconvex optimization. The
method pipelines model dissemination and gradient aggregation over a rooted
shortest-path spanning tree. Only the root performs descent and retains
momentum, which combines the routed gradient aggregate with its history
without additional communication. Each iteration uses one exchange of model
and gradient-tracker information between tree neighbors and one
stochastic-gradient evaluation per agent, without repeated consensus steps
between successive optimization updates. Under smoothness and unbiased bounded-variance
oracles, Tree-RGT achieves uniformly optimal network dependence over connected
graphs of prescribed size and diameter. It matches the convergence rate of centralized mini-batch SGD
with batch size $n$ after $\cO(nD_{\cG}^2)$ transient iterations, where $n$ is the number
of agents and $D_{\cG}$ is the graph diameter. Across all graph families
considered, this transient bound strictly improves upon the reported bounds
of all compared methods using one communication round per update. It also
improves upon or matches the reported bounds of all compared methods using
multiple communication rounds between successive optimization updates.
\end{abstract}


\section{Introduction}
\label{sec:introduction}
Consider $n$ agents connected by a fixed, connected, undirected communication
graph $\cG=(\cN,\cE)$. Agent $i$ has access only to a local stochastic
objective $f_i:\reals^p\to\reals$ and can communicate directly only with its
one-hop neighbors. The agents seek to solve
\begin{equation}
    \label{eq:obj}
    \min_{x\in\reals^p}
    \quad
    f(x):=\frac{1}{n}\sum_{i=1}^n f_i(x),
\end{equation}
where
$f_i(x):=\expect_{\xi_i\sim\cD_i}[F_i(x;\xi_i)]$ and $\cD_i$ is the local
data distribution at agent $i$. Problem~\eqref{eq:obj} encompasses
heterogeneous stochastic optimization and distributed empirical risk
minimization, since the local distributions need not be identical. Problems of
this form arise in distributed estimation and large-scale learning, where
data and computation are spread across agents and communication is constrained
by the underlying network
\cite{nedic2018network,lian2017can,yuan2022decentralized}. The stochastic
first-order oracle used by each agent is specified formally in
Assumption~\ref{a.var}.

Given the graph $\cG$, decentralized methods such as decentralized stochastic
gradient descent (DSGD)
\cite{lian2017can,pu2021sharp,yuan2023removing}, decentralized stochastic
gradient tracking (DSGT)
\cite{pu2021distributed,koloskova2021improved,alghunaim2022unified}, and
Push--Pull \cite{pu2020push,you2026stochastic,liang2025linear} update local
states by exchanging messages with neighboring agents. The agents assign
weights to the states they retain, transmit, or receive. When these
coefficients are collected across the network, they form one or more
mixing matrices whose off-diagonal supports are constrained by $\cG$. Much of
the literature takes these matrices as part of the problem specification. For
methods over undirected graphs that repeatedly apply a symmetric doubly
stochastic matrix $\bW$, network-dependent terms are commonly expressed
through the inverse spectral gap $(1-\lambda_{\bW})^{-1}$, where
$\lambda_{\bW}$ is the second-largest eigenvalue modulus. This quantity also
appears in transient-iteration bounds that identify when a decentralized
stochastic method begins to match the iteration and agent scaling of
centralized mini-batch SGD \cite{pu2020asymptotic,pu2021sharp,
koloskova2021improved,
alghunaim2022unified,yuan2023removing}.

Considerable effort has therefore been devoted to improving the dependence on
the inverse spectral gap. In deterministic smooth strongly convex optimization,
Optimal Gradient Tracking uses loopless Chebyshev acceleration within a
single-loop recursion \cite{song2024optimal}. Distributed Stochastic Momentum
Tracking combines this mechanism with momentum tracking for stochastic
nonconvex optimization and uses one neighbor exchange per iteration, without
an inner communication loop or the accumulation of a large stochastic-gradient
batch between successive updates \cite{huang2026accelerated}. A different strategy
uses several consensus rounds for each effective optimization update, through
accelerated multi-gossip \cite{scaman2017optimal,yuan2022revisiting} or
factorized finite-time averaging \cite{lu2021optimal}. Such schemes can reduce
disagreement per update, but they separate the communication-round count from
the number of completed updates.

The spectral gap is not the only network descriptor used in distributed
optimization. For directed communication, analyses of Push-DIGing-type
methods combine a generalized spectral gap with the equilibrium skewness of a
column-stochastic matrix \cite{liang2025understanding}; analogous effective
metrics have recently been developed for row-stochastic matrices
\cite{liang2025row}. Routing-based methods expose graph distances more
directly: STPP propagates information over two directed spanning trees and
expresses its network dependence through their maximum and average root
distances \cite{you2025distributed}. For time-varying directed networks,
one-step contraction parameters have been derived, under strong connectivity
of every instantaneous graph, explicitly in terms of graph diameter, maximal
edge-utility (a shortest-path congestion measure), and the mixing weights
\cite{nedic2025ab}. These examples show that the appropriate network quantity
depends on both the communication primitive and the network model.

Several works further ask whether an algorithm has an optimal dependence on
such quantities. The answer is necessarily relative to a comparison class.
Spectral lower bounds and accelerated algorithms have been developed for
matrix-constrained classes with prescribed spectral connectivity
\cite{scaman2017optimal,yuan2022revisiting,liang2025understanding}, whereas a
tight stochastic nonconvex minimax complexity has been established over graph
classes with prescribed size and diameter \cite{lu2021optimal}. The latter is
a statement about physical graphs and one-hop information propagation; the
former statements additionally restrict the weighted communication operator
available to an algorithm. Calling either dependence ``network optimal''
without identifying the network class and admissible algorithm class can
therefore conflate distinct minimax problems.

This distinction leads to the central question of this paper:
\begin{center}
    \emph{Can a distributed stochastic method use one round of one-hop
    communication per iteration and still achieve uniformly optimal network
    dependence over a prescribed graph class?}
\end{center}

We answer this question in three steps. First, we separate graph invariants
from matrix-dependent quantities
by revisiting a classical one-sided diameter--\allowbreak inverse-gap relation and
demonstrating the range of possible scalings. Second, we revisit what uniform
network optimality means over graph classes and over graph--matrix pair
classes. Third, we develop Tree-Routed Gradient Tracking (Tree-RGT), which
pipelines model dissemination and gradient aggregation on a rooted
shortest-path spanning tree. Only the root performs descent and retains
momentum, combining delayed gradient information from all agents with a
geometrically weighted history without additional communication. Tree-RGT uses one communication round and one
stochastic-gradient evaluation per agent at each iteration, without
repeated neighbor-averaging steps between successive optimization updates,
and achieves uniformly optimal network
dependence over connected undirected graph classes of prescribed size and
diameter. With problem and initialization quantities fixed, its
time-averaged expected squared gradient norm over $T$ iterations attains
the $\cO((nT)^{-1/2})$ rate of centralized mini-batch SGD after
$\cO(nD_{\cG}^2)$ transient iterations.

\subsection{Related Work}
\label{sec:related-work}

The most relevant literature falls into three lines: network descriptors used
in convergence analysis, methods that couple one communication round to each
update, and complexity-optimal methods based on multi-round or phase-separated
communication.

\textbf{Network descriptors and graph--matrix relations.}
Classical consensus and decentralized optimization analyses quantify
communication through products of stochastic matrices
\cite{nedic2009distributed,nedic2018network}. On a fixed undirected graph, the
choice of admissible weights directly controls the mixing time, and optimizing
these weights is itself a well-studied problem
\cite{xiao2004fast,boyd2004fastest}.

Directed networks require additional descriptors because the stationary
weights of the chosen communication matrices (their normalized left or right
Perron eigenvectors) need not be uniform. Push-sum and Push--Pull methods use
column-stochastic or coupled row- and column-stochastic matrices
\cite{xi2017add,pu2020push,xin2020general}. A recent stochastic Push--Pull
analysis uses infinite sums of matrix-norm terms to encompass several existing
algorithms, although the meaning of these quantities remains unclear
\cite{you2026stochastic}. Other recent
complexity analyses identify equilibrium skewness, together with generalized
spectral gaps, as an effective measure of the resulting imbalance
\cite{liang2025understanding,liang2025row}. For time-varying networks, a common
model instead assumes connectivity of the union graph over every fixed-length
window \cite{nedic2015distributed,nedic2017achieving}. A more explicit
contraction analysis relates time-varying Push--Pull matrices to graph diameter
and maximal edge-utility \cite{nedic2025ab}. These quantities address distinct
directed or time-varying network models; analyses of static undirected networks
more commonly begin with a prescribed symmetric mixing matrix.

\textbf{One-communication-per-update and routing-based methods.}
DSGD \cite{lian2017can}, gradient-tracking methods such as DIGing
\cite{nedic2017achieving} and DSGT \cite{pu2021distributed}, and
correction-based methods such as EXTRA \cite{shi2015extra}, $D^2$
\cite{tang2018d}, and EDAS \cite{huang2022improving} interleave local
computation with one neighbor exchange per update. Their network-dependent
transient-iteration bounds have been progressively sharpened
\cite{pu2020asymptotic,pu2021sharp,koloskova2021improved,
alghunaim2022unified,yuan2023removing}. OGT \cite{song2024optimal} and DSMT
\cite{huang2026accelerated} further improve spectral-gap dependence through
loopless Chebyshev acceleration without an inner gossip loop.

Momentum also shapes the optimization direction in distributed
stochastic methods. Quasi-global momentum uses successive mixed-model
displacements to form a momentum buffer at each agent and mitigate the
optimization difficulties caused by heterogeneous data
\cite{lin2021quasi}. DSMT maintains local stochastic momentum and tracks its
network average using loopless Chebyshev acceleration
\cite{huang2026accelerated}. These constructions couple momentum with
mixing-based updates at every agent. In Tree-RGT, tree routing instead
concentrates the descent direction at the root, so only one momentum vector
is needed for the entire network. The resulting analysis combines finite
routing delays with geometric momentum memory and retains diameter-based
network guarantees.

Routing offers another single-round mechanism. RelaySum propagates delayed
information on a spanning tree \cite{vogels2021relaysum}, whereas BTPP 
\cite{you2024b} and STPP \cite{you2025distributed} route models and
gradient trackers through designed trees. Finite-time methods instead rely on
structured communication-matrix sequences. DSGD-CECA is built around the
prescribed CECA sequence and a fully controllable communication topology
\cite{ding2023dsgd}. Finite-time-consensus gradient tracking gives explicit
sequences only for selected topology families; no general construction within
an arbitrary prescribed graph is known, and its convergence guarantee depends
explicitly on the sequence length, which is likewise unavailable without such
a construction \cite{nguyen2025graphs}.

\textbf{Multi-round communication and complexity-optimal methods.}
Under prescribed symmetric-matrix models, accelerated gossip yields optimal
spectral-gap, communication, and gradient complexities for deterministic
smooth strongly convex problems
\cite{scaman2017optimal,kovalev2020optimal}. In stochastic smooth nonconvex
optimization, MG-DSGD \cite{yuan2022revisiting} and DeTAG \cite{lu2021optimal} 
both use multi-step Chebyshev-accelerated
gossip based on a prescribed matrix between effective updates. Both combine it
with gradient accumulation, while DeTAG additionally maintains a gradient
tracker. Their corresponding bounds match the prescribed-matrix lower bounds
up to logarithmic factors. Directed extensions include MG-Push-DIGing, based
on multi-step gossip with a prescribed column-stochastic matrix
\cite{liang2025understanding}, and MG-PULL-DIAG-GT, its row-stochastic
counterpart \cite{liang2025row}. For stochastic convex optimization, DDA-SGD
combines minibatching with accelerated gossip \cite{kluger2026nearoptimal},
whereas MG-ADSGD couples minibatching with multi-round accelerated gossip
\cite{sun2026accelerated}.

DeFacto \cite{lu2021optimal} instead attains the tight diameter-minimax scaling
by factorizing exact averaging into graph-dependent communication matrices,
while treating this factorization as a black-box preprocessing step and
separating computation from communication.

\subsection{Main Contributions}
We make three main contributions:

\begin{itemize}
    \item \textbf{Graph invariants and class-based network optimality.}
    We distinguish the graph-invariant diameter from matrix-dependent
    spectral gaps, revisit the classical one-sided relation
    $\tau_{\bW}\ge(D_{\cG}-1)/\log n$
    \cite{chung1989diameters,chung1997spectral}, and verify its tightness up to
    constant factors on constant-degree expanders. Representative graph
    families exhibit incompatible diameter--inverse-gap scalings even under
    standard weighting schemes, while reweighting can substantially change
    the gap on a fixed graph. We define \emph{uniformly optimal network
    dependence} by matching the minimax powers of network size and diameter
    uniformly over a prescribed graph class. Revisiting the discussion of
    inverse-gap optimality in \cite{lu2021optimal}, we clarify that worst-case
    bounds for fixed-matrix gossip do not, by themselves, establish optimality
    over arbitrary graph--matrix pair classes; the admissible use of the
    designated matrix must also be specified through the algorithm class.

    \item \textbf{Tree-RGT with uniformly optimal network dependence.}
    We develop Tree-Routed Gradient Tracking (Tree-RGT) on a rooted
    shortest-path spanning tree constructed by a one-time distributed
    procedure. Its tree-based pipeline confines descent updates and momentum
    storage to the root, while other agents forward delayed root iterates and
    return gradient-tracking information. The root applies momentum to the
    routed aggregate, incorporating gradient history from all agents with
    one additional local vector and no additional communicated variable. Each iteration uses one neighbor-exchange
    round and one stochastic-gradient evaluation per agent, without repeated
    averaging steps between updates. For smooth nonconvex objectives and
    unbiased bounded-variance stochastic oracles, Tree-RGT achieves
    \emph{uniformly optimal network dependence} over connected undirected
    graphs of prescribed size and diameter. To the best of our knowledge,
    it is the first distributed stochastic method to do so with one
    communication round per update.

    \item \textbf{Tree-based momentum analysis and transient guarantees.}
    We derive exact representations that separate the finite-memory
    tree routing from the geometric history of root momentum.
    A joint push--pull propagation estimate preserves the unique arrival
    layer of each gradient source and controls its convolution with the
    momentum weights without additional tree-depth factors.
    Because momentum is retained only at the root, no inter-agent momentum
    discrepancy needs to be controlled. Layerwise noise orthogonality,
    bounds on the correlations introduced by momentum, and coupled estimates
    of root increments and node-to-root disagreement close the descent
    argument. The resulting convergence bound permits the momentum
    coefficient to approach one with the iteration budget, and yields
    $\cO(nD_{\cG}^2)$ transient iterations: with problem and initialization
    quantities fixed, the averaged expected squared gradient norm then
    matches the $\cO((nT)^{-1/2})$ rate of centralized mini-batch SGD.
    The same guarantee holds for every smaller momentum coefficient at the
    prescribed stepsize.
    Across all graph families considered,
    this transient bound strictly improves upon the reported bounds of all
    compared methods using one communication round per update
    (Table~\ref{tab:one-communication-comparison}). It also improves upon or
    matches the reported bounds of all compared methods using multiple
    communication rounds between successive optimization updates
    (Table~\ref{tab:multi-communication-comparison}).
\end{itemize}

\begin{table}[ht]
    \begin{threeparttable}
    \centering
    \small
        \setlength{\tabcolsep}{5pt}
        \renewcommand{\arraystretch}{1.30}
        \begin{tabular*}{\textwidth}{@{\extracolsep{\fill}}lcccc@{}}
            \toprule
            Method
            & \makecell{Transient\\iterations\tnote{(a)}}
            & \makecell{Ring graph\tnote{(b)}\\($\bW_{\mathrm{r}}$)}
            & \makecell{Exponential\\graph ($\bW_{\mathrm{ex}}$)}
            & \makecell{Double-star\\graph ($\bW_{\mathrm{ds}}$)} \\
            \midrule
            DSGD \cite{lian2017can}
            & $n^3\tau_{\bW}^4$
            & $n^{11}$
            & $n^3\log^4 n$
            & $n^{11}$ \\
            DSGT \cite{alghunaim2022unified}
            & $\max\crk{n^3\tau_{\bW}^2,n\tau_{\bW}^{8/3}}$
            & $n^{7}$
            & $n^3\log^2 n$
            & $n^{7}$ \\
            DSMT \cite{huang2026accelerated}
            & $n^{5/3}\tau_{\bW}$
            & $n^{11/3}$
            & $n^{5/3}\log n$
            & $n^{11/3}$ \\
            RelaySum \cite{vogels2021relaysum}
            & $n^3D_\cG^4$
            & $n^7$
            & $n^3\log^4 n$
            & $n^3$ \\
            STPP\tnote{(c)} \cite{you2025distributed}
            & $nD_\cG^6$
            & $n^7$
            & $n\log^6 n$
            & $n$ \\
            \midrule
            \makecell[l]{\textbf{Tree-RGT}\\\textbf{(this paper)}}
            & $\bm{nD_{\cG}^{2}}$
            & $\bm{n^3}$
            & $\bm{n\log^2 n}$
            & $\bm{n}$ \\
            \bottomrule
        \end{tabular*}
        \begin{tablenotes}
            \footnotesize
            \item[(a)] Each displayed expression is the network-dependent term
            in an asymptotic upper bound on the number of transient iterations.
            The common $\cO(\cdot)$ notation is omitted; problem- and
            initialization-dependent quantities, treated as constants with
            respect to the network parameters, and method-specific assumptions
            are also suppressed. Here
            $\tau_{\bW}=(1-\lambda_{\bW})^{-1}$.
            \item[(b)] This and the next two columns specialize the
            transient-iteration bounds to the indicated physical graph families.
            Matrix-based methods use the displayed matrices. DSMT uses
            $(\bI+\bW)/2$ when a displayed matrix is not positive semidefinite,
            which preserves the reported asymptotic orders. RelaySum, STPP,
            and Tree-RGT use only the corresponding physical graphs.
            \item[(c)] The STPP bound in \cite{you2025distributed} depends on
            maximum and average root distances in its two routing trees.
            Choosing the two trees as opposite orientations of a rooted
            shortest-path tree bounds these quantities by $D_{\cG}$ and yields
            the conservative network-dependent term $nD_{\cG}^{6}$ reported
            here.
        \end{tablenotes}
        \vspace{4pt}
        \caption{Comparison of the convergence speed of Tree-RGT and
        representative distributed stochastic methods using one communication
        round per update, measured through the network-dependent terms in their
        reported transient-iteration upper bounds for smooth nonconvex
        objectives.}
        \label{tab:one-communication-comparison}
    \end{threeparttable}
\end{table}

\begin{table}[ht]
    \begin{threeparttable}
    \centering
    \small
        \setlength{\tabcolsep}{1.5pt}
        \renewcommand{\arraystretch}{1.30}
        \begin{tabular*}{\textwidth}{@{\extracolsep{\fill}}lcccccc@{}}
            \toprule
            Method
            & \makecell{Transient\\iterations\tnote{(a)}}
            & \makecell{Ring graph\\($\bW_{\mathrm{r}}$)}
            & \makecell{Exponential\\graph\\($\bW_{\mathrm{ex}}$)}
            & \makecell{Double-star\\graph\\($\bW_{\mathrm{ds}}$)}
            & \makecell{UOND\tnote{(b)}}
            & \makecell{One comm.\\round per\\update\tnote{(c)}} \\
            \midrule
            DeFacto \cite{lu2021optimal}
            & $nD_{\cG}^{2}$
            & $n^3$
            & $n\log^2 n$
            & $n$
            & $\checkmark$
            & $\times$ \\
            DeTAG \cite{lu2021optimal}
            & $n\tau_{\bW}\log n$
            & $n^3\log n$
            & $n\log^2 n$
            & $n^3\log n$
            & $\times$
            & $\times$ \\
            MG-DSGD \cite{yuan2022revisiting}
            & $n\tau_{\bW}\log n$
            & $n^3\log n$
            & $n\log^2 n$
            & $n^3\log n$
            & $\times$
            & $\times$ \\
            \makecell[l]{MG-Push-DIGing\\\cite{liang2025understanding}}
            & $n\tau_{\bW}^2$
            & $n^5$
            & $n\log^2 n$
            & $n^5$
            & $\times$
            & $\times$ \\
            \makecell[l]{MG-PULL-\\DIAG-GT \cite{liang2025row}}
            & $n\tau_{\bW}^2$
            & $n^5$
            & $n\log^2 n$
            & $n^5$
            & $\times$
            & $\times$ \\
            \midrule
            \makecell[l]{\textbf{Tree-RGT}\\\textbf{(this paper)}}
            & $\bm{nD_{\cG}^{2}}$
            & $\bm{n^3}$
            & $\bm{n\log^2 n}$
            & $\bm{n}$
            & $\bm{\checkmark}$
            & $\bm{\checkmark}$ \\
            \bottomrule
        \end{tabular*}
        \begin{tablenotes}
            \footnotesize
            \item[(a)] Transient iterations count all synchronous rounds, including both
            computation/update rounds and communication-only rounds. Each displayed
            expression estimates the network-dependent number of such rounds required
            before the reported rate matches the scaling of centralized mini-batch SGD
            with one sample per agent per update ($B=1$). The asymptotic notation and
            parameter-suppression conventions follow
            Table~\ref{tab:one-communication-comparison}.
            \item[(b)] UOND abbreviates \emph{uniformly optimal network
            dependence}. A checkmark means that, with problem-dependent
            quantities fixed, the cited round-complexity bound holds uniformly
            over the corresponding graph class or graph--matrix pair class and
            matches the associated minimax powers of its network parameters.
            A cross means only that the cited guarantee does not establish UOND
            on such a class. Tree-RGT achieves UOND over the class of connected undirected graphs with \(n\) nodes and diameter \(\hat D\).
            \item[(c)] A checkmark means that each stochastic-update iteration
            uses one synchronous neighbor-exchange round and inserts no inner
            consensus or multi-gossip phase between successive updates.
            Multiple state variables may be transmitted concurrently within
            that round.
        \end{tablenotes}
        \vspace{4pt}
        \caption{Comparison of the convergence speed of Tree-RGT and
        representative distributed stochastic methods that insert multi-round
        or phase-separated communication between effective updates, measured
        through their network-dependent transient-iteration estimates for smooth
        nonconvex objectives. The graph-family specializations follow
        Table~\ref{tab:one-communication-comparison}.}
        \label{tab:multi-communication-comparison}
    \end{threeparttable}
\end{table}

\subsection{Notation and Assumptions}
\label{sec:notations}
We collect here the notations and assumptions used throughout the paper.

\textbf{Basic notation.}
For an integer $m\ge1$, let $[m]:=\{1,\ldots,m\}$. Throughout the paper, all
vectors are column vectors. For vectors $a$ and $b$, $\ip{a}{b}=a^\T b$ denotes their
Euclidean inner product and $\norm{a}$ denotes the Euclidean norm. For a vector
$x$, $\brk{x}_q$ denotes its $q$-th entry. For a matrix $\bA$,
$\norm{\bA}_2$ and $\norm{\bA}_F$ denote the spectral and Frobenius norms,
respectively; $\brk{\bA}_{ij}$ and
$\brk{\bA}_{i:}$ denote its $(i,j)$-th entry and its $i$-th row. We use
$\bI_m$, $\mone_m$, and $\mathbf{0}$ for the identity matrix, all-ones vector,
and zero matrix or vector of the indicated dimensions, and omit the subscript
when the dimension is clear.

For two positive sequences $\{a(n)\}_{n\ge1}$ and $\{b(n)\}_{n\ge1}$, we
write $a(n)=\cO(b(n))$ if there exist constants $C>0$ and $n_0\ge1$ such that
$a(n)\le Cb(n)$ for all $n\ge n_0$. Similarly, $a(n)=\Omega(b(n))$ if there
exist constants $c>0$ and $n_0\ge1$ such that $a(n)\ge cb(n)$ for all
$n\ge n_0$. Finally, $a(n)=\Theta(b(n))$ if both $a(n)=\cO(b(n))$ and
$a(n)=\Omega(b(n))$ hold.

\textbf{Graph and communication notation.}
As a standing convention, unless explicitly stated otherwise, every physical
communication graph
$\cG=(\cN,\cE)$ considered in our problem formulation and theoretical results
is static, connected, and undirected, with $\cN=[n]$. The directed-graph
notation below is used only to encode matrix support and to discuss explicitly
identified extensions.

A directed graph $\cG=(\cN,\cE)$ consists of the node set
$\cN=[n]$ and an arc set
$\cE\subseteq\{(j,i):i,j\in\cN,\ i\ne j\}$. The arc $(j,i)\in\cE$ means
that node $j$ can send information directly to node $i$. An undirected
physical communication graph is represented by a symmetric arc set: for every
pair of distinct nodes $i,j\in\cN$,
\[
    (j,i)\in\cE
    \quad\Longleftrightarrow\quad
    (i,j)\in\cE.
\]
Each pair of opposite arcs is identified with one undirected physical link
when discussing paths, distances, or other undirected graph properties.

For a nonnegative matrix $\bA\in\reals^{n\times n}$, define its directed
off-diagonal support graph $\cG_{\bA}=(\cN,\cE_{\bA})$ by
\[
    \cE_{\bA}
    :=
    \bigl\{(j,i):i,j\in\cN,\ i\ne j,\ \brk{\bA}_{ij}>0\bigr\}.
\]
The matrix $\bA$ is \emph{compatible} with $\cG$ if
$\cE_{\bA}\subseteq\cE$, or equivalently if
$\brk{\bA}_{ij}=0$ whenever $i\ne j$ and $(j,i)\notin\cE$. Diagonal entries
represent local state retention and require no physical communication. If
$\bA$ is symmetric, its symmetric arc support is identified with the
corresponding simple undirected support graph.

When $\cG$ is a connected undirected physical communication graph, a matrix
$\bW\in\reals^{n\times n}$ is called a \emph{mixing matrix on $\cG$} if it
is nonnegative, symmetric, and doubly stochastic, and its off-diagonal support
agrees exactly with $\cG$:
\[
    \brk{\bW}_{ij}>0
    \quad\Longleftrightarrow\quad
    (j,i)\in\cE,
    \qquad i\ne j.
\]
Equivalently, $\cE_{\bW}=\cE$, or $\cG=\cG_{\bW}$ after identifying each
pair of opposite arcs with one undirected link. We call $(\cG,\bW)$ a
\emph{graph--matrix pair}. Thus, every mixing matrix on $\cG$ is compatible
with $\cG$, but it additionally assigns positive weight to every physical
link.

\textbf{Algorithmic notation.}
For local variables $x_i^{(t)},y_i^{(t)}\in\reals^p$, we use
$\bX^{(t)},\bY^{(t)}\in\reals^{n\times p}$ as their respective row-wise
stacked variables, with
$\brk{\bX^{(t)}}_{i:}=(x_i^{(t)})^\T$ and
$\brk{\bY^{(t)}}_{i:}=(y_i^{(t)})^\T$. The stochastic-gradient stack
$\bG^{(t)}$ is defined by
\[
    \brk{\bG^{(t)}}_{i:}
    :=g_i(x_i^{(t)},\xi_i^{(t)})^\T.
\]
Here, $g_i(x,\xi_i)$ denotes the stochastic first-order oracle available to
agent $i$, whose properties are specified in Assumption~\ref{a.var}.
For any $\bU\in\reals^{n\times p}$ whose $i$-th row is $u_i^\T$, define the
row-wise gradient mapping $\nabla\bF(\bU)$ by
\[
    \brk{\nabla\bF(\bU)}_{i:}:=\nabla f_i(u_i)^\T.
\]

\textbf{Standing assumptions.}
The local objectives may be nonconvex and satisfy the following regularity
condition.
\begin{assumption}[Smooth objectives]
    \label{a.smooth}
    Each $f_i:\reals^p\to\reals$ is continuously differentiable and has an
    $L$-Lipschitz gradient: for all $x,y\in\reals^p$,
    \[
        \norm{\nabla f_i(x)-\nabla f_i(y)}
        \le L\norm{x-y}.
    \]
    Moreover, the global objective is bounded below:
    $f^\star:=\inf_{x\in\reals^p}f(x)>-\infty$.
\end{assumption}
In particular, the average objective $f=n^{-1}\sum_{i=1}^n f_i$ is also
$L$-smooth.

Each agent accesses its local objective through a stochastic first-order
oracle $g_i(x,\xi_i)$. We use fresh samples at different agents and oracle
calls.
\begin{assumption}[Stochastic first-order oracles]
    \label{a.var}
    For every $i\in\cN$ and $t\ge0$, the sample
    $\xi_i^{(t)}\sim\cD_i$, and the family
    $\{\xi_i^{(t)}:i\in\cN,\ t\ge0\}$ is mutually independent. For every
    $i\in\cN$ and $x\in\reals^p$,
    \[
        \expect_{\xi_i\sim\cD_i}\brk{g_i(x,\xi_i)}
        =\nabla f_i(x),
        \qquad
        \expect_{\xi_i\sim\cD_i}
        \brk{\norm{g_i(x,\xi_i)-\nabla f_i(x)}^2}
        \le\sigma^2
    \]
    for some $\sigma^2>0$.
\end{assumption}
    Unless particular sources of randomness are specified, every expectation
in this paper is taken with respect to all sources of randomness.

\subsection{Organization of the Paper}

The remainder of the paper is organized as follows.
Section~\ref{sec:spec} distinguishes graph invariants from matrix-dependent
spectral quantities and motivates diameter as an intrinsic measure of network
dependence. Section~\ref{sec:lower-bounds} formalizes minimax optimality through
convergence guarantees that hold uniformly over graph classes and recalls the
diameter-dependent complexity bounds. Section~\ref{sec:tree-rgt} constructs the
spanning-tree routing mechanism, introduces Tree-RGT, and states its main
convergence guarantee. Section~\ref{sec:convergence-analysis} develops the
analysis underlying this guarantee. Section~\ref{sec:numerical-experiments}
presents numerical experiments, and Section~\ref{sec:conclusion} concludes the
paper. Supporting proofs are collected in the appendices.

\section{Graph Diameter as an Intrinsic Measure of Network Dependence}
\label{sec:spec}

Convergence rates of decentralized algorithms based on repeated mixing are
often characterized by the spectral gap of a chosen mixing matrix. This
quantity is not a graph invariant: the graph determines which communication
links are available, whereas the mixing matrix assigns weights to those links.
Consequently, different mixing matrices on the same graph can have
substantially different spectral gaps.

To formalize this distinction, consider a graph--matrix pair
$(\cG,\bW)$ as defined in Section~\ref{sec:notations}. The graph component
$\cG=(\cN,\cE)$ has $n=|\cN|$ nodes and diameter
\[
    D_{\cG}
    :=
    \max_{i,j\in\cN}\operatorname{dist}_{\cG}(i,j),
\]
where $\operatorname{dist}_{\cG}(i,j)$ is the length of a shortest path
between $i$ and $j$. The matrix component $\bW$ specifies the mixing weights
on the links of $\cG$. Define
\[
    \lambda_{\bW}:=\norm{\bW-\bJ}_2,
    \qquad
    \tau_{\bW}:=\frac{1}{1-\lambda_{\bW}}
    \quad\text{when }\lambda_{\bW}<1,
\]
where $\bJ=\mone\mone^\T/n$. Since $\bW$ is symmetric and doubly stochastic,
$\lambda_{\bW}$ is its second-largest eigenvalue modulus.
Accordingly, $1-\lambda_{\bW}$ is the absolute spectral gap and
$\tau_{\bW}$ is its inverse. For brevity, we refer to them below as the
spectral gap and the inverse spectral gap, respectively.

The pair therefore carries two different kinds of network descriptors:
$n$ and $D_{\cG}$ depend only on the graph component and are graph invariants,
whereas $\lambda_{\bW}$ and $\tau_{\bW}$ depend on the selected matrix
component. We first show that $\tau_{\bW}$ can change by an arbitrary factor
while $\cG$ remains fixed. We then revisit a classical one-sided constraint in terms of
$n$ and $D_{\cG}$ and illustrate the possible relationships between diameter
and inverse spectral gap through several graph--matrix pairs.

\subsection{The Inverse Spectral Gap Is Not a Graph Invariant}

We begin with the strongest possible fixed-topology separation: even after the
entire support graph has been fixed, the inverse spectral gap can still be
changed by an arbitrary $n$-dependent factor.

\begin{theorem}[Arbitrary inverse-gap separation on a fixed graph]
    \label{thm:matrix-dependent-gap}
    Let $\cG=(\cN,\cE)$ be a connected undirected graph with $n\ge2$ nodes,
    and let $\bW$ be a positive semidefinite mixing matrix on $\cG$ with
    $\lambda_{\bW}<1$. For any $\delta\in(0,1]$, define
    \[
        \bW^{(\delta)}
        :=
        (1-\delta)\bI+\delta\bW.
    \]
    Then $\bW^{(\delta)}$ is a positive semidefinite mixing matrix on $\cG$.
    Moreover,
    \[
        1-\lambda_{\bW^{(\delta)}}
        =
        \delta\prt{1-\lambda_{\bW}},
        \qquad
        \tau_{\bW^{(\delta)}}
        =
        \frac{\tau_{\bW}}{\delta}.
    \]
\end{theorem}
\begin{proof}
    See Appendix \ref{pf:thm:matrix-dependent-gap}.
\end{proof}

The transformation in Theorem \ref{thm:matrix-dependent-gap} changes only the
amount of weight assigned to communication relative to the diagonal self-loop.
It preserves not only connectivity but every off-diagonal zero and nonzero
entry.  
Taking $\delta=n^{-p}$ for any constant $p>0$ in Theorem
\ref{thm:matrix-dependent-gap} gives a matrix with the same support graph as
$\bW$ but with
\[
    \tau_{\bW^{(n^{-p})}}=n^p\tau_{\bW}.
\]
Thus, fixing the network topology does not determine the inverse spectral gap.  
A convergence guarantee expressed through
$\tau_{\bW}$ therefore describes the behavior of the selected weighted mixing
protocol, rather than a property of the support graph alone.  Nevertheless,
the topology still imposes constraints shared by every mixing matrix on
$\cG$. We next show that graph diameter provides such a constraint.

\subsection{A Graph-Invariant Constraint: Diameter}

Theorem \ref{thm:matrix-dependent-gap} naturally raises a second question: if
the inverse spectral gap is not itself a graph invariant, what aspects of it
are forced by the topology?  

The most basic obstruction is finite-speed
information propagation. There exist two nodes at distance $D_{\cG}$, and no
sequence of fewer than $D_{\cG}$ one-hop communications can carry information
from one to the other. This graph-invariant obstruction imposes the following
universal, although one-sided, restriction on every mixing matrix on $\cG$.

The relation below follows from classical spectral-convergence estimates
\cite{diaconis1991geometric,chung1997spectral}, and related spectral bounds
on graph diameter appear in \cite{chung1989diameters,chung1994upper}.

\begin{theorem}[Diameter lower bound for the inverse spectral gap]
    \label{l:spectral}
    Let $\cG=(\cN,\cE)$ be a connected undirected graph with $n\ge2$ nodes,
    and let $\bW$ be a mixing matrix on $\cG$ with $\lambda_{\bW}<1$. Then
    \[
    \tau_{\bW}=\frac{1}{1-\lambda_{\bW}} \ge \frac{D_{\cG}-1}{\log n}.
    \]
\end{theorem}
\begin{proof}
    See Appendix \ref{pf:l:spectral}.
\end{proof}

The four graph families used in the remainder of this section are illustrated
in Figure~\ref{fig:section2-graphs}. Each displayed undirected link represents
the two opposite ordered pairs in the symmetric off-diagonal support of a
mixing matrix; diagonal self-loops are omitted.
\begin{figure}[t]
    \centering
    \subfloat[]{
        \includegraphics[width=0.22\textwidth]{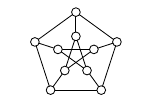}
    }
    \hfill
    \subfloat[]{
        \includegraphics[width=0.22\textwidth]{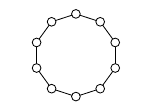}
    }
    \hfill
    \subfloat[]{
        \includegraphics[width=0.22\textwidth]{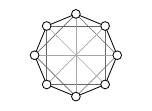}
    }
    \hfill
    \subfloat[]{
        \includegraphics[width=0.22\textwidth]{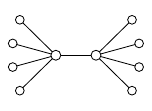}
    }
    \caption{Representative support graphs used in Section~\ref{sec:spec}:
    (a) constant-degree expander graphs, represented here by the 3-regular
    Petersen graph with $d=3$ (Lemma~\ref{l:expander});
    (b) the ring graph (Lemma~\ref{l:ring});
    (c) the undirected exponential graph (Lemma~\ref{l:exponential}); and
    (d) the balanced double-star graph 
    (Lemma~\ref{l:double_star}).}
    \label{fig:section2-graphs}
\end{figure}

We first use constant-degree expander graphs, as reviewed in
\cite{hoory2006expander}, to show that the diameter lower bound in Theorem
\ref{l:spectral} is tight up to constant factors.  For the spectral statements
below, we write
$\mu_1(\bM)\ge\mu_2(\bM)\ge\cdots\ge\mu_n(\bM)$
for the eigenvalues of a real symmetric matrix
$\bM\in\reals^{n\times n}$ in nonincreasing algebraic order.  Thus, for a
mixing matrix $\bW$, $\mu_1(\bW)=1$ and
$\lambda_{\bW}=\max\{|\mu_2(\bW)|,|\mu_n(\bW)|\}$; when $\bW$ is positive
semidefinite, this reduces to $\lambda_{\bW}=\mu_2(\bW)$.

\begin{lemma}[Constant-degree expander graphs]
    \label{l:expander}
    Let $\cG_{\mathrm{cde}}$ denote an $n$-node member of a family of connected
    $d$-regular expander graphs in the sense of Definition 2.2 in
    \cite{hoory2006expander}, where $d\ge3$ is fixed.  Let
    $\bA_{\mathrm{cde}}$ be the adjacency matrix of $\cG_{\mathrm{cde}}$ and
    define the symmetric doubly stochastic weight matrix
    $\bW_{\mathrm{cde}}:=\bA_{\mathrm{cde}}/d$.  Then there exists a constant
    $c>0$, independent of $n$, such that
    \[
        \mu_2(\bW_{\mathrm{cde}})\le1-c.
    \]
    Moreover, the diameter of its support graph satisfies
    \[
        D_{\cG_{\mathrm{cde}}}=\Theta(\log n).
    \]
\end{lemma}
\begin{proof}
    See Appendix~\ref{pf:l:expander}.
\end{proof}

\begin{remark}[Tightness on constant-degree expander graphs]
    \label{r:expander}
    For the expander family in Lemma \ref{l:expander}, consider the canonical
    lazy random-walk matrix
    \[
        \widehat{\bW}_{\mathrm{cde}}
        :=\frac12\prt{\bI+\bW_{\mathrm{cde}}}.
    \]
    This matrix is symmetric, positive semidefinite, and doubly stochastic,
    and its support graph is $\cG_{\mathrm{cde}}$; hence
    $D_{\cG_{\mathrm{cde}}}=\Theta(\log n)$.
    Its nontrivial eigenvalues are
    $\prt{1+\mu_i(\bW_{\mathrm{cde}})}/2$, $i=2,\ldots,n$, so
    \[
        1-\lambda_{\widehat{\bW}_{\mathrm{cde}}}
        =\frac{1-\mu_2(\bW_{\mathrm{cde}})}{2}
        \ge\frac{c}{2},
        \qquad
        \tau_{\widehat{\bW}_{\mathrm{cde}}}
        \le\frac{2}{c}=\cO(1).
    \]
    Since $\tau_{\widehat{\bW}_{\mathrm{cde}}}\ge1$, we have
    $\tau_{\widehat{\bW}_{\mathrm{cde}}}=\Theta(1)$.  On the other hand,
    Lemma \ref{l:expander} gives
    \[
        \frac{D_{\cG_{\mathrm{cde}}}-1}{\log n}=\Theta(1).
    \]
    Hence the lower bound in Theorem \ref{l:spectral} is tight up to constant
    factors on constant-degree expander families.
    This tightness statement concerns the regime $D_{\cG}=\Theta(\log n)$;
    it does not assert that the bound is sharp for every prescribed pair
    $(n,D_{\cG})$.
\end{remark}

The bound in Theorem \ref{l:spectral} depends only on the graph invariants $n$
and $D_{\cG}$ and holds
regardless of how the admissible edge weights are selected.  Its direction is
important: a large diameter forces the inverse gap to be large, up to the
$\log n$ factor, but a small diameter does not force the inverse gap to be
small.  Indeed, Theorem \ref{thm:matrix-dependent-gap} keeps $n$ and
$D_{\cG}$ fixed while making $\tau_{\bW}$ arbitrarily large.  Diameter is
therefore an intrinsic propagation barrier, whereas the inverse spectral gap
may contain an additional and potentially much larger penalty caused by the
chosen weights.

\subsection{Diameter and the Inverse Spectral Gap: Insights from Three Graph Families}

\label{sec:examples}

We now examine the ring, undirected exponential, and balanced
double-star graph families in Lemmas~\ref{l:ring}--\ref{l:double_star}.
For each family, the diameter is determined solely by the support
graph and can be computed without selecting communication weights, whereas
$\tau_{\bW}$ requires the spectral analysis of a particular weight matrix and
can change substantially under reweighting.
These examples show that even under standard weighting schemes,
no universal two-sided scaling law determines the inverse spectral
gap from the diameter alone. Thus,
although diameter does not determine the inverse spectral gap, its invariance
under reweighting makes it a more appropriate measure for isolating the effect
of graph topology on distributed optimization.

\begin{lemma}[Ring graph]
    \label{l:ring}
    For an $n$-node ring graph $\cG_{\mathrm{r}}$ with $n\ge 3$, let
    $\bW_{\mathrm{r}}$ be the lazy ring matrix defined by
    \[
    \brk{\bW_{\mathrm{r}}}_{j,j}=\frac{1}{2},
    \qquad
    \brk{\bW_{\mathrm{r}}}_{j,j-1}
    =
    \brk{\bW_{\mathrm{r}}}_{j,j+1}
    =
    \frac{1}{4},
    \]
    with indices taken modulo $n$. Then $\bW_{\mathrm{r}}$ is nonnegative,
    symmetric, doubly stochastic, and positive semidefinite,
    $\cG_{\bW_{\mathrm{r}}}=\cG_{\mathrm{r}}$, the graph diameter is
    $D_{\cG_{\mathrm{r}}}=\floor{n/2}$, and
    \[
    \lambda_{\bW_{\mathrm{r}}}=\cos^2\prt{\frac{\pi}{n}},
    \qquad
    \tau_{\bW_{\mathrm{r}}}
    =
    \frac{1}{\sin^2\prt{\pi/n}}
    =\Theta(n^2)
    =\Theta\prt{D_{\cG_{\mathrm{r}}}^2}.
    \]
\end{lemma}
\begin{proof}
    See Appendix~\ref{pf:l:ring}.
\end{proof}

\begin{lemma}[Undirected exponential graph]
    \label{l:exponential}
    Let $n=2^m$ with $m\ge 2$, let $\bfe_0,\ldots,\bfe_{n-1}$ denote the standard basis vectors, and define the cyclic-shift permutation matrix $\bP$ by
    \[
        \bP\bfe_j=\bfe_{(j+1)\bmod n},
        \qquad j=0,\ldots,n-1.
    \]
    Define the mixing matrix by
    \[
        \bW_{\mathrm{ex}}
        :=
        \frac{1}{2m}\sum_{r=0}^{m-1}
        \prt{\bP^{2^r}+\prt{\bP^\T}^{2^r}}.
    \]
    Then $\bW_{\mathrm{ex}}$ is symmetric and doubly stochastic. Let
    $\cG_{\mathrm{ex}}:=\cG_{\bW_{\mathrm{ex}}}$ denote its support graph,
    which has edges connecting nodes at cyclic distances
    $2^0,2^1,\ldots,2^{m-1}$, and
    \[
        D_{\cG_{\mathrm{ex}}}
        =
        \ceil{\frac{m}{2}}
        =
        \Theta(\log n).
    \]
    Moreover,
    \[
        1-\frac{2}{m}
        \le
        \lambda_{\bW_{\mathrm{ex}}}
        \le
        1-\frac{2-\sqrt{2}}{m}.
    \]
    Consequently,
    \[
        \tau_{\bW_{\mathrm{ex}}}
        =
        \Theta(m)
        =
        \Theta(\log n)
        =
        \Theta(D_{\cG_{\mathrm{ex}}}).
    \]
\end{lemma}
\begin{proof}
    See Appendix~\ref{pf:l:exponential}.
\end{proof}

\begin{lemma}[Two weightings of a balanced double-star graph]
    \label{l:double_star}
    Let $n=2m+2$ with $m\ge 1$. Consider the balanced double-star graph
    $\cG_{\mathrm{ds}}$ with two center nodes $c_1,c_2$, two disjoint leaf sets
    $\cL_1$ and $\cL_2$ of size $m$, edges between $c_1$ and every node in
    $\cL_1$, edges between $c_2$ and every node in $\cL_2$, and one edge
    between $c_1$ and $c_2$. Set
    \[
        a:=\frac{1}{m+2},
        \qquad
        b:=\frac{m+1}{m+2},
        \qquad
        \alpha:=\frac{1}{2m},
        \qquad
        \beta:=1-\frac{1}{2m}.
    \]
    After ordering the nodes as $(\cL_1,c_1,c_2,\cL_2)$, define the two
    weight matrices
    \[
    \begin{aligned}
        \bW_{\mathrm{ds}}
        &:=
        \begin{bmatrix}
            b\bI_m & a\mone_m & 0 & 0 \\
            a\mone_m^\T & a & a & 0 \\
            0 & a & a & a\mone_m^\T \\
            0 & 0 & a\mone_m & b\bI_m
        \end{bmatrix},\quad
        \widetilde{\bW}_{\mathrm{ds}}
        &:=
        \begin{bmatrix}
            \beta\bI_m & \alpha\mone_m & 0 & 0 \\
            \alpha\mone_m^\T & \frac14 & \frac14 & 0 \\
            0 & \frac14 & \frac14 & \alpha\mone_m^\T \\
            0 & 0 & \alpha\mone_m & \beta\bI_m
        \end{bmatrix}.
    \end{aligned}
    \]
    Both $\bW_{\mathrm{ds}}$ and $\widetilde{\bW}_{\mathrm{ds}}$ are mixing
    matrices on $\cG_{\mathrm{ds}}$, whose diameter is
    $D_{\cG_{\mathrm{ds}}}=3$. Their spectral quantities satisfy
    \[
    \begin{aligned}
        \lambda_{\bW_{\mathrm{ds}}}
        &=
        \frac{m+1+\sqrt{m^2+6m+1}}{2(m+2)},
        \\
        \tau_{\bW_{\mathrm{ds}}}
        &=
        \frac{(m+2)\prt{m+3+\sqrt{m^2+6m+1}}}{4}
        =
        \Theta(n^2),\\
        \lambda_{\widetilde{\bW}_{\mathrm{ds}}}
        &=
        \frac{1-\frac{1}{2m}+\sqrt{1+\frac{1}{4m^2}}}{2},
        \\
        \tau_{\widetilde{\bW}_{\mathrm{ds}}}
        &=
        2m+1+\sqrt{4m^2+1}
        =
        \Theta(n).
    \end{aligned}
    \]
\end{lemma}
\begin{proof}
    See Appendix~\ref{pf:l:double_star}.
\end{proof}

\begin{remark}[Topology-specific weight tuning]
    The matrix $\bW_{\mathrm{ds}}$ is generated by the Metropolis rule, a
    standard weight choice in decentralized optimization; see
    \cite{nedic2018network}.  Nevertheless, Lemma~\ref{l:double_star} shows
    that this generic rule need not be spectrally well tuned to a given
    topology: replacing $\bW_{\mathrm{ds}}$ with
    $\widetilde{\bW}_{\mathrm{ds}}$ reduces the inverse-gap factor from
    $\Theta(n^2)$ to $\Theta(n)$.  Thus topology-specific weight tuning can
    improve the network-dependent term of algorithms whose convergence rates
    depend on $\tau_{\bW}$, without changing the communication graph.
\end{remark}

\begin{remark}[No universal scaling relation between diameter and the inverse spectral gap]
    As shown in Lemmas~\ref{l:ring}--\ref{l:double_star}, the following
    graph--matrix pairs exhibit fundamentally different asymptotic relations
    between diameter and inverse spectral gap:
    \begin{center}
        \small
        \begin{tabular}{@{}lcc@{}}
            \toprule
            Graph--matrix pair $(\cG,\bW)$ & Diameter $D_{\cG}$ & Inverse spectral gap $\tau_{\bW}$ \\
            \midrule
            Ring $(\cG_{\mathrm{r}},\bW_{\mathrm{r}})$
            & $\Theta(n)$ & $\Theta(n^2)$ \\
            Undirected exponential $(\cG_{\mathrm{ex}},\bW_{\mathrm{ex}})$
            & $\Theta(\log n)$ & $\Theta(\log n)$ \\
            Balanced double-star $(\cG_{\mathrm{ds}},\bW_{\mathrm{ds}})$
            & $3$ & $\Theta(n^2)$ \\
            \bottomrule
        \end{tabular}
    \end{center}
    Thus no universal scaling relation between $D_{\cG}$ and $\tau_{\bW}$
    holds across these representative pairs: the inverse gap is quadratic in
    the diameter on the ring, linear in the diameter on the undirected
    exponential graph, and grows quadratically with $n$ even though the
    double-star diameter remains equal to $3$. 
\end{remark}

The preceding examples exhibit several incompatible scalings between diameter
and inverse spectral gap.  The following theorem strengthens this observation:
even among graph--matrix pairs with arbitrarily large diameter, every
prescribed positive power-law relation between the two quantities can be
realized exactly.

\begin{theorem}[Arbitrary power-law relations between diameter and the inverse spectral gap]
    \label{thm:arbitrary-diameter-gap-scaling}
    For any fixed $p>0$ and any $D_0\ge1$, there exists a graph--matrix pair
    $(\cG^*,\bW^*)$ such that
    \[
        D_{\cG^*}\ge D_0,
        \qquad
        \tau_{\bW^*}=D_{\cG^*}^p.
    \]
\end{theorem}
\begin{proof}
    See Appendix \ref{pf:thm:arbitrary-diameter-gap-scaling}.
\end{proof}

Taken together, the results in this section separate the effect of network
topology from that of a selected weighted mixing protocol.  Theorem
\ref{thm:matrix-dependent-gap} shows that the inverse spectral gap is not a
graph invariant and can change by an arbitrary factor on a fixed support
graph.  Theorem \ref{l:spectral} identifies diameter as a graph-invariant,
one-sided constraint imposed by finite-speed information propagation.
The constant-degree expander family shows that this constraint
is tight up to constant factors, while the ring, exponential, and
double-star examples show that no universal two-sided scaling law
in terms of diameter alone holds, even under standard weighting schemes.
Theorem
\ref{thm:arbitrary-diameter-gap-scaling} strengthens these examples by showing
that every positive power-law relation between the two quantities can be
realized exactly by graph--matrix pairs with arbitrarily large diameter.

Consequently, a convergence rate expressed through $\tau_{\bW}$ characterizes
the repeated use of the selected matrix $\bW$, rather than the intrinsic
influence of network topology.  When the goal is to isolate this topological
influence, diameter is a more appropriate graph-level measure: it is invariant
under reweighting and equals the worst-case minimum number of one-hop rounds
needed for information to traverse the network.  This distinction motivates
the diameter-based tree communication mechanism developed in this paper.

\section{Network Dependence and Uniform Optimality}
\label{sec:lower-bounds}

This section specifies the sense in which a distributed method can be
optimal with respect to the communication network.  The distinction is
primarily one of quantifiers.  A lower complexity bound over a class of
graphs may be certified by a single hard graph in that class, whereas an
algorithmic upper bound is uniform only if it holds for every graph in the
class.  Matching these two statements identifies the smallest network
dependence that any admissible method can guarantee in the worst case.

This formulation leaves open which network characteristics should be used to
specify the graph class. 
We first formulate this notion for graph classes described by the number of
nodes and the graph diameter.  We then discuss what changes when a spectral
constraint is imposed on a mixing matrix.  
As shown in
Section~\ref{sec:spec}, diameter is a graph invariant, while the spectral gap
also depends on the selected weighted communication protocol.  Consequently,
diameter-based and spectral-gap-based optimality refer, in general, to
different classes of objects and different algorithmic restrictions.

\subsection{Diameter-Based Optimality over Graph Classes}
\label{sec:lower-bounds-diameter}

Let $\mathsf{G}$ be a class of connected undirected communication graphs,
each with node set $\cN$ and $|\cN|=n$.  The local objectives and stochastic
first-order oracles satisfy Assumptions~\ref{a.smooth} and \ref{a.var},
respectively.  All agents start from a common point $x^{(0)}$, and
\[
    \Delta_f:=f(x^{(0)})-\inf_x f(x)\le\Delta.
\]
The framework in \cite{lu2021optimal} initializes all agents at the origin;
the formulation above is equivalent after translating the common initial
point to zero.  Let $\mathsf{I}_n(\Delta,L,\sigma^2)$ denote the class of
pairs $(\{f_i\}_{i=1}^n,O)$ satisfying Assumptions~\ref{a.smooth} and
\ref{a.var} together with $\Delta_f\le\Delta$.

We adopt the broad synchronous first-order algorithm class $\mathsf{A}_B$
introduced in \cite{lu2021optimal}; its zero-respecting condition follows
\cite{carmon2021lower}.  In every round, each agent makes at most $B$ queries
to its local stochastic oracle and communicates only with its one-hop
neighbors.  A method is zero-respecting if an oracle query can be nonzero only
in coordinates previously present in the agent's own or a neighboring state,
while its next local state may additionally activate coordinates revealed by
its current local oracle responses.  Thus, an agent cannot create an unseen
coordinate before it is revealed locally or received through communication.
As discussed in \cite{lu2021optimal}, this class covers standard first-order
methods including SGD \cite{ghadimi2013stochastic}, momentum SGD
\cite{nesterov1983method}, Adam \cite{kingma2014adam}, RMSProp
\cite{tieleman2012lecture}, AdaGrad \cite{ward2018adagrad}, and AdaDelta
\cite{zeiler2012adadelta}.
The final output may be any linear combination of the local model states.  In
particular, $\mathsf{A}_B$ is not restricted to gossip algorithms that
repeatedly apply a prescribed mixing matrix; subject to the same one-hop
communication constraint, it permits algorithm-specific weighting and
auxiliary-state recursions, thereby covering, for example, DSMT \cite{huang2026accelerated}, DeTAG \cite{lu2021optimal}, and our Tree-RGT (see Algorithm \ref{alg:ostl}).

For a method $A\in\mathsf{A}_B$, an instance
$(\{f_i\},O)\in\mathsf{I}_n(\Delta,L,\sigma^2)$, and a graph
$\cG\in\mathsf{G}$, let $\hat{x}_A^{(t)}$ be the output after $t$
synchronous one-hop rounds.  Its $\varepsilon$-round complexity is
\begin{equation}
    \label{eq:epsilon-complexity}
    T_\varepsilon(A,f,O,\cG)
    :=
    \min\crk{
        t\in\nats:
        \expect\norm{\nabla f(\hat{x}_{A,f,O,\cG}^{(t)})}\le\varepsilon
    }.
\end{equation}
Here, one unit of communication complexity corresponds to one synchronous
neighbor-to-neighbor round, irrespective of how many edges are active or how
the communication load is distributed across the nodes. Thus, the model
captures communication latency, but not aggregate bandwidth consumption or
per-node congestion.

The optimal worst-case round complexity over $\mathsf{G}$ is
\begin{equation}
    \label{eq:minimax-round-complexity-general}
    \mathsf{T}_\varepsilon(\mathsf{G},B)
    :=
    \inf_{A\in\mathsf{A}_B}
    \sup_{\substack{
        \cG\in\mathsf{G},\\
        (\{f_i\},O)\in\mathsf{I}_n(\Delta,L,\sigma^2)
    }}
    T_\varepsilon(A,f,O,\cG).
\end{equation}
The joint supremum requires one algorithmic guarantee to hold over every graph
and every admissible optimization instance in the prescribed classes.  This
is the full minimax criterion.  We also use the term
\emph{uniform over the graph class} in a more focused sense: for each fixed
admissible optimization instance, the same bound, with the same numerical
constants and the same dependence on the class parameters, holds for every
graph in the class.  Instance-dependent quantities must then be displayed
explicitly and cannot be hidden in graph-dependent constants.

We now specialize the graph class to
\begin{equation}
    \label{eq:diameter-graph-class}
    \mathsf{G}_{n,\hat D}
    :=
    \crk{
        \cG=(\cN,\cE):
        \cG\text{ is connected and undirected},\ |\cN|=n,\ D_{\cG}=\hat D
    }.
\end{equation}
Both $n$ and $D_{\cG}$ are graph invariants, so membership in
$\mathsf{G}_{n,\hat D}$ is independent of any communication weights subsequently
chosen on the graph.

The following benchmark combines the results of
\cite[Theorems~1 and~2]{lu2021optimal}.

\begin{theorem}[Diameter-based minimax benchmark]
    \label{thm:lu-de-sa-network-benchmark}
    Fix $\Delta,L,\sigma>0$, integers $n\ge2$ and
    $\hat D\in\{1,\ldots,n-1\}$, and an oracle budget $B\ge1$.  Then, up to
    universal numerical constants,
    \begin{equation}
        \label{eq:diameter-minimax-optimal}
        \mathsf{T}_\varepsilon(\mathsf{G}_{n,\hat D},B)
        =
        \Theta\prt{
            \frac{\Delta L\sigma^2}{nB\varepsilon^4}
            +
            \frac{\Delta L\hat D}{\varepsilon^2}
        }.
    \end{equation}
\end{theorem}

For the lower bound, \cite[Theorem~1]{lu2021optimal} constructs, for every
admissible $n,\hat D$ and target accuracy, a graph
$\cG_{\mathrm{hard}}\in\mathsf{G}_{n,\hat D}$ and a hard
optimization-oracle instance such that every $A\in\mathsf{A}_B$ requires
\[
    \Omega\prt{
        \frac{\Delta L\sigma^2}{nB\varepsilon^4}
        +
        \frac{\Delta L\hat D}{\varepsilon^2}
    }
\]
rounds.  The first term follows from the total budget of at most $nB$
oracle queries per round.  The second term is obtained by distributing
a zero-chain construction across groups separated by a path-like bottleneck
of length proportional to $\hat D$.  Information cannot reveal the next relevant
coordinates before traversing this bottleneck.

The graph quantifier is essential: the lower bound identifies one worst-case
graph in $\mathsf{G}_{n,\hat D}$ and does not imply that every graph of diameter
$\hat D$ is equally hard. The matching uniform upper bound in
\eqref{eq:diameter-minimax-optimal} is attained by DeFacto
\cite[Theorem~2]{lu2021optimal}. For a graph of diameter $\hat D$, DeFacto
precomputes an integer $R\in\{\hat D,\ldots,2\hat D\}$ and a graph-dependent
sequence of support-compatible communication matrices whose product is the
exact averaging matrix
\cite[Lemma~1 and Algorithm~1]{lu2021optimal}. Each $2R$-round block comprises
$R$ computation-only rounds followed by $R$ communication-only rounds and
produces one model update; thus, $T$ rounds yield only
$\lfloor T/(2R)\rfloor$ completed updates. Accordingly, DeFacto relies on both
phase separation and a graph-specific factorization of the exact averaging
matrix.

To formalize the network dependence captured by this minimax benchmark,
we say that a method
$A\in\mathsf{A}_B$ has \emph{uniformly optimal network dependence} over
$\mathsf{G}_{n,\hat D}$ if it admits an upper bound on its
$\varepsilon$-round complexity that is uniform over this graph class and
matches the minimax powers of $n$ and $\hat D$ in
\eqref{eq:diameter-minimax-optimal}.  With $\Delta$, $L$, and $\sigma^2$ fixed,
the two benchmark terms scale as $n^{-1}B^{-1}\varepsilon^{-4}$ and
$\hat D\varepsilon^{-2}$, respectively.  This terminology concerns only
network-parameter dependence and does not assert optimality of every problem-
or instance-dependent coefficient in the underlying convergence bound.
Section~\ref{sec:tree-rgt} develops Tree-RGT and establishes its uniformly
optimal network dependence over $\mathsf{G}_{n,\hat D}$, without phase
separation or a graph-specific factorization of the exact averaging matrix.

\subsection{Optimality under Mixing-Matrix Constraints}

The notion of uniformly optimal network dependence above is topology based:
both the graph class
$\mathsf{G}_{n,\hat D}$ and the admissible algorithm class $\mathsf{A}_B$ are
specified without a mixing matrix.  A claim that the dependence on a spectral
gap is optimal poses a different and, without additional modeling choices,
not yet fully specified problem.

When network dependence is expressed through quantities derived from a mixing
matrix, the relevant network class consists of graph--matrix pairs rather than
graphs alone.  More specifically, combining the diameter lower bound in
\cite[Theorem~1]{lu2021optimal} with the weighted-path construction underlying
\cite[Corollary~1]{lu2021optimal} yields the following coupled specialization.

\begin{theorem}[Weighted-path inverse-gap lower bound]
    \label{thm:weighted-path-inverse-gap}
    Fix $\Delta,L,\sigma>0$, an integer $n\ge2$, and an oracle budget $B\ge1$.
    There exists a parameter $\tau>0$ satisfying
    $\tau=\Theta\prt{(n-1)^2}$.  For this choice, define
    \begin{equation}
        \label{eq:lu-path-pair-class}
        \widehat{\mathsf{P}}_n
        :=\crk{
            (\cG,\bW):
            \cG\in\mathsf{G}_{n,n-1},\
            \bW\text{ is a mixing matrix on }\cG,\
            \tau_{\bW}\le\tau
        }.
    \end{equation}
    Then
    \begin{equation}
        \label{eq:lu-path-class-lower-bound}
        \inf_{A\in\mathsf{A}_B}
        \sup_{\substack{
            (\cG,\bW)\in\widehat{\mathsf{P}}_n,\\
            (\{f_i\},O)\in\mathsf{I}_n(\Delta,L,\sigma^2)
        }}
        T_\varepsilon(A,f,O,\cG)
        =
        \Omega\prt{
            \frac{\Delta L\sigma^2}{nB\varepsilon^4}
            +
            \frac{\Delta L\sqrt{\tau}}{\varepsilon^2}
        }.
    \end{equation}
\end{theorem}
\begin{proof}
    See Appendix~\ref{pf:thm:weighted-path-inverse-gap}.
\end{proof}

The proof of Theorem~\ref{thm:weighted-path-inverse-gap} identifies a pair
$(\widehat{\cG}_{\mathrm{hard}},\bW_{\mathrm{hard}})
\in\widehat{\mathsf{P}}_n$ and an optimization--oracle instance for which every
$A\in\mathsf{A}_B$ requires
\[
    \Omega\prt{
        \frac{\Delta L\sigma^2}{nB\varepsilon^4}
        +
        \frac{\Delta L}
        {\sqrt{1-\lambda_{\bW_{\mathrm{hard}}}}\,\varepsilon^2}
    }
\]
rounds.  In the minimax sense, this bound applies to methods required to work
uniformly over $\widehat{\mathsf{P}}_n$, although it is witnessed by a single
worst-case pair.  Unlike Theorem~\ref{thm:lu-de-sa-network-benchmark}, it covers
only the coupled weighted-path regime $\hat D=n-1$ and
$\tau_{\bW_{\mathrm{hard}}}=\tau=\Theta(\hat D^2)$, where
$(1-\lambda_{\bW_{\mathrm{hard}}})^{-1/2}=\Theta(\hat D)$.  Thus, the
inverse-gap term reparameterizes the path-diameter barrier rather than imposing
an independently prescribed constraint.  The next corollary shows how its
exponent changes with the graph--matrix pair class.

\begin{corollary}[Reweighted-path inverse-gap lower bound]
    \label{cor:reweighted-path-inverse-gap}
    Fix $\Delta,L,\sigma>0$, an integer $n\ge2$, and an oracle budget $B\ge1$.
    With $\tau$ chosen as in Theorem~\ref{thm:weighted-path-inverse-gap}, let
    $\bar\tau:=\Theta\prt{(n-1)^3}$ and define
    \begin{equation}
        \label{eq:lu-reweighted-path-pair-class}
        \overline{\mathsf{P}}_n
        :=\crk{
            (\cG,\bW):
            \cG\in\mathsf{G}_{n,n-1},\
            \bW\text{ is a mixing matrix on }\cG,\
            \tau_{\bW}\le\bar\tau
        }.
    \end{equation}
    Then
    \begin{equation}
        \label{eq:lu-reweighted-path-class-lower-bound}
        \inf_{A\in\mathsf{A}_B}
        \sup_{\substack{
            (\cG,\bW)\in\overline{\mathsf{P}}_n,\\
            (\{f_i\},O)\in\mathsf{I}_n(\Delta,L,\sigma^2)
        }}
        T_\varepsilon(A,f,O,\cG)
        =
        \Omega\prt{
            \frac{\Delta L\sigma^2}{nB\varepsilon^4}
            +
            \frac{\Delta L\bar\tau^{1/3}}{\varepsilon^2}
        }.
    \end{equation}
\end{corollary}
\begin{proof}
    See Appendix~\ref{pf:cor:reweighted-path-inverse-gap}.
\end{proof}

Corollary~\ref{cor:reweighted-path-inverse-gap} shows that there exists a pair
$(\widehat{\cG}_{\mathrm{hard}},\bW_{\mathrm{harder}})
\in\overline{\mathsf{P}}_n$ and an optimization--oracle instance for which
every $A\in\mathsf{A}_B$ requires
\[
    \Omega\prt{
        \frac{\Delta L\sigma^2}{nB\varepsilon^4}
        +
        \frac{\Delta L}
        {(1-\lambda_{\bW_{\mathrm{harder}}})^{1/3}\varepsilon^2}
    }
\]
rounds.  Here $D_{\widehat{\cG}_{\mathrm{hard}}}=n-1$ and
$\tau_{\bW_{\mathrm{harder}}}=\bar\tau=\Theta\prt{(n-1)^3}$.  Together with
Theorem~\ref{thm:weighted-path-inverse-gap}, this shows that the inverse-gap
dependence of the minimax complexity is pair-class dependent: the exponent is
$1/2$ over $\widehat{\mathsf{P}}_n$ but $1/3$ over
$\overline{\mathsf{P}}_n$.  Thus, any claim of \emph{uniformly optimal
inverse-gap dependence} must first specify the graph--matrix pair class over
which uniformity is required.  This does not yet complete the minimax model:
one must additionally specify how admissible algorithms are allowed to use the
designated matrix.

\textbf{The graph--matrix pair class.}
A general class is easy to formulate.  For example, one may define
\begin{equation}
    \label{eq:graph-protocol-class}
    \mathsf{P}_{n,\hat D,\hat{\tau}}
    :=
    \crk{
        (\cG,\bW):
        \cG\in\mathsf{G}_{n,\hat D},\
        \bW\text{ is a mixing matrix on }\cG,\
        \tau_{\bW}\le\hat{\tau}
    }.
\end{equation}
The earlier path classes are the coupled specializations
$\widehat{\mathsf{P}}_n=\mathsf{P}_{n,n-1,\tau}$ and
$\overline{\mathsf{P}}_n=\mathsf{P}_{n,n-1,\bar\tau}$, but choosing informative
constraints for general $(n,\hat D,\hat{\tau})$ is more delicate.
Theorem~\ref{thm:matrix-dependent-gap} makes the upper-scale issue explicit:
for a positive-semidefinite mixing matrix $\bW$, the lazy reweighting
$\bW^{(\delta)}=(1-\delta)\bI+\delta\bW$ preserves the support graph but
satisfies $\tau_{\bW^{(\delta)}}=\tau_{\bW}/\delta$.  Thus, topology imposes no
finite upper scale on $\tau_{\bW}$.  If $\hat{\tau}$ is allowed to grow too
freely, the class admits increasingly slow weightings of the same graph, and a
worst-case dependence on $\hat{\tau}$ may reflect this artificial slowing
rather than greater topological difficulty.  At the other extreme,
Theorem~\ref{l:spectral} gives
$\tau_{\bW}\ge(D_{\cG}-1)/\log n$, so
\eqref{eq:graph-protocol-class} is empty whenever
$\hat{\tau}<(\hat D-1)/\log n$.  This necessary condition alone does not
establish nonemptiness, and the examples in Section~\ref{sec:examples} show that
admissible diameter--inverse-gap relations can vary substantially across graph
families and weight choices.  Selecting a nonempty and informative parameter
regime therefore requires additional modeling choices.

\textbf{The algorithm class.}
The broad class $\mathsf{A}_B$ does not make the designated $\bW$ part of its
communication primitive: a method may instead use any support-compatible
weights or routing rule.  Moreover, $T_\varepsilon(A,f,O,\cG)$ has no matrix
argument.  Hence, once the admissible support graphs are fixed, the numerical
spectral gap of the designated matrix can disappear from the minimax problem.
One possible fixed-mixing convention is to let
$\mathsf{A}^{\mathrm{mix}}_{B,\bW}\subseteq\mathsf{A}_B$ contain methods whose
cross-agent aggregations are restricted to applications of
$\by\mapsto\bW\by$, while retaining unrestricted local computation, memory,
auxiliary states, and at most $B$ local oracle queries per round.  Since $\bW$
varies across graph--matrix pairs, a uniform guarantee should not select an
unrelated algorithm separately for each matrix.  We therefore introduce the
single-template class
\[
    \mathfrak{A}_{B}^{\mathrm{mix}}
    :=
    \left\{
        \mathcal{A}:
        \begin{array}{l}
        \mathcal{A}\text{ is one uniformly specified}\\
        \text{algorithmic template, and}\\
        \mathcal{A}[\bW]\in\mathsf{A}^{\mathrm{mix}}_{B,\bW}
        \text{ for every admissible}\\
        \text{graph--matrix pair }(\cG,\bW)
        \end{array}
    \right\},
\]
where $\mathcal{A}$ is chosen before the pair and $\mathcal{A}[\bW]$ is its
instance under the designated matrix.  This interface includes fixed-matrix
auxiliary-state methods such as DeTAG \cite{lu2021optimal} and DSMT
\cite{huang2026accelerated}.  It does not include DeFacto \cite{lu2021optimal}, whose exact-consensus
phase uses a graph-dependent sequence of support-compatible matrices;
including such methods requires a richer protocol-valued model, whereas
allowing arbitrary support-compatible matrices would again make $\bW$
bypassable.

With both modeling choices fixed, the corresponding minimax complexity is
\begin{equation}
    \label{eq:fixed-mixing-minimax}
    \mathsf{T}^{\mathrm{mix}}_\varepsilon(\mathsf{P}_{n,\hat D,\hat{\tau}},B)
    :=
    \inf_{\mathcal{A}\in\mathfrak{A}_{B}^{\mathrm{mix}}}
    \sup_{\substack{
        (\cG,\bW)\in\mathsf{P}_{n,\hat D,\hat{\tau}},\\
        (\{f_i\},O)\in\mathsf{I}_n(\Delta,L,\sigma^2)
    }}
    T_\varepsilon(\mathcal{A}[\bW],f,O,\cG).
\end{equation}
The diameter benchmark in Theorem~\ref{thm:lu-de-sa-network-benchmark} and the
two coupled path specializations above do not provide a characterization of
\eqref{eq:fixed-mixing-minimax} for
$\mathsf{P}=\mathsf{P}_{n,\hat D,\hat{\tau}}$ with general, independently
prescribed $(n,\hat D,\hat{\tau})$.  The present paper does not pursue this
distinct spectral minimax problem.  Section~\ref{sec:main-convergence-guarantee}
instead establishes that Tree-RGT has \emph{uniformly optimal network
dependence} over $\mathsf{G}_{n,\hat D}$.

\section{Tree-Routed Gradient Tracking}
\label{sec:tree-rgt}

Motivated by the diameter-based benchmark in
Section~\ref{sec:lower-bounds-diameter}, we develop Tree-Routed Gradient
Tracking (Tree-RGT), a distributed method over the given
communication graph. Its communication mechanism is governed by finite graph
propagation rather than by the asymptotic contraction of a prescribed mixing
matrix.

Given the physical communication graph, we select a rooted shortest-path
spanning tree. Its outward orientation disseminates decision-variable iterates
from the root to the leaves, while its inward orientation aggregates
gradient-tracking information toward the root. If $D$ denotes the rooted tree
depth, then $D\le D_{\cG}$, and the corresponding propagation matrices reach
their rank-one forms after $D$ one-hop steps.

Tree-RGT pipelines these two routing directions. In every synchronous
iteration, each agent performs one round of one-hop tree communication, draws
one stochastic-gradient sample, and completes one local update; dissemination
and aggregation advance concurrently by one tree level. Thus, each
optimization update is accompanied by a single neighbor exchange, with no
additional rounds of neighbor averaging between successive updates.

We first construct the rooted shortest-path tree, then present the local
Tree-RGT updates and their matrix representation, establish the exact
finite-step routing property, and finally show that Tree-RGT has
\emph{uniformly optimal network dependence} over $\mathsf{G}_{n,\hat D}$.

\subsection{Rooted Shortest-Path Communication Trees}

Let $\cG=(\cN,\cE)$ be a connected undirected communication graph, and suppose
that a root node has been prescribed. Without loss of generality, we relabel the
nodes so that the root has index~$1$. This relabeling changes neither the
physical graph nor any graph distance.

Algorithm~\ref{alg:gt} constructs the local routing sets by a
level-synchronous breadth-first search. When a non-root node $i$ is first
discovered, it selects the smallest-index sender from the preceding level as
its parent $p(i)$. Since all candidate parents lie at the same level, this
tie-breaking rule affects only the selected parent, not the assigned depth or
the shortest-path property. For notational uniformity, set $p(1):=1$ and define
\[
    \mathrm{P}(i):=\{p(i)\},
    \qquad
    \mathrm{Ch}(i):=\{j\in\cN:\ p(j)=i\}.
\]

\begin{algorithm}[htbp]
   \caption{Distributed construction of local routing sets for a rooted
   shortest-path spanning tree}
   \label{alg:gt}
    \begin{algorithmic}
    \State {\bfseries Local input:} Each node $i$ knows the network size
    $n\ge2$, its own index, its physical neighbors in $\cG$, and the prescribed
    root index~$1$.
    \State {\bfseries Initialization:} Node $1$ sets $p(1)=1$, $h_1=0$, and
    $\mathrm{P}(1)=\mathrm{Ch}(1)=\{1\}$ and marks itself as discovered;
    every node $i\ne1$ sets $h_i=\infty$ and
    $\mathrm{P}(i)=\mathrm{Ch}(i)=\emptyset$ and marks itself as undiscovered.
    \For{$\ell=0,1,\ldots,n-2$}
        \State Every node $i$ with $h_i=\ell$ sends a discovery message to all
        of its physical neighbors.
        \State Every undiscovered node $j$ that receives at least one message
        selects the sender with the smallest node index as $p(j)$, sets
        $\mathrm{P}(j)=\{p(j)\}$ and $h_j=\ell+1$, and marks itself as
        discovered.
        \State Every newly discovered node $j$ sends an acknowledgment to
        $p(j)$; upon receipt, node $p(j)$ adds $j$ to
        $\mathrm{Ch}(p(j))$.
    \EndFor
    \State {\bfseries Distributed output:} Each node $i$ stores its own local
    routing sets $\mathrm{P}(i)$ and $\mathrm{Ch}(i)$.
\end{algorithmic}
\end{algorithm}

Figure~\ref{fig:algorithm1-spanning-tree} illustrates how the local routing
sets select an active routing tree without modifying the physical topology.
\begin{figure}[ht]
    \centering
    \subfloat[]{
        \includegraphics[width=0.29\textwidth]{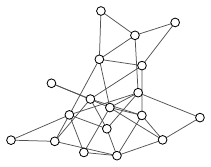}
    }
    \hspace{0.08\textwidth}
    \subfloat[]{
        \includegraphics[width=0.29\textwidth]{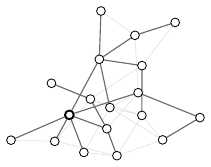}
    }
    \caption{Illustration of Algorithm~\ref{alg:gt}.
    (a) A connected realization of the Erd\H{o}s--R\'enyi random graph model
    \cite{erdos1959random}, with $n=18$ and edge probability $0.22$.
    (b) The rooted shortest-path spanning tree $\cT$ selected by
    Algorithm~\ref{alg:gt}. The thick-outlined node is the prescribed root,
    relabeled as node~$1$. Emphasized edges are the active parent--child routing
    links, while faded edges remain physical links of $\cG$ but are not used
    during the Tree-RGT optimization iterations. Node labels are used for
    deterministic parent tie-breaking and omitted in the drawing.}
    \label{fig:algorithm1-spanning-tree}
\end{figure}

The fixed horizon $\ell=0,\ldots,n-2$ suffices to discover every node because a
connected $n$-node graph has no simple path longer than $n-1$.

For analysis, collect the locally stored parent--child relations into the
bidirected arc set
\[
    \cE_{\cT}
    :=
    \bigcup_{i\in\cN\setminus\{1\}}
    \bigl\{(p(i),i),(i,p(i))\bigr\},
    \qquad
    \cT:=(\cN,\cE_{\cT}).
\]
Every non-root node has one parent in the preceding discovery level, and
following parent links reaches the root. Thus, the underlying undirected graph
of $\cT$ is a spanning tree of $\cG$, represented above by both orientations of
each tree link. All distances associated with $\cT$ refer to this underlying
undirected tree. The notation is analytical: each node stores only its local
sets $\mathrm{P}(i)$ and $\mathrm{Ch}(i)$.

\textbf{Communication over the selected tree.}
After preprocessing, node $i$ communicates only with the physical neighbors in
$\prt{\mathrm{P}(i)\cup\mathrm{Ch}(i)}\setminus\{i\}$; messages from its other
physical neighbors may equivalently be ignored. The self-entry at the root is
implemented by local state retention and requires no physical communication.
Thus, selecting $\cT$ restricts the active links without modifying the physical
topology.

Algorithm~\ref{alg:gt} assigns each node $i$ a discovery level $h_i$. Since
every non-root node at level $h_i$ selects a parent at level $h_i-1$, following
parent links gives $h_i=\operatorname{dist}_{\cT}(1,i)$. Thus, $h_i$ is the
number of one-hop communication steps from the root to node $i$ and represents
its node-specific propagation delay. Define the maximum and average depths by
\begin{equation}
    \label{eq:tree-depth-profile}
    D:=\max_{i\in\cN}h_i,
    \qquad
    D_{\avg}:=\frac{1}{n}\sum_{i\in\cN}h_i.
\end{equation}
We call $D$ the \emph{tree depth} of $\cT$, also known as its height. Unlike
the pairwise tree diameter
$\max_{i,j\in\cN}\operatorname{dist}_{\cT}(i,j)$, $D$ measures the largest
distance from the root, while $D_{\avg}$ is the average distance from the root.

\begin{theorem}[Optimal rooted depth profile]
    \label{thm:bfs-tree-optimal}
    Let $\cG=(\cN,\cE)$ be a connected undirected communication graph with
    $n\ge2$ nodes and prescribed root node~$1$. Let $\cT$ be the rooted tree
    induced by the routing sets produced by Algorithm~\ref{alg:gt} on $\cG$.
    Then
    \begin{equation}
        \label{eq:bfs-pointwise-depth}
        h_i=\operatorname{dist}_{\cG}(1,i)
        \qquad\text{for every }i\in\cN.
    \end{equation}
    Consequently, among all spanning trees $\cT'$ of $\cG$ rooted at node~$1$,
    $\cT$ simultaneously minimizes the tree depth and the average distance to
    the root:
    \begin{align}
        D
        &=
        \min_{\cT'}\max_{i\in\cN}
        \operatorname{dist}_{\cT'}(1,i)
        =
        \max_{i\in\cN}\operatorname{dist}_{\cG}(1,i)
        \le D_{\cG}
        \le 2D, \label{eq:bfs-optimal-D}\\
        D_{\avg}
        &=
        \min_{\cT'}\frac{1}{n}\sum_{i\in\cN}
        \operatorname{dist}_{\cT'}(1,i)
        =
        \frac{1}{n}\sum_{i\in\cN}\operatorname{dist}_{\cG}(1,i).
        \label{eq:bfs-optimal-Davg}
    \end{align}
    In particular, $D$ and $D_{\cG}$ differ by at most a factor of two and
    therefore have the same asymptotic order over any family of graphs.
\end{theorem}
\begin{proof}
    See Appendix \ref{pf:thm:bfs-tree-optimal}.
\end{proof}

For the remainder of the paper, $\cT$, $h_i$, $D$, and $D_{\avg}$ refer to
the tree and depth quantities generated from $\cG$ by
Algorithm~\ref{alg:gt}.
The optimality above is with respect to a prescribed root. If the root were
also a design variable, minimizing $D$ and minimizing $D_{\avg}$ could select
different root nodes.

\subsection{Tree-RGT: Pipelined Gradient Tracking over the Tree}

Algorithm~\ref{alg:gt} constructs the bidirected graph $\cT$ associated with
a rooted shortest-path spanning tree of $\cG$. We separate the two orientations
of every tree link and define
\[
    \begin{aligned}
        \cE_{\bR}
        &:=
        \bigl\{(p(i),i):i\in\cN\setminus\{1\}\bigr\},
        &\quad
        \cT_{\bR}&:=(\cN,\cE_{\bR}),\\
        \cE_{\bC}
        &:=
        \bigl\{(i,p(i)):i\in\cN\setminus\{1\}\bigr\},
        &
        \cT_{\bC}&:=(\cN,\cE_{\bC}).
    \end{aligned}
\]
The two arc sets partition $\cE_{\cT}$. The out-arborescence $\cT_{\bR}$
disseminates root iterates to the agents, while the reverse in-arborescence
$\cT_{\bC}$ routes gradient-tracking information back to the root. Together,
they form a root-to-agent-to-root pipeline: a root iterate reaches agent $i$,
where a stochastic gradient is evaluated, and the resulting contribution then
returns to the root. After the initial warm-up, the root aggregate contains one
depth-delayed contribution from every local objective. The corresponding
propagation matrices $\bR$ and $\bC$ are constructed in
Section~\ref{sec:ost-matrix}.

Rather than waiting for a complete broadcast and aggregation between
optimization updates, Tree-RGT advances both streams concurrently by one tree
edge per synchronous iteration. Each iteration uses one communication round and
one fresh stochastic-gradient sample per agent. Only the root performs a descent update and maintains a momentum
vector; every non-root agent copies its parent's preceding iterate and updates
its gradient tracker. Momentum therefore adds only one $p$-dimensional state
at the root and requires no additional communicated variable.
Algorithm~\ref{alg:ostl} formalizes this single-loop construction.

\begin{algorithm}[htbp]
   \caption{Tree-Routed Gradient Tracking with Momentum (\textbf{Tree-RGT})}
   \label{alg:ostl}
\begin{algorithmic}
    \State {\bfseries Input:} Connected physical communication graph
    $\cG=(\cN,\cE)$, root node $1$, common initial point
    $x^{(0)}\in\reals^p$, stepsize $\gamma>0$,
    momentum coefficient $\beta\in[0,1)$, and iteration budget $T$.
    \State {\bfseries Preprocessing:} Before training, the agents jointly
    execute Algorithm~\ref{alg:gt} over $\cG$; each agent $i$ obtains and
    stores its local routing sets $\mathrm{P}(i)$ and $\mathrm{Ch}(i)$.
    \State {\bfseries Initialization:} Each agent $i$ sets
    $x_i^{(0)}=x^{(0)}$, independently draws $\xi_i^{(0)}$, and initializes
    $y_i^{(0)}=g_i(x_i^{(0)},\xi_i^{(0)})$.
    The root additionally sets $m_1^{(0)}=y_1^{(0)}$.
    \For{$t=0,1,\ldots,T-1$}
    \For{each agent $i\in\cN$ in parallel}
    \State In one communication round, agent $i$ simultaneously obtains
    $x_j^{(t)}$ from every $j\in\mathrm{P}(i)$ and $y_j^{(t)}$ from every
    $j\in\mathrm{Ch}(i)$; all self-entries are handled locally.
    \State Independently draw a random sample $\xi_i^{(t+1)}$.
    \State Update in the following order
        \[
        \begin{aligned}
            x_i^{(t+1)}
            &=
            \begin{cases}
                x_i^{(t)}-\gamma m_1^{(t)}, & i=1,\\
                \displaystyle\sum_{j\in\mathrm{P}(i)}x_j^{(t)}, & i\ne1,
            \end{cases}\\
            y_i^{(t+1)}
            &=
            \sum_{j\in\mathrm{Ch}(i)}y_j^{(t)}
            +g_i(x_i^{(t+1)},\xi_i^{(t+1)})
            -g_i(x_i^{(t)},\xi_i^{(t)}),\\
            m_1^{(t+1)}
            &=\beta m_1^{(t)}+(1-\beta)y_1^{(t+1)}
            \quad\text{if }i=1.
        \end{aligned}
        \]
    \EndFor
    \EndFor
    \State {\bfseries Output:} Root iterate $x_1^{(T)}$.
\end{algorithmic}
\end{algorithm}

Figure~\ref{fig:tree-rgt-information-flow} illustrates both the simultaneous
tree communication in one iteration and the resulting depth-dependent
round-trip delay.
\begin{figure}[ht]
    \centering
    \subfloat[]{
        \includegraphics[width=0.34\textwidth]{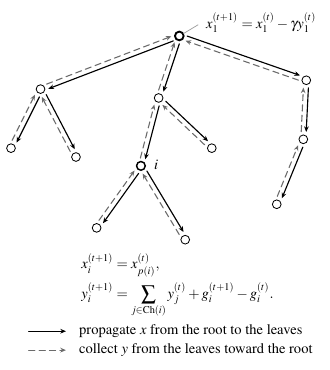}
    }
    \hfill
    \subfloat[]{
        \includegraphics[width=0.62\textwidth]{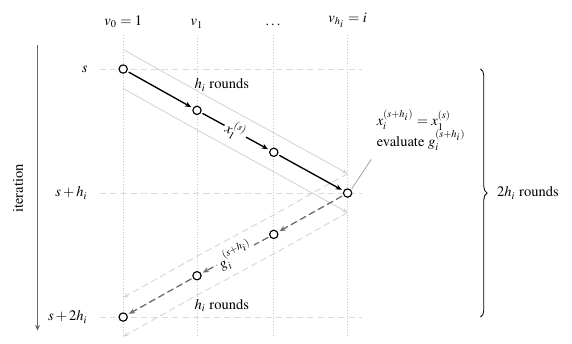}
    }
    \caption{Pipelined information flow in Tree-RGT.
    (a) At iteration $t$, decision-variable information propagates away from
    the root while gradient-tracking information propagates toward it; messages
    sent from the time-$t$ states enter the time-$(t+1)$ updates.
    (b) Along the unique path $v_0=1,\ldots,v_{h_i}=i$, the highlighted trajectory
    traces $x_1^{(s)}$ to agent $i$ and the resulting gradient contribution back
    to the root, a $2h_i$-round trip. Lighter trajectories show adjacent stages
    of the pipeline.}
    \label{fig:tree-rgt-information-flow}
\end{figure}

Because every non-root agent copies its parent's preceding iterate and all
agents start from $x^{(0)}$, induction on $h_i$ gives
\begin{equation}
    \label{eq:tree-rgt-decision-delay}
    x_i^{(t)}
    =
    \begin{cases}
        x^{(0)}, & 0\le t<h_i,\\
        x_1^{(t-h_i)}, & t\ge h_i.
    \end{cases}
\end{equation}
Thus, after an $h_i$-step downward warm-up, agent $i$ holds the root iterate
delayed by exactly $h_i$ steps. This is a deterministic routing delay rather
than an incomplete-mixing error.

For the upward recursion, the root tracker $y_1^{(t)}$ aggregates
the agents' stochastic gradients. Write
$g_i^{(s)}:=g_i(x_i^{(s)},\xi_i^{(s)})$. Because
$1\in\mathrm{Ch}(1)$, the root retains its preceding aggregate, and the
gradient increments telescope as they propagate upward. Unrolling the
recursion gives
\begin{equation}
    \label{eq:tree-rgt-root-tracker}
    y_1^{(t)}
    =
    \sum_{i\in\cN:\,h_i\le t}
    g_i^{(t-h_i)}
    =
    \sum_{i\in\cN:\,h_i\le t}
    g_i\bigl(x_i^{(t-h_i)},\xi_i^{(t-h_i)}\bigr).
\end{equation}
For $t<D$, this sum covers only agents with $h_i\le t$; for $t\ge D$,
it contains exactly one gradient contribution from every agent.

These contributions have delays determined by the tree depths.
Combining \eqref{eq:tree-rgt-root-tracker} and
\eqref{eq:tree-rgt-decision-delay}, agent $i$'s contribution is evaluated at
$x_1^{(t-2h_i)}$ once $t\ge2h_i$. During $D\le t<2D$, all agents are already
represented, although those with $2h_i>t$ still use $x^{(0)}$. After the
uniform round-trip warm-up $t\ge2D$,
\[
    y_1^{(t)}
    =
    \sum_{i\in\cN}
    g_i\bigl(x_1^{(t-2h_i)},\xi_i^{(t-h_i)}\bigr).
\]
Thus, $y_1^{(t)}$ aggregates stochastic gradients evaluated along the
root trajectory with node-specific round-trip delays $2h_i$.

The root then applies momentum to this delayed aggregate to form
its actual descent direction $m_1^{(t)}$. Unrolling the momentum recursion
and substituting \eqref{eq:tree-rgt-root-tracker} gives, for $t\ge0$,
\begin{equation}
    \label{eq:tree-rgt-root-momentum}
    \begin{aligned}
        m_1^{(t)}
        &=\beta^t g_1^{(0)}
        +(1-\beta)\sum_{s=1}^{t}\beta^{t-s}
        \sum_{i\in\cN:\,h_i\le s}g_i^{(s-h_i)}.
    \end{aligned}
\end{equation}
The direction therefore combines the current
routed aggregate with a geometrically weighted history of earlier aggregates,
each of which already contains delayed local gradients. This temporal
smoothing is designed to accelerate stochastic optimization along the root
trajectory. Since only the root uses a direction to update the model,
only the root needs to retain momentum; the other agents continue to
disseminate models and aggregate gradient increments. Besides, when $\beta=0$,
$m_1^{(t)}=y_1^{(t)}$.

\textbf{Comparison with STPP.}
STPP \cite{you2025distributed} can also use a bidirected spanning tree.
However, it performs local descent updates at all agents and does not incorporate
momentum. Tree-RGT confines descent to the root, where momentum combines
the current routed gradient aggregate with its history without additional
communication. Tree-RGT also improves the diameter dependence of the
transient bound (see Table~\ref{tab:one-communication-comparison}) and achieves
uniformly optimal network dependence over $\mathsf{G}_{n,\hat D}$
(see Section~\ref{sec:main-convergence-guarantee}).
Section~\ref{sec:numerical-experiments} further shows that Tree-RGT achieves better
practical performance than STPP.

\subsection{Matrix Form and Finite-Step Routing Properties}
\label{sec:ost-matrix}

We next write Algorithm~\ref{alg:ostl} in matrix form and record the finite-step
routing properties of its propagation matrices.
Let $\bfe_1\in\reals^n$ be the first standard basis vector and set
$\bpi:=\bfe_1$. Using the local routing sets constructed by
Algorithm~\ref{alg:gt}, define the binary propagation matrices
$\bR,\bC\in\reals^{n\times n}$ elementwise, for all $i,j\in\cN$, by
\begin{equation}
    \label{eq:rc}
    \brk{\bR}_{ij}
    :=
    \begin{cases}
        1, & \text{if } j\in\mathrm{P}(i),\\
        0, & \text{otherwise,}
    \end{cases}
    \qquad
    \brk{\bC}_{ij}
    :=
    \begin{cases}
        1, & \text{if } j\in\mathrm{Ch}(i),\\
        0, & \text{otherwise.}
    \end{cases}
\end{equation}
These matrices encode deterministic tree routing rather than weighted mixing.
Since $j\in\mathrm{Ch}(i)$ if and only if $i\in\mathrm{P}(j)$, these
definitions give $\bC=\bR^\T$. Every row of $\bR$ and every column of $\bC$
contains exactly one unit entry. Consequently,
\[
    \bR\mone=\mone,
    \qquad
    \mone^\T\bC=\mone^\T,
    \qquad
    \bpi^\T\bR=\bpi^\T,
    \qquad
    \bC\bpi=\bpi.
\]
The diagonal entries $\brk{\bR}_{11}=\brk{\bC}_{11}=1$ represent local state
retention at the root. They require no physical self-communication.

For any stacked matrix, left multiplication by $\bR$ replaces each non-root
row by its parent's row, whereas left multiplication by $\bC$ sums the rows
received from the corresponding child set. With the stacked variables defined
in Section~\ref{sec:notations}, set $\bS:=\bpi\bpi^\T$, the root-selection
matrix. Define the momentum stack matrix by
\[
    \bM^{(t)}:=\bpi(m_1^{(t)})^\T\in\reals^{n\times p},
\]
whose first row is $(m_1^{(t)})^\T$ and whose remaining rows are zero.
Algorithm~\ref{alg:ostl} then has the matrix form
\begin{equation}
    \label{eq:matrix-update}
    \begin{aligned}
        \bX^{(t+1)}
        &=
        \bR\bX^{(t)}-\gamma\bM^{(t)},\\
        \bY^{(t+1)}
        &=
        \bC\bY^{(t)}+\bG^{(t+1)}-\bG^{(t)},\\
        \bM^{(t+1)}
        &=\beta\bM^{(t)}+(1-\beta)\bS\bY^{(t+1)},
    \end{aligned}
\end{equation}
The common initialization in Algorithm~\ref{alg:ostl} gives
\[
    \bX^{(0)}=\mone(x^{(0)})^\T,
    \qquad
    \bY^{(0)}=\bG^{(0)},
    \qquad
    \bM^{(0)}=\bS\bG^{(0)}.
\]
Thus, $\bR$ and $\bC$ encode the two simultaneous one-hop information flows,
while $\bM^{(t)}$ has support only at the root and $\bS$ selects the
root row in the momentum update.

To quantify how many nodes have been reached after $k$ propagation steps, let
\[
    n_k
    :=
    \bigl|\{i\in\cN:h_i\le k\}\bigr|,
    \qquad k\ge0 \text{ an integer}.
\]
Then $n_k=n$ for all $k\ge D$, and the tail-sum identity gives
\[
    D_{\avg}
    =
    \frac{1}{n}\sum_{i\in\cN}h_i
    =
    \frac{1}{n}\sum_{k=0}^{D-1}\prt{n-n_k}.
\]

\begin{lemma}[Finite-step propagation bound]
    \label{lem:Rk_norm}
    For every integer $k\ge 0$,
    \[
        \begin{aligned}
        \norm{\bR^k-\mone\bpi^\T}_2^2
        &=
        \norm{\bC^k-\bpi\mone^\T}_2^2
        \le 2\prt{n-n_k}.
        \end{aligned}
    \]
\end{lemma}
\begin{proof}
    See Appendix~\ref{pf:lem:Rk_norm}.
\end{proof}

For an entrywise interpretation, let $p^0(i):=i$ and
$p^{k+1}(i):=p\bigl(p^k(i)\bigr)$. For all $i,j\in\cN$ and integers $k\ge0$,
\begin{equation}
    \label{eq:RC-power-entrywise}
    \brk{\bR^k}_{ij}
    =
    \mathbbm{1}_{\{p^k(i)=j\}},
    \qquad
    \brk{\bC^k}_{ij}
    =
    \mathbbm{1}_{\{p^k(j)=i\}}.
\end{equation}
Thus, the first column of $\bR^k$ marks the agents with $h_i\le k$, while
column $j\ne1$ marks the agents whose $k$-th ancestor is $j$. Since
$\bC^k=(\bR^k)^\T$, each column of $\bC^k$ follows one
agent's information for $k$ steps toward the root, where it remains after
arrival.

\begin{corollary}[Exact finite-step routing]
    \label{lem:Rk_d}
    At the tree depth $D$,
    \[
        \bR^D=\mone\bpi^\T,
        \qquad
        \bC^D=\bpi\mone^\T .
    \]
    The same identities hold for every integer $k\ge D$, and $D$ is the
    smallest nonnegative integer exponent for which they hold.
\end{corollary}
\begin{proof}
    See Appendix~\ref{pf:lem:Rk_d}.
\end{proof}

For any stacked matrix $\bU$, the identity
$\bR^D\bU=\mone\bpi^\T\bU$ replicates the root row at every agent, whereas
$\bC^D\bU=\bpi\mone^\T\bU$ places the sum of all rows at the root. Hence, both
routing directions complete exactly after $D$ one-hop propagation steps.

\subsection{\texorpdfstring{Convergence Guarantees for Tree-RGT}{Convergence Guarantees for Tree-RGT}}
\label{sec:main-convergence-guarantee}

Tree routing and root momentum together yield a convergence guarantee
whose network dependence is determined by the tree depth. Theorem~\ref{thm:convergence}
gives a compact bound consisting of a stochastic term of order
$1/\sqrt{nT}$ and a routing correction of order $D/T$.
It explicitly accounts for the initial gradient energy and stochastic noise,
and allows the momentum coefficient to approach one as $T$ grows.
The momentum choice ensures $(1-\beta)^{-1}\le\sqrt{T}$,
allowing the initial momentum noise to be absorbed into the stochastic term.
Section~\ref{sec:convergence-analysis} develops the supporting estimates.

\begin{theorem}[Convergence of Tree-RGT]
    \label{thm:convergence}
    Let $\cG=(\cN,\cE)$ be a connected undirected communication graph with
    $n\ge2$, and let $D$ be the depth of the tree generated by
    Algorithm~\ref{alg:gt}. Suppose Assumptions~\ref{a.smooth}
    and~\ref{a.var} hold, and define
    \[
    \Delta_f:=f(x^{(0)})-f^\star,
    \qquad H_0:=\frac{1}{n}\sum_{i=1}^{n}
    \norm{\nabla f_i(x^{(0)})}^2.
    \]
    Assume $\Delta_f L+H_0>0$. For any integer $T\ge1$, run
    Algorithm~\ref{alg:ostl} with
    \[
    \begin{aligned}
        \gamma&=\min\crk{
        \frac{\sqrt{\Delta_f L+H_0}}
        {L\sqrt{nT\sigma^2}},\,
        \frac{1}{60nDL}},\\
        \beta&\le 1-\max\crk{
        \min\crk{60\sqrt{\frac{n(\Delta_f L+H_0)}{T\sigma^2}},\,
        \frac{1}{D}},\,
        \frac{1}{\sqrt{T}}}.
    \end{aligned}
    \]
    Then $0\le\beta<1$, and
    \[
    \begin{aligned}
        &\frac{1}{T}\sum_{t=0}^{T-1}
        \expect\norm{\nabla f(x_1^{(t)})}^2 \le
        \frac{240\sqrt{(\Delta_f L+H_0)\sigma^2}}
        {\sqrt{nT}}
        +\frac{1200D(\Delta_f L+H_0)}{T}.
    \end{aligned}
    \]
\end{theorem}
\begin{proof}
    See Appendix~\ref{pf:thm:convergence}.
\end{proof}

\textbf{Transient iterations.}
For fixed $\Delta_f$, $L$, $\sigma^2>0$, and $H_0$, the stochastic term
in Theorem~\ref{thm:convergence} dominates once $T$ exceeds a
problem-dependent constant multiple of $nD^2$.
Since $D\le D_{\cG}$, Tree-RGT achieves linear speedup after
$\cO(nD_{\cG}^2)$ transient iterations by matching the
$\cO((nT)^{-1/2})$ rate of centralized mini-batch SGD, with one stochastic
gradient per agent and one communication round per iteration.

This transient bound holds throughout the admissible range of $\beta$ in
Theorem~\ref{thm:convergence}, including $\beta=0$.

\begin{remark}[The role of the average depth.]
    Theorem~\ref{thm:convergence} uses $D_{\avg}\le D$ to remove
    $D_{\avg}$ from the final bound. The average depth still yields finer
    estimates of the root-relative disagreement in Lemma~\ref{lem:PiX}
    and the descent of the global objective in Lemma~\ref{lem:smooth}.
\end{remark}

To compare with the minimax benchmark in
\eqref{eq:diameter-minimax-optimal}, we allow $B$ independent
stochastic-gradient samples per agent per round and translate the convergence
bound into an accuracy complexity.
\begin{corollary}[Accuracy complexity with $B$-sample mini-batches]
    \label{cor:tree-rgt-accuracy}
    Under the assumptions of Theorem~\ref{thm:convergence}, fix an integer
    $B\ge1$. At each oracle evaluation, every agent independently draws a
    mini-batch $\boldsymbol{\xi}_i=(\xi_{i,1},\ldots,\xi_{i,B})$ and uses
    \[
        g_{i,B}(x,\boldsymbol{\xi}_i)
        :=\frac{1}{B}\sum_{b=1}^B g_i(x,\xi_{i,b}).
    \]
    For any integer $T\ge1$, choose $\gamma_B$ and $\beta_B$ by replacing
    $\sigma^2$ with $\sigma^2/B$ in Theorem~\ref{thm:convergence}.
    Then
    \[
    \begin{aligned}
        &\frac{1}{T}\sum_{t=0}^{T-1}
        \expect\norm{\nabla f(x_1^{(t)})}^2\le
        \frac{240\sqrt{(\Delta_f L+H_0)\sigma^2}}
        {\sqrt{nBT}}
        +\frac{1200D(\Delta_f L+H_0)}{T}.
    \end{aligned}
    \]
    Consequently, for any $\varepsilon>0$, taking
    \begin{equation}
        \label{eq:tree-rgt-epsilon-trajectory}
        T=\left\lceil
        \frac{230400(\Delta_f L+H_0)\sigma^2}
        {nB\varepsilon^4}
        +\frac{2400D(\Delta_f L+H_0)}{\varepsilon^2}
        \right\rceil
    \end{equation}
    suffices to ensure
    $T^{-1}\sum_{t=0}^{T-1}\expect\norm{\nabla f(x_1^{(t)})}^2
    \le\varepsilon^2$.
\end{corollary}
\begin{proof}
    See Appendix~\ref{pf:cor:tree-rgt-accuracy}.
\end{proof}

\textbf{Uniformly optimal network dependence.}
To obtain a uniform bound over $\mathsf{G}_{n,\hat D}$, fix an admissible
optimization--oracle instance. Let
    $\mathsf{T}_\varepsilon^{\mathrm{Tree\text{-}RGT}}
    \prt{\mathsf{G}_{n,\hat D},B}$ denote the worst-case
    $\varepsilon$-round complexity of $B$-sample Tree-RGT over
    $\mathsf{G}_{n,\hat D}$, where the algorithm outputs
    $\widehat x=x_1^{(\widehat t)}$ and $\widehat t$ is uniformly distributed
    over $\{0,\ldots,T-1\}$. The fixed problem instance and all associated
    problem-dependent quantities are suppressed from the notation.
    For this output rule, Corollary~\ref{cor:tree-rgt-accuracy} and
    Jensen's inequality ensure $\expect\norm{\nabla f(\widehat x)}\le\varepsilon$.
    Using $D\le D_{\cG}=\hat D$ gives, with
    $\Delta_f$, $L$, $H_0$, and $\sigma^2$ fixed,
    \begin{equation}
        \label{eq:tree-rgt-class-complexity}
        \mathsf{T}_\varepsilon^{\mathrm{Tree\text{-}RGT}}
        \prt{\mathsf{G}_{n,\hat D},B}
        =
        \cO\prt{
            \frac{1}{nB\varepsilon^4}
            +
            \frac{\hat D}{\varepsilon^2}
        }.
    \end{equation}
    This bound is uniform over $\mathsf{G}_{n,\hat D}$ and matches the powers
    of $n$ and $\hat D$ in \eqref{eq:diameter-minimax-optimal}. Hence Tree-RGT
    has \emph{uniformly optimal network dependence} over
    $\mathsf{G}_{n,\hat D}$ in the sense of
    Section~\ref{sec:lower-bounds-diameter}.

\begin{remark}[One-time tree-construction cost.]
    The bounds above count optimization rounds after the routing tree has
    been constructed.
    Algorithm~\ref{alg:gt} is executed only once and takes at most $n$ synchronous
control rounds, using one-bit discovery and acknowledgment messages. Thus a run
of $T$ Tree-RGT iterations takes at most $T+n$ wall-clock rounds when this setup
is included. For a fixed network, the additive setup cost is lower order
in the high-accuracy regime $T_\varepsilon\to\infty$ and is further amortized
when the routing tree is reused.
\end{remark}

\section{Convergence Analysis}
\label{sec:convergence-analysis}

This section develops the technical estimates used in the proof of
Theorem~\ref{thm:convergence}. We first combine the Tree-RGT recursion with the
finite-step routing identities from Section~\ref{sec:ost-matrix} to obtain
exact representations of the gradient tracker, root momentum, and network
disagreement, identifying the finite-memory structure of tree routing and the
geometrically decaying history retained by momentum.
Since the momentum state is confined to the root, the analysis requires no
inter-agent momentum-disagreement estimate; its interaction with tree routing
is captured by a convolution kernel. We then control the
stochastic-gradient error, the error introduced by momentum, the root
increments, and the disagreement through a sequence of coupled estimates.
Finally, these estimates are inserted into a descent inequality to establish
the stated convergence rate. Elementary deterministic and probabilistic tools
used within the detailed proofs are collected in
Appendix~\ref{app:convergence-proofs}.

\subsection{Exact Representations of Tree Routing and Momentum}

For notational convenience, throughout this section we use the shorthand
$\nabla\bF^{(t)}:=\nabla\bF(\bX^{(t)})$ and write the stacked
stochastic-gradient error as
$\bTh^{(t)}:=\bG^{(t)}-\nabla\bF^{(t)}$.

To distinguish the finite-memory structure of tree routing from the
geometrically decaying history of momentum in \eqref{eq:matrix-update}, we introduce
two decompositions specific to the convergence analysis. The first separates
the root-aligned component of an iterate from the discrepancy between the
agents and the root. Define
\begin{equation}
    \label{eq:analysis-projections}
    \bPi_{\bR}:=\bI-\mone\bpi^\T,
    \qquad
    \bPi_{\bC}:=\bI-\bpi\mone^\T.
\end{equation}
Accordingly, set
\[
\begin{aligned}
    \bar{\bX}^{(t)}
    &:= \mone\bpi^\T \bX^{(t)},
    \qquad
    \Delta\bar{\bX}^{(t)}
    := \bar{\bX}^{(t+1)}-\bar{\bX}^{(t)},\\
    \hat{\bX}^{(t)}
    &:= \bPi_{\bR}\bX^{(t)},
    \qquad
    \nabla\bar{\bF}^{(t)}
    := \nabla \bF^{(t)}-\nabla \bF(\bar{\bX}^{(t)}).
\end{aligned}
\]

Here, $\bar{\bX}^{(t)}$ replicates the root iterate across all rows,
$\hat{\bX}^{(t)}$ measures the resulting root-relative disagreement, and
$\nabla\bar{\bF}^{(t)}$ is the gradient mismatch generated by that
disagreement. The second decomposition resolves the push dynamics according to
the tree layer whose information reaches the root. For each
$k\in\{0,\ldots,D\}$, define the $k$-th tree layer by
\[
    \cI_{1,k}:=\bigl\{i\in\cN:h_i=k\bigr\},
\]
and let $\bfe_{1,k}\in\reals^n$ denote its indicator vector. In particular,
$\cI_{1,0}=\{1\}$ and $\bfe_{1,0}=\bpi$. To identify the information
contributed by each layer, define the layer-increment matrices
\begin{equation}
    \label{eq:Ak}
    \bA_0:=\bI,
    \qquad
    \bA_k:=\bC^k-\bC^{k-1},\quad k\ge 1.
\end{equation}
The next identity makes this interpretation precise: after left multiplication
by $\bpi^\T$, $\bA_k$ selects exactly the nodes at depth $k$.
This identity depends only on the tree routing matrix $\bC$ and is
independent of the momentum coefficient $\beta$.
\begin{lemma}[Layer-increment identity]
    \label{lem:Ak}
    For every $0\le k\le D$,
    \[
        \bpi^\T\bA_k=\bfe_{1,k}^\T.
    \]
    Moreover, $\bA_k=\mathbf{0}$ for all $k>D$.
\end{lemma}
\begin{proof}
    See Appendix~\ref{pf:lem:Ak}.
\end{proof}

We derive the closed forms sequentially from \eqref{eq:matrix-update}.
Using $\bY^{(0)}=\bG^{(0)}$, unrolling the gradient-tracker recursion and
collecting the coefficient of each stochastic-gradient matrix gives, for
$t\ge0$,
\begin{equation}
    \label{eq:pp2-layer}
    \begin{aligned}
        \bY^{(t)}
        &=\bC^t\bG^{(0)}
        +\sum_{s=0}^{t-1}\bC^{t-s-1}
        \brk{\bG^{(s+1)}-\bG^{(s)}}=\sum_{k=0}^{\min\crk{t,D}}\bA_k\bG^{(t-k)}.
    \end{aligned}
\end{equation}
Here and below, empty sums are understood as zero. The second equality
uses $\bA_k=\mathbf{0}$ for $k>D$, so the gradient tracker retains its
finite-memory representation.

Next, unrolling the momentum recursion gives
\begin{equation}
    \label{eq:momentum-unrolled}
    \bM^{(t)}
    =\beta^t\bM^{(0)}
    +(1-\beta)\sum_{s=1}^{t}\beta^{t-s}\bS\bY^{(s)}.
\end{equation}
To collect the tree-layer and momentum weights, define the convolution
kernel
\begin{equation}
    \label{eq:momentum-kernel}
    \bK_k
    :=(1-\beta)\sum_{\ell=0}^{\min\crk{k,D}}
    \beta^{k-\ell}\bS\bA_\ell,
    \qquad k\ge0.
\end{equation}
Since $\bM^{(0)}=\bS\bY^{(0)}=\bS\bG^{(0)}$, including the time-zero
tracker in the sum in \eqref{eq:momentum-unrolled} and then substituting
\eqref{eq:pp2-layer} yields
\begin{equation}
    \label{eq:momentum-closed-form}
    \begin{aligned}
        \bM^{(t)}
        &=\beta^t\brk{\bM^{(0)}-(1-\beta)\bS\bY^{(0)}}
        +(1-\beta)\sum_{s=0}^{t}\beta^{t-s}\bS\bY^{(s)}\\
        &=\beta^{t+1}\bM^{(0)}
        +(1-\beta)\sum_{s=0}^{t}\beta^{t-s}\bS\bY^{(s)}\\
        &=\beta^{t+1}\bM^{(0)}
        +(1-\beta)\sum_{s=0}^{t}\sum_{\ell=0}^{\min\crk{s,D}}
        \beta^{t-s}\bS\bA_\ell\bG^{(s-\ell)}\\
        &=\beta^{t+1}\bM^{(0)}
        +\sum_{r=0}^{t}\bK_{t-r}\bG^{(r)}.
    \end{aligned}
\end{equation}
The second equality uses $\bM^{(0)}=\bS\bY^{(0)}$. The final equality
follows by setting $r=s-\ell$ and using \eqref{eq:momentum-kernel}.

Finally, unrolling the iterate recursion for $t\ge1$, substituting
\eqref{eq:momentum-closed-form}, and reordering the finite sums gives
\begin{equation}
    \label{eq:pp1-layer}
    \begin{aligned}
        \bX^{(t)}
        &=\bR^t\bX^{(0)}
        -\gamma\sum_{j=1}^{t}\bR^{j-1}\bM^{(t-j)}\\
        &=\bR^t\bX^{(0)}
        -\gamma\beta\sum_{j=1}^{t}
        \beta^{t-j}\bR^{j-1}\bM^{(0)}  -\gamma\sum_{s=0}^{t-1}\sum_{j=1}^{t-s}
        \bR^{j-1}\bK_{t-s-j}\bG^{(s)}.
    \end{aligned}
\end{equation}
Since $\bpi^\T\bR=\bpi^\T$, the root-trajectory increment satisfies
\begin{equation}
    \label{eq:root-diff}
    \begin{aligned}
        \Delta\bar{\bX}^{(t)}
        &=-\gamma\mone\bpi^\T\bM^{(t)}=-\gamma\beta^{t+1}\mone\bpi^\T\bM^{(0)}
        -\gamma\mone\sum_{s=0}^{t}\bpi^\T\bK_{t-s}\bG^{(s)}.
    \end{aligned}
\end{equation}
The common initialization gives
$\bPi_{\bR}\bR^t\bX^{(0)}=\mathbf{0}$. Applying $\bPi_{\bR}$ to
\eqref{eq:pp1-layer} therefore yields
\begin{equation}
    \label{eq:hatX-Y}
    \begin{aligned}
        \hat{\bX}^{(t)}
        &=-\gamma\beta\sum_{j=1}^{t}
        \beta^{t-j}\bPi_{\bR}\bR^{j-1}\bM^{(0)}-\gamma\sum_{s=0}^{t-1}\sum_{j=1}^{t-s}
        \bPi_{\bR}\bR^{j-1}\bK_{t-s-j}\bG^{(s)}.
    \end{aligned}
\end{equation}

The following lemma bounds the momentum routing kernel in
\eqref{eq:hatX-Y}.
\begin{lemma}[Pull--push interaction bound]
    \label{lem:RC-interaction}
    For $0\le\beta<1$ and every integer $q\ge1$, define
    \[
    \begin{aligned}
        \bJ_q&:=\sum_{j=1}^{\min\crk{q,D}}
        \bPi_{\bR}\bR^{j-1}\bK_{q-j},\\
        \rho_q&:=(1-\beta)\sum_{r=1}^{\min\crk{q,2D}}\beta^{q-r}.
    \end{aligned}
    \]
    Then
    \[
        \norm{\bJ_q}_F^2\le n^2\rho_q^2.
    \]
    Moreover,
    \[
        \sum_{q=1}^{\infty}\rho_q=2D,
        \qquad
        \sum_{q=1}^{\infty}\rho_q^2
        \le\min\crk{2D,\,4D^2\frac{1-\beta}{1+\beta}}.
    \]
\end{lemma}
\begin{proof}
    See Appendix~\ref{pf:lem:RC-interaction}.
\end{proof}

For $q>2D$, $\rho_q=\beta^{q-2D}(1-\beta^{2D})$.
Thus, momentum retains a geometric tail beyond the routing window.

\subsection{Stochastic Error and Descent Estimates}

With the exact representations of tree routing and momentum in place,
we now control the stochastic
and deterministic components that enter the root and disagreement recursions.
For this purpose, define the pre-sampling filtration by
\[
    \cF_0:=\{\emptyset,\Omega\},
    \qquad
    \cF_t
    :=
    \sigma\bigl(
        \xi_i^{(k)}:
        i\in[n],\ 0\le k\le t-1
    \bigr),
    \quad t\ge1.
\]
The iterate $\bX^{(t)}$ is $\cF_t$-measurable: $\bX^{(0)}$ is
deterministic, and for $t\ge1$ the iterate depends only on samples with time
indices at most $t-1$. The tracker $\bY^{(t)}$ and momentum $\bM^{(t)}$ are
$\cF_{t+1}$-measurable, since they may also depend on the time-$t$ samples.
Assumption~\ref{a.var} therefore
implies that, conditionally on $\cF_t$, the rows of $\bTh^{(t)}$ are
independent and satisfy
\begin{equation}
    \label{eq:noise-conditional-properties}
    \expect\brk{\bTh^{(t)}\mid\cF_t}=\mathbf{0},
    \qquad
    \expect\brk{
        \norm{\brk{\bTh^{(t)}}_{i:}}^2\mid\cF_t
    }
    \le\sigma^2,
    \quad i\in\cN.
\end{equation}
We begin with three consequences of these properties, which will be used to
estimate stochastic-gradient noise layer by layer.
\begin{lemma}
    \label{lem:var}
    Suppose Assumption \ref{a.var} holds. Then, for any index set
    $\cS\subseteq[n]$ with indicator vector $\bfe_{\cS}$,
    \[
        \expect\brk{\norm{\bfe_{\cS}^\T\bTh^{(t)}}_2^2}
        \le |\cS|\sigma^2 .
    \]
\end{lemma}
\begin{proof}
    By \eqref{eq:noise-conditional-properties}, conditionally on $\cF_t$ the
    rows of $\bTh^{(t)}$ are independent and centered, and each has conditional
    second moment at most $\sigma^2$. Therefore, the cross terms vanish and
    \[
        \expect\brk{\norm{\bfe_{\cS}^\T\bTh^{(t)}}_2^2\mid\cF_t}
        =\sum_{i\in\cS}\expect\brk{
        \norm{\brk{\bTh^{(t)}}_{i:}}_2^2\mid\cF_t}
        \le |\cS|\sigma^2.
    \]
    Taking total expectation proves the result.
\end{proof}

\begin{corollary}
    \label{lem:var0}
    Suppose Assumption \ref{a.var} holds. Then, for every $0\le k\le D-1$,
    \[
        \expect\brk{\norm{\prt{\bC^{k}-\bpi\mone^\T}\bTh^{(t)}}_F^2}
        \le 2\prt{n-n_k}\sigma^2 .
    \]
\end{corollary}
\begin{proof}
    Let $\bB_k:=\bC^k-\bpi\mone^\T$. Conditional independence and centering in
    \eqref{eq:noise-conditional-properties} imply
    \[
        \expect\brk{\norm{\bB_k\bTh^{(t)}}_F^2\mid \cF_t}
        \le \sigma^2\norm{\bB_k}_F^2 .
    \]
    Equation~\eqref{eq:Rk-residual-frobenius}, established in the proof of
    Lemma~\ref{lem:Rk_norm}, and the identity
    $\bC^k-\bpi\mone^\T=(\bR^k-\mone\bpi^\T)^\T$ give
    $\norm{\bB_k}_F^2=2(n-n_k)$. Taking total expectation proves the claim.
\end{proof}
\begin{corollary}
    \label{lem:var1}
    Suppose Assumption \ref{a.var} holds. Then
    \[
        \sum_{k=1}^{D}
        \expect \norm{\prt{\bpi-\mone}^\T\bA_k\bTh^{(t)}}_F^2
        \le \prt{n-1}\sigma^2 .
    \]
\end{corollary}
\begin{proof}
    Since $\mone^\T\bA_k=\mathbf{0}$ for $k\ge 1$ and
    $\bpi^\T\bA_k=\bfe_{1,k}^\T$ by Lemma \ref{lem:Ak},
    \[
        \prt{\bpi-\mone}^\T\bA_k=\bfe_{1,k}^\T .
    \]
    Applying Lemma \ref{lem:var} to the pairwise disjoint layers
    $\cI_{1,k}$ and summing over $k=1,\ldots,D$ gives
    $\sum_{k=1}^{D}|\cI_{1,k}|\sigma^2=(n-1)\sigma^2$.
\end{proof}

The preceding results control stochastic-gradient noise across the tree layers.
We next bound the deterministic gradient mismatch caused by evaluating local
gradients away from the root iterate. Since the layer increments partition the
agents according to their rooted depths, the resulting delayed terms can be
charged directly to the cumulative root-relative disagreement.
\begin{lemma}
    \label{lem:Ak-barF}
    Suppose Assumption \ref{a.smooth} holds. Then, for every integer $T\ge1$,
    \[
    \begin{aligned}
        \sum_{t=0}^{T-1}
        \expect\norm{
        \sum_{m=0}^{\min\crk{t,D}}\bpi^\T\bA_m\nabla \bar{\bF}^{(t-m)}
        }^2
        \le nL^2\sum_{t=0}^{T-1}\expect\norm{\hat{\bX}^{(t)}}_F^2 .
    \end{aligned}
    \]
\end{lemma}
\begin{proof}
    See Appendix~\ref{pf:lem:Ak-barF}.
\end{proof}

Since $\bar{\bX}^{(t)}=\mone(x_1^{(t)})^\T$, its increment satisfies
\[
    \norm{\Delta\bar{\bX}^{(t)}}_F^2
    =n\norm{x_1^{(t+1)}-x_1^{(t)}}^2.
\]
Thus, summing this quantity over $t=0,\ldots,T-1$ measures the accumulated
squared update energy of the root trajectory. The next lemma controls this
energy by separating the effects of stochastic-gradient noise, root-relative
disagreement, the routing initialization effect and the geometrically
decaying contribution of the initial momentum, and the global gradient along
the root trajectory.
\begin{lemma}
    \label{lem:X_diff}
    Suppose Assumptions~\ref{a.smooth} and \ref{a.var} hold and
    $0\le\beta<1$. If $\gamma\le1/(6n\sqrt{DD_{\avg}}L)$, then, for
    every integer $T\ge1$,
    \[
    \begin{aligned}
        & \sum_{t=0}^{T-1}\expect\norm{\Delta\bar{\bX}^{(t)}}_F^2
        \le9\gamma^2n^2\frac{1-\beta}{1+\beta}\sigma^2T
        +9\gamma^2n^2L^2
        \sum_{t=0}^{T-1}\expect\norm{\hat{\bX}^{(t)}}_F^2\\
        &\quad+18\gamma^2n \prt{n D_{\avg} + \frac{\beta^2}{1-\beta}}\norm{\nabla\bF^{(0)}}_F^2
        +9\gamma^2n^3
        \sum_{t=0}^{T-1}\expect\norm{\nabla f(x_1^{(t)})}^2 + 9\gamma^2n\frac{\beta^2}{1-\beta}
            \sigma^2.
    \end{aligned}
    \]
\end{lemma}
\begin{proof}
    See Appendix~\ref{pf:lem:X_diff}.
\end{proof}

To complement the preceding root-increment estimate, we next quantify how far
the local iterates deviate from the root trajectory. By the definition of
$\hat{\bX}^{(t)}$,
\[
    \norm{\hat{\bX}^{(t)}}_F^2
    =
    \sum_{i=1}^{n}\norm{x_i^{(t)}-x_1^{(t)}}^2.
\]
Hence, summing this quantity over $t=0,\ldots,T-1$ measures the
cumulative node-wise disagreement relative to the root, which is the
root-relative analogue of the standard consensus error. The following lemma
bounds this disagreement in terms of stochastic-gradient noise, the global
gradient along the root trajectory, and the effects of routing
initialization and initial momentum.
\begin{lemma}
    \label{lem:PiX}
    Suppose Assumptions~\ref{a.smooth} and \ref{a.var} hold and
    $0\le\beta<1$. If $\gamma\le1/(20nDL)$, then, for every integer $T\ge1$,
\[
\begin{aligned}
    \sum_{t=0}^{T-1}\expect\norm{\hat{\bX}^{(t)}}_F^2
    &\le5\gamma^2n^2\sigma^2T
    \min\crk{2D,\,4D^2\frac{1-\beta}{1+\beta}}+15\gamma^2n^3DD_{\avg}\sum_{t=0}^{T-1}
    \expect\norm{\nabla f(x_1^{(t)})}^2\\
    &\quad+30\gamma^2n^2D^2D_{\avg}\norm{\nabla\bF^{(0)}}_F^2 +5\gamma^2nDD_{\avg}\frac{\beta^2}{1-\beta^2}
    \brk{\norm{\nabla\bF^{(0)}}_F^2+\sigma^2}.
\end{aligned}
\]
\end{lemma}
\begin{proof}
    See Appendix~\ref{pf:lem:PiX}.
\end{proof}

\begin{remark}[A tree-based analogue of consensus-error bounds]
    \label{rem:tree-consensus-error}
    Lemma~\ref{lem:PiX} provides a consensus-error estimate: pairwise
    discrepancies satisfy
    \[
        \norm{x_i^{(t)}-x_j^{(t)}}^2
        \le2\norm{x_i^{(t)}-x_1^{(t)}}^2
        +2\norm{x_j^{(t)}-x_1^{(t)}}^2,
    \]
    and Assumption~\ref{a.smooth} gives
    $\norm{\nabla\bar{\bF}^{(t)}}_F^2
    \le L^2\norm{\hat{\bX}^{(t)}}_F^2$, controlling the gradient mismatch
    in Lemma~\ref{lem:X_diff}. Momentum preserves the finite routing delays
    in \eqref{eq:tree-rgt-decision-delay}. The network-dependent coefficients
    involve $D$ and $D_{\avg}$ without a spectral gap; for fixed $\gamma$
    and $\beta$, a shallower depth profile reduces these coefficients.
\end{remark}

\begin{remark}[Momentum and initialization]
    \label{rem:momentum-initialization}
    In Lemmas~\ref{lem:X_diff} and \ref{lem:PiX}, the accumulated noise
    coefficients involve
    \[
        \frac{1-\beta}{1+\beta}
        \quad\text{and}\quad
        \min\crk{2D,\,4D^2\frac{1-\beta}{1+\beta}},
    \]
    respectively, while the initial momentum coefficients involve
    $\beta^2/(1-\beta)$ and $\beta^2/(1-\beta^2)$.
    For fixed $\gamma$, larger $\beta$ can reduce the noise coefficients but
    increases the initial momentum coefficients, which diverge as $\beta\to1$.
    The initialization $\bM^{(0)}=\bS\bG^{(0)}$ contains both the initial
    gradient and stochastic noise.
\end{remark}

A layer-$m$ sample generated at time $t-m$, with $t\ge m$, first
reaches the root tracker at time $t$ and can first affect the root iterate
at time $t+1$, also with momentum. This gives the following orthogonality
property.
\begin{lemma}
    \label{lem:independ1}
    Suppose Assumption \ref{a.var} holds. For every $m\in\{0,1,\ldots,D\}$ and
    every $t\ge m$,
    \[
        \expect\left\langle
        \nabla f(x_1^{(t)}),
        \sum_{j\in\cI_{1,m}}
        \prt{
            g_j(x_j^{(t-m)},\xi_j^{(t-m)})
            -\nabla f_j(x_j^{(t-m)})
        }
        \right\rangle = 0 .
    \]
\end{lemma}
\begin{proof}
    Fix $m$ and $t\ge m$. If $t=0$, then $m=0$, and the claim follows
    from the deterministic initialization and Assumption~\ref{a.var}.
    For $t\ge1$, set $u=t-m$ and define
    \[
    \cF_{t,m}:=\sigma\bigl(
        \xi_\ell^{(r)}:\ell\in\cN,\ 0\le r\le t,
        \ \ell\notin\cI_{1,m}\ \text{or}\ r\ne u
    \bigr).
    \]
    Since $\cF_u\subseteq\cF_{t,m}$, each $x_j^{(u)}$ with
    $j\in\cI_{1,m}$ is $\cF_{t,m}$-measurable.
    The excluded samples do not include $\xi_1^{(0)}$, so
    $m_1^{(0)}=y_1^{(0)}$ is also $\cF_{t,m}$-measurable.

    We show that $x_1^{(t)}$ is $\cF_{t,m}$-measurable by induction.
    The initial point is deterministic. Suppose the root iterates through
    time $r<t$ are measurable. By \eqref{eq:tree-rgt-decision-delay}, all
    evaluation points in \eqref{eq:tree-rgt-root-tracker} for $y_1^{(r)}$
    are measurable. An excluded sample could occur in this tracker only
    if $r-m=u$, which would require $r=t$.
    Thus, $y_1^{(v)}$ is measurable for every $0\le v\le r$. The momentum representation
    \eqref{eq:momentum-unrolled} and the root update give
    \[
    \begin{aligned}
        m_1^{(r)}
        &=\beta^r m_1^{(0)}
        +(1-\beta)\sum_{v=1}^{r}\beta^{r-v}y_1^{(v)},\\
        x_1^{(r+1)}&=x_1^{(r)}-\gamma m_1^{(r)}.
    \end{aligned}
    \]
    Both are $\cF_{t,m}$-measurable, completing the induction.

    Write $\theta_j^{(s)}:=g_j(x_j^{(s)},\xi_j^{(s)})
    -\nabla f_j(x_j^{(s)})$. The excluded samples
    $\{\xi_j^{(u)}:j\in\cI_{1,m}\}$ are independent of $\cF_{t,m}$ by
    Assumption~\ref{a.var}. Since their evaluation points are measurable,
    pointwise unbiasedness gives
    $\expect[\theta_j^{(u)}\mid\cF_{t,m}]=\mathbf{0}$.
    The tower property therefore yields
    \[
    \begin{aligned}
        &\expect\ip{\nabla f(x_1^{(t)})}{
        \sum_{j\in\cI_{1,m}}\theta_j^{(u)}} =\expect\left\langle\nabla f(x_1^{(t)}),
        \sum_{j\in\cI_{1,m}}
        \expect[\theta_j^{(u)}\mid\cF_{t,m}]\right\rangle=0.
    \end{aligned}
    \]
\end{proof}

We next analyze descent along the root trajectory. The following lemma
combines smoothness, the momentum recursion, and the preceding error
estimates to bound the averaged squared gradient norm.

\begin{lemma}
    \label{lem:smooth}
    Suppose Assumptions~\ref{a.smooth} and \ref{a.var} hold and
    $n\ge2$. For any integer $T\ge1$, if
    \[
    0<\gamma\le\min\crk{\frac{1}{60nDL},\,
    \frac{1-\beta}{60nL}},
    \qquad 0\le\beta\le1-\frac{1}{\sqrt{T}},
    \]
    then
    \[
        \begin{aligned}
        & \frac{1}{T}\sum_{t=0}^{T-1}\expect\norm{\nabla f(x_1^{(t)})}^2 \le \frac{20\Delta_f}{\gamma n T} + 220 \gamma L \sigma^2  + \frac{60}{nT} \prt{D_{\avg} + \frac{1}{1-\beta}}\norm{\nabla\bF^{(0)}}_F^2.
    \end{aligned}
    \]
    where $\Delta_f:=f(x^{(0)})-\inf_x f(x)$.
\end{lemma}
\begin{proof}
    See Appendix~\ref{pf:lem:smooth}.
\end{proof}

\section{Numerical Experiments}
\label{sec:numerical-experiments}

We evaluate Tree-RGT on a heterogeneous
binary-classification problem and compare its performance with representative
mixing-based and routing-based methods. We consider three graph--matrix pairs
from Section~\ref{sec:examples}: the lazy-ring pair
$(\cG_{\mathrm{r}},\bW_{\mathrm{r}})$, the undirected exponential pair
$(\cG_{\mathrm{ex}},\bW_{\mathrm{ex}})$, and the balanced double-star pair
$(\cG_{\mathrm{ds}},\bW_{\mathrm{ds}})$. These pairs provide contrasting
graph-distance and matrix-mixing regimes.

\textbf{Problem and data.}
Following the nonconvex logistic-regression experiment in
\cite{huang2026accelerated}, we classify airplanes (label $+1$) and trucks
(label $-1$) from CIFAR-10 \cite{krizhevsky2009learning}.  The resulting
$10{,}000$ training samples are partitioned among $n$ agents.  If
$\cS_i$ is the local data set of agent $i$, with feature--label pairs
$(u_{i,j},v_{i,j})\in\reals^p\times\{-1,+1\}$, we solve
\begin{equation}
    \label{eq:experiment-nonconvex-logistic}
    \begin{aligned}
        &\min_{x\in\reals^p}
        f(x):=\frac{1}{n}\sum_{i=1}^n f_i(x),\\
        &f_i(x):=
        \frac{1}{|\cS_i|}\sum_{j\in\cS_i}
        \log\!\left(1+\exp(-v_{i,j} u_{i,j}^\T x)\right)+\frac{\omega}{2}\sum_{q=1}^p
        \frac{\brk{x}_q^2}{1+\brk{x}_q^2},
    \end{aligned}
\end{equation}
where $\omega=0.05$.  Each image is rescaled to $[0,1]$, flattened, augmented
with a constant bias coordinate, and normalized to unit Euclidean norm, giving
$p=3073$.  To induce heterogeneity, we sort the samples by label and divide
them into $n$ contiguous shards whose sizes differ by at most one.

\textbf{Experimental protocol.}
We compare Tree-RGT with decentralized SGD (DSGD), decentralized stochastic
gradient tracking (DSGT) \cite{pu2021distributed}, DSMT
\cite{huang2026accelerated}, STPP \cite{you2025distributed}, RelaySUM
\cite{vogels2021relaysum}, and synchronous centralized SGD with momentum
(CSGD). Tree-RGT and CSGD use normalized momentum with a fixed coefficient
$\beta=0.997$. CSGD averages one stochastic gradient per agent before
momentum filtering and serves as a centralized benchmark.
For $n\in\{128,256\}$, distributed methods run for $10{,}000$ communication
rounds, and CSGD performs $10{,}000$ synchronous updates. All methods use
mini-batch size one and share the unit-norm Gaussian initialization and
local sampling schedules within each run.

RelaySUM, STPP, and Tree-RGT share the breadth-first spanning tree
$\cT$ constructed by Algorithm~\ref{alg:gt}, rooted at agent~$1$.
STPP and Tree-RGT use its outward and inward orientations.
The tree diameter $\mathrm{Dist}_{\cT}$, measured in edges, equals
$n-1$, $\log_2 n$, and $3$ for the lazy ring, exponential graph, and
double-star, respectively.

The common base stepsize $\eta_t$ starts at $0.02$ and is halved after
every $2250$ rounds. DSGD, DSGT, DSMT, and CSGD use $\eta_t$; STPP and
Tree-RGT use $\gamma_t=\eta_t/n$ because their trackers aggregate gradient
sums; RelaySUM uses $\eta_t\sqrt{\mathrm{Dist}_{\cT}}$.
Following~\cite{huang2026accelerated}, DSMT uses
$\widetilde{\bW}=\bW$ on the lazy ring and
$\widetilde{\bW}=(\bI+\bW)/2$ on the other two graphs, with
$\eta_w=1/(1+\sqrt{1-\rho_w^2})$,
$\widetilde{\rho}_w=\sqrt{\eta_w}$, and
$\beta_{\mathrm{DSMT}}=1-(1-\widetilde{\rho}_w)/n^{1/3}$,
where $\rho_w$ is the second-largest eigenvalue modulus of
$\widetilde{\bW}$.

Every $50$ rounds, we evaluate the squared norm of the full empirical
gradient at the root iterate for Tree-RGT and STPP, the network-average
iterate for other distributed methods, and the common iterate for CSGD.
Figure~\ref{fig:ncvx-logistic} shows means over five independent runs,
with pointwise minimum--maximum envelopes.

\begin{figure}[p]
    \centering
    \subfloat[Lazy ring, $n=128$.]{
        \includegraphics[width=0.45\textwidth]{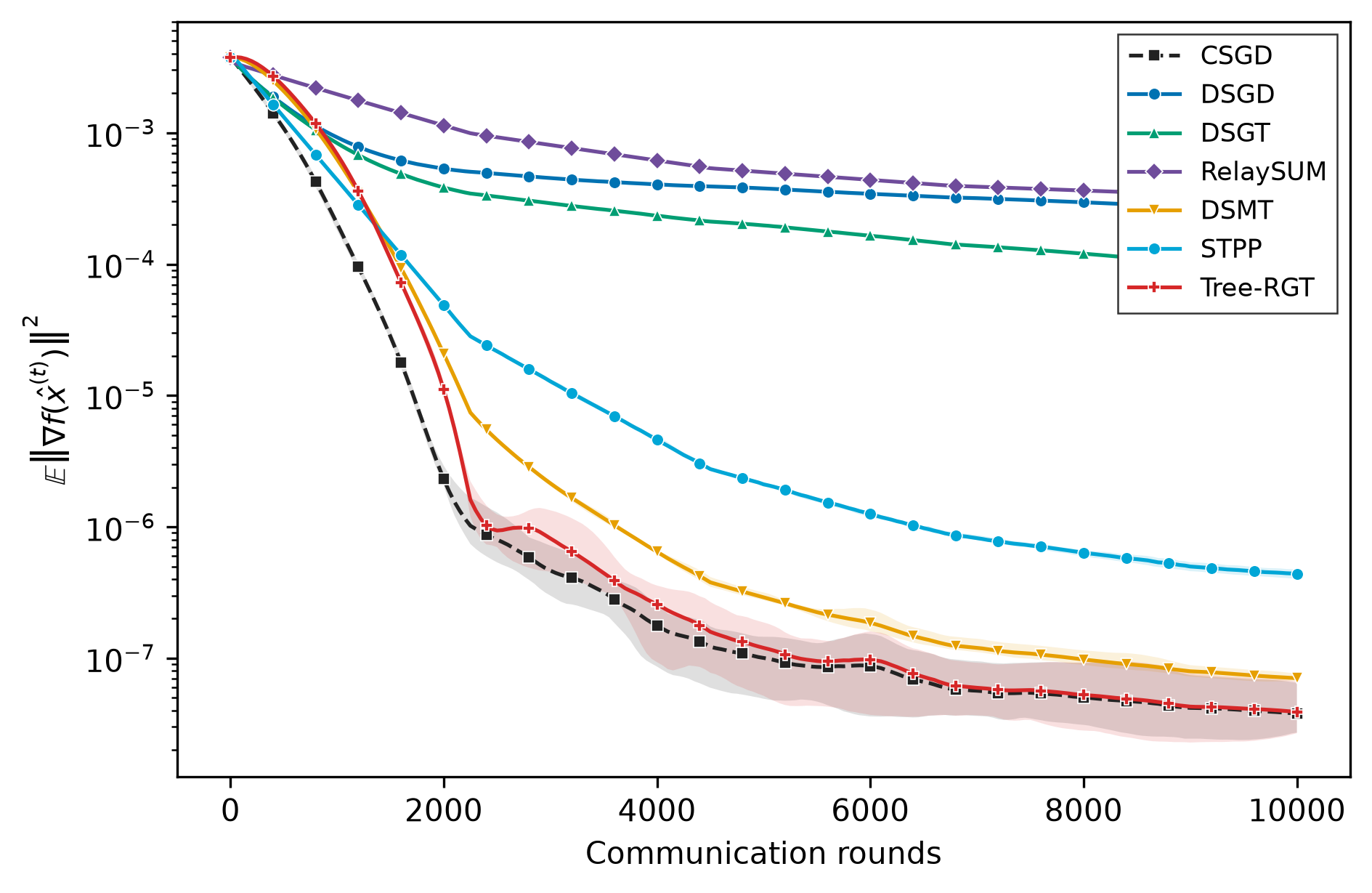}
    }%
    \hfill
    \subfloat[Lazy ring, $n=256$.]{
        \includegraphics[width=0.45\textwidth]{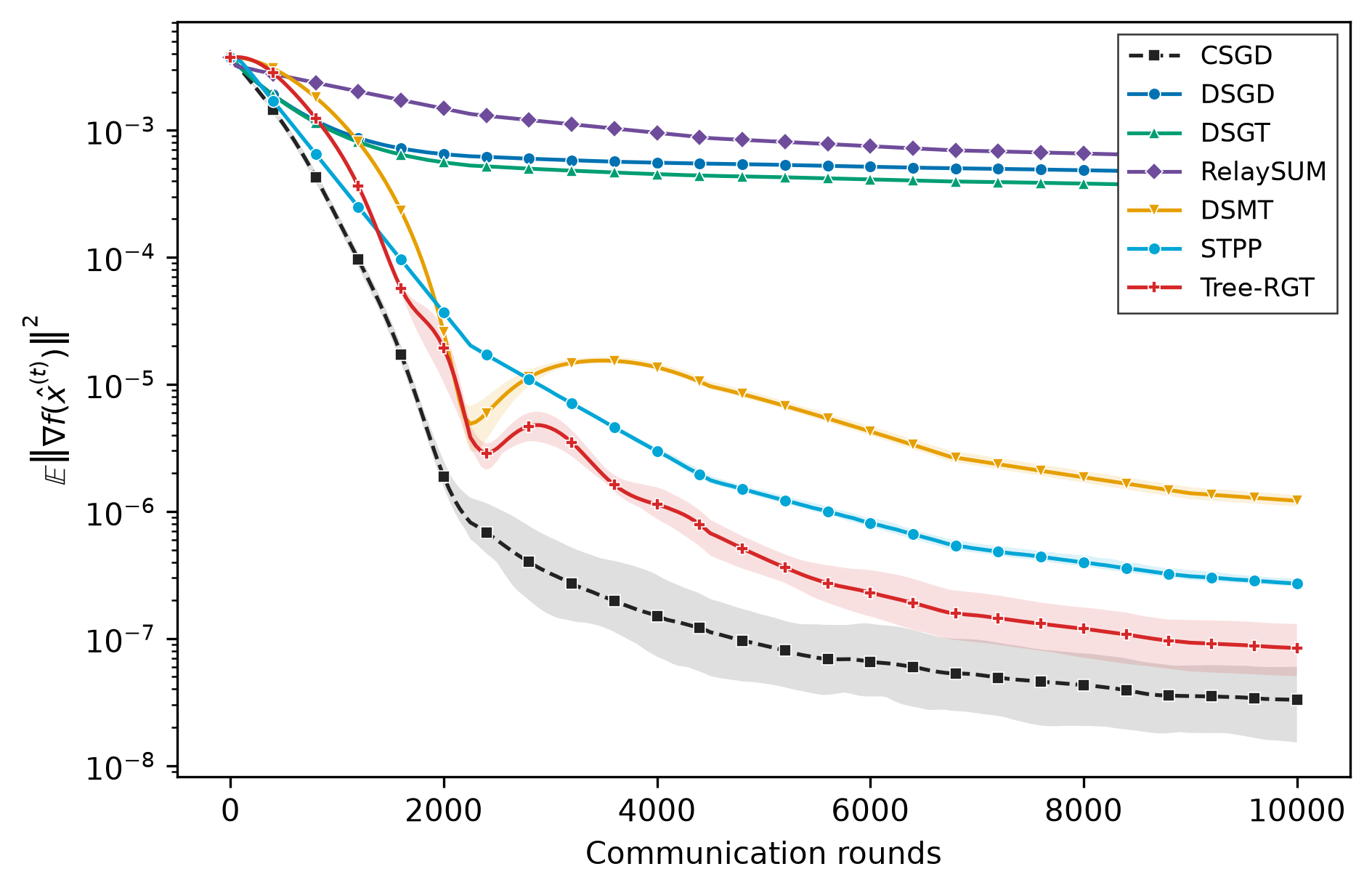}
    }

    \subfloat[Exponential graph, $n=128$.]{
        \includegraphics[width=0.45\textwidth]{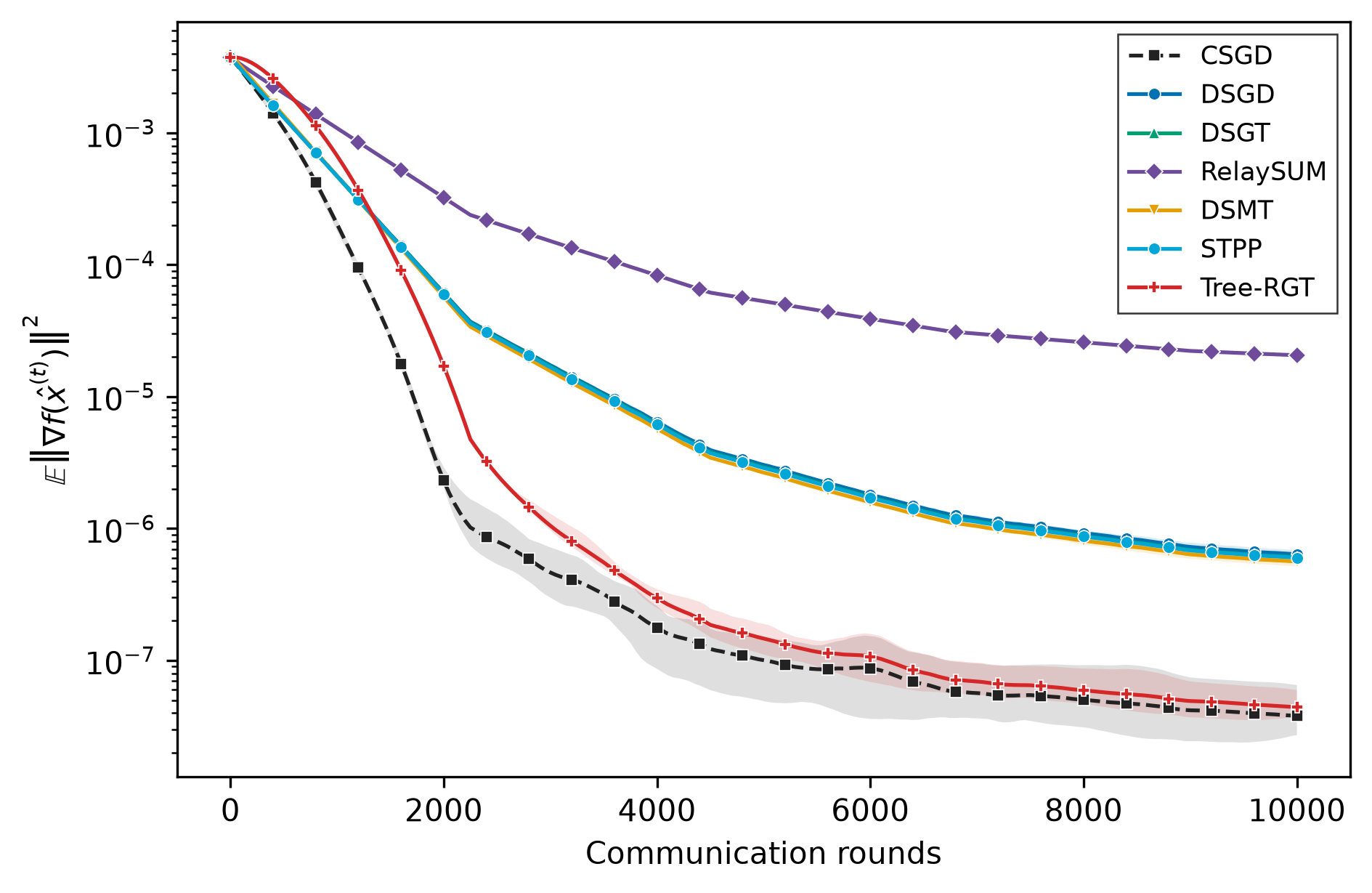}
    }%
    \hfill
    \subfloat[Exponential graph, $n=256$.]{
        \includegraphics[width=0.45\textwidth]{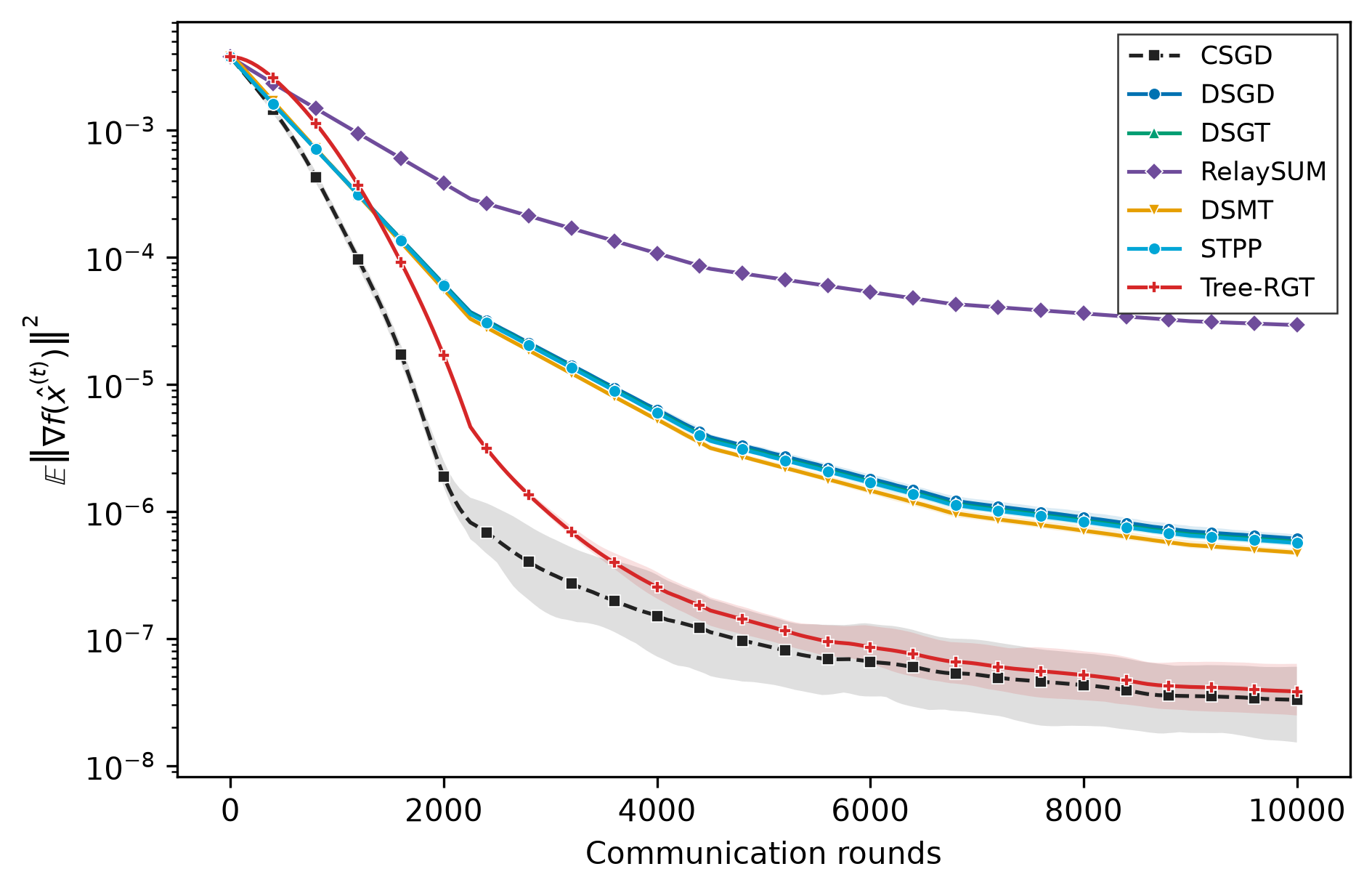}
    }

    \subfloat[Double-star, $n=128$.]{
        \includegraphics[width=0.45\textwidth]{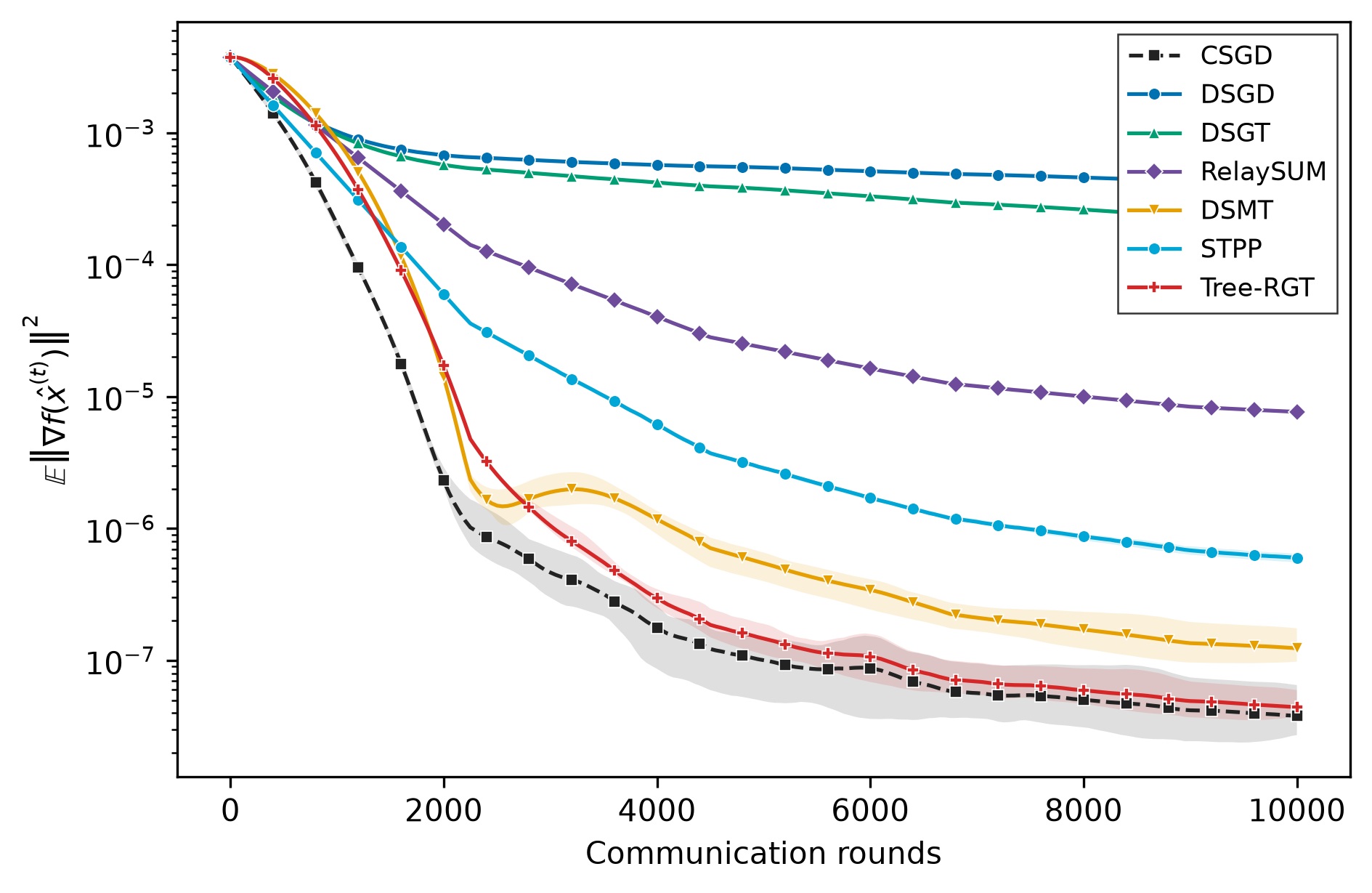}
    }%
    \hfill
    \subfloat[Double-star, $n=256$.]{
        \includegraphics[width=0.45\textwidth]{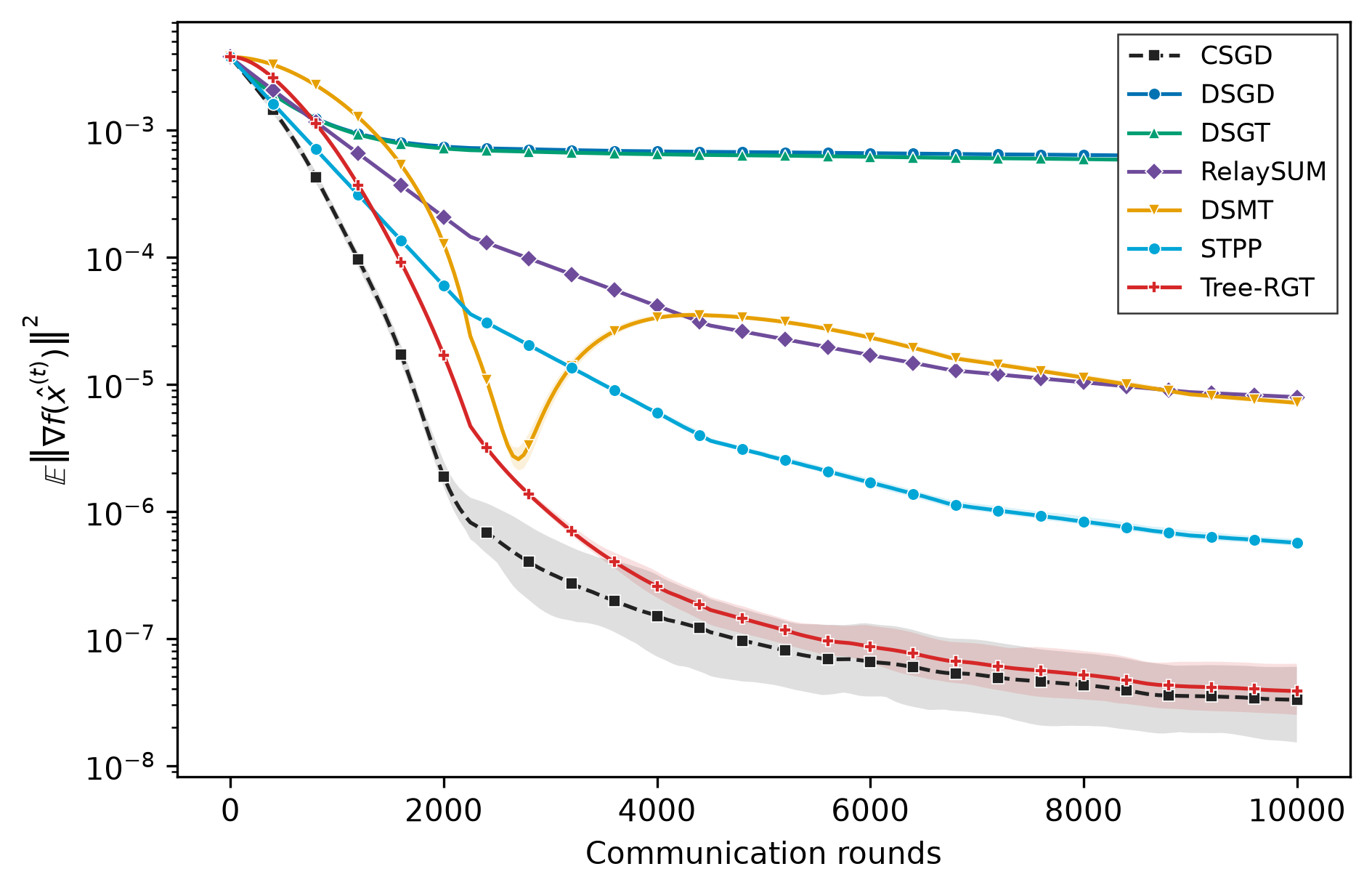}
    }
    \caption{Squared norm of the full empirical gradient versus communication
    rounds for the lazy ring (top), undirected exponential graph (middle), and
    balanced double-star (bottom), with $n=128$ (left) and $n=256$ (right).
    Curves are means over five independent runs, and shaded bands show
    the pointwise minimum--maximum envelopes. Both Tree-RGT and the centralized benchmark
    CSGD use normalized momentum with coefficient $0.997$.}
    \label{fig:ncvx-logistic}
\end{figure}

Across all six panels in Figure~\ref{fig:ncvx-logistic}, Tree-RGT attains
the lowest final mean squared gradient norm among distributed methods.
It closely follows CSGD on the exponential graph and double-star; the
larger routing delay on the lazy ring with $n=256$ produces a longer
transient and a more visible gap to CSGD. Tree-RGT consistently improves
on STPP in the later rounds, despite their shared routing tree.
Over rounds $8000$--$10{,}000$, its squared gradient norm, averaged over
five runs, is lower than STPP's by factors of $3.3$--$15.6$.

The network comparison highlights the distinction between routing
distance and matrix mixing. On the lazy ring, DSGD and DSGT retain large
residuals, and RelaySUM converges slowly. On the exponential graph, DSGD
and DSGT approach STPP. On the double-star, however, they retain much
larger residuals despite $D_{\cG}=3$, consistent with the Metropolis
matrix's mixing time $\tau_{\bW}=\Theta(n^2)$. DSMT exhibits a pronounced
nonmonotone transient, especially for $n=256$, whereas Tree-RGT stays
close to CSGD. This behavior is consistent with the analysis in
Section~\ref{sec:convergence-analysis}, which controls routing errors
through the tree-depth profile.

\FloatBarrier

\section{Conclusion}
\label{sec:conclusion}

We separated graph-invariant quantities from properties of a chosen mixing
matrix by revisiting a classical one-sided diameter--inverse-gap
constraint and illustrating how communication weights can change the inverse
spectral gap without changing the graph. This makes the network class
and the admissible algorithm class part of any precise optimality statement.
Using the diameter-based minimax benchmark, we developed Tree-Routed Gradient
Tracking (Tree-RGT), a distributed method that pipelines
model dissemination and gradient aggregation on a shortest-path tree.
The root alone performs descent and maintains momentum, incorporating
historical and delayed gradients without additional communication.
Each update uses one round of one-hop communication without repeated
neighbor-averaging steps. For smooth stochastic nonconvex optimization,
Tree-RGT achieves uniformly optimal network dependence over
$\mathsf{G}_{n,\hat D}$.
With problem and initialization quantities fixed, its time-averaged
expected squared gradient norm over $T$ iterations attains the centralized
$\cO((nT)^{-1/2})$ rate after $\cO(nD_{\cG}^{2})$ transient iterations.

A general fixed-mixing minimax theory remains open. It requires informative,
nonempty graph--matrix pair constraints and an algorithm class that specifies
how the designated matrix may be used. Extending topology-aware routing to
broader network models and studying the trade-offs among tree depth, setup
cost, and communication load are natural directions for future work.

\appendix

\section{Proofs for Section~\ref{sec:spec}}
\label{app:spectral-gap-analysis}

\subsection{Proof of Theorem \ref{thm:matrix-dependent-gap}}
\label{pf:thm:matrix-dependent-gap}

    Since $\bI$ and $\bW$ are both nonnegative, symmetric, positive semidefinite, and doubly stochastic, their convex combination $\bW^{(\delta)}$ has the same four properties.
    Let the eigenvalues of $\bW$ be
    \[
        1=\mu_1>\mu_2\ge\cdots\ge\mu_n\ge 0,\text{ and }
        \lambda_{\bW}=\mu_2.
    \]
    The eigenvalues of $\bW^{(\delta)}$ are $1$ and
    \[
        1-\delta+\delta\mu_i,
        \qquad i=2,\ldots,n.
    \]
    They are nonnegative, so
    \[
        \lambda_{\bW^{(\delta)}}
        =
        1-\delta+\delta\mu_2
        =
        1-\delta\prt{1-\lambda_{\bW}}.
    \]
    This proves the two gap identities. For every $i\ne j$,
    \[
        \brk{\bW^{(\delta)}}_{ij}
        = 
        \delta\brk{\bW}_{ij},
    \]
    which means that $\bW^{(\delta)}$ is a mixing matrix on the same graph
    $\cG$.

\subsection{Proof of Theorem \ref{l:spectral}}
\label{pf:l:spectral}

We give the standard spectral-convergence argument in the notation of
this paper; see \cite{chung1989diameters} and
\cite[Section~1.5]{chung1997spectral} for the underlying spectral estimates.

If $D_{\cG}=1$, the claim is immediate. Suppose $D_{\cG}\ge2$.
Since $\bW$ is doubly stochastic,
$\bW\bJ=\bJ\bW=\bJ$, and hence
\[
    \bW^k-\bJ=\prt{\bW-\bJ}^k,
    \qquad
    \norm{\bW^k-\bJ}_2\le\lambda_{\bW}^k.
\]
Choose two nodes $i$ and $j$ at distance $D_{\cG}$ in $\cG$. Since
$\cG$ is the support graph of $\bW$, no walk of length $D_{\cG}-1$
connects $j$ to $i$, so
\[
    \brk{\bW^{D_{\cG}-1}}_{ij}=0.
\]
Since every matrix entry is bounded in magnitude by the spectral norm,
\[
    \lambda_{\bW}^{D_{\cG}-1}
    \ge
    \norm{\bW^{D_{\cG}-1}-\bJ}_2
    \ge
    \left|\brk{\bW^{D_{\cG}-1}-\bJ}_{ij}\right|
    =\frac1n.
\]
Thus $\lambda_{\bW}>0$, and taking logarithms gives
\[
    \prt{D_{\cG}-1}\prt{-\log\lambda_{\bW}}\le\log n.
\]
Finally, $1-\lambda_{\bW}\le-\log\lambda_{\bW}$ for
$\lambda_{\bW}\in(0,1)$, which yields
\[
    1-\lambda_{\bW}\le\frac{\log n}{D_{\cG}-1}.
\]
Taking reciprocals proves the result.

\subsection{Proof of Lemma \ref{l:expander}}
\label{pf:l:expander}

By \cite[Definition~2.2 and Theorem~2.4]{hoory2006expander}, the spectral
gap $d-\mu_2(\bA_{\mathrm{cde}})$ is bounded away from zero by an
independent constant, which gives the stated bound for
$\mu_2(\bW_{\mathrm{cde}})$.  Applying the ball-growth argument given in
\cite[immediately following Lemma~2.5]{hoory2006expander} to the uniform
edge expansion gives $D_{\cG_{\mathrm{cde}}}=\cO(\log n)$.  The matching
$\Omega(\log_{d-1}n)$ lower bound for fixed $d$ is stated in
\cite[immediately before Corollary~5.4]{hoory2006expander}. Since $d\ge3$ is fixed,
$\log_{d-1}n=\log n/\log(d-1)=\Theta(\log n)$, so the upper and lower bounds
together give $D_{\cG_{\mathrm{cde}}}=\Theta(\log n)$.

\subsection{Proof of Lemma \ref{l:ring}}
\label{pf:l:ring}
The matrix $\bW_{\mathrm{r}}$ is symmetric and row-stochastic by construction, hence it is doubly stochastic. Let $\omega:=e^{2\pi i/n}$ and define $v_k\in\compls^n$ by
\[
    \brk{v_k}_j=\omega^{kj},
    \qquad
    k=0,\ldots,n-1.
\]
Then
\[
\begin{aligned}
\brk{\bW_{\mathrm{r}}v_k}_j
&=
\frac{1}{2}\omega^{kj}
+
\frac{1}{4}\omega^{k(j-1)}
+
\frac{1}{4}\omega^{k(j+1)}  \\
&=
\prt{
\frac{1}{2}
+
\frac{\omega^{-k}+\omega^k}{4}
}\omega^{kj}
=
\frac{1+\cos\prt{2\pi k/n}}{2}\brk{v_k}_j
=
\cos^2\prt{\frac{\pi k}{n}}\brk{v_k}_j.
\end{aligned}
\]
Thus, the eigenvalues of $\bW_{\mathrm{r}}$ are $\cos^2(\pi k/n)$, $k=0,\ldots,n-1$ and hence $\bW_{\mathrm{r}}$ is positive semidefinite. The eigenvalue associated with $k=0$ is $1$, and on the orthogonal complement of $\mone$ the largest eigenvalue in absolute value is $\cos^2(\pi/n)$. Since $\bW_{\mathrm{r}}$ is symmetric,
\[
    \lambda_{\bW_{\mathrm{r}}}
    =
    \norm{\bW_{\mathrm{r}}-\bJ}_2
    =
    \cos^2\prt{\frac{\pi}{n}}.
\]
Therefore
\[
    \tau_{\bW_{\mathrm{r}}}
    =
    \frac{1}{1-\cos^2\prt{\pi/n}}
    =
    \frac{1}{\sin^2\prt{\pi/n}}
    =
    \Theta(n^2),
\]
where the last equality follows from $\sin(\pi/n)=\Theta(1/n)$ derived by Taylor expansion.

\subsection{Proof of Lemma \ref{l:exponential}}
\label{pf:l:exponential}
The spectral analysis below is inspired by the study of directed
exponential graphs in \cite{ying2021exponential}.

For every integer $q$, the matrix $\bP^q$ is a permutation matrix and
$(\bP^q)^\T=\bP^{-q}$.  Hence each summand in the definition of
$\bW_{\mathrm{ex}}$ is nonnegative and symmetric, with every row and column
sum equal to $2$.  It follows that $\bW_{\mathrm{ex}}$ is nonnegative,
symmetric, and doubly stochastic. We label the nodes by
$\mathbb{Z}_{2^m}:=\crk{0,1,\ldots,2^m-1}$, where arithmetic is understood
modulo $2^m$. The support graph then contains an edge between nodes $i$ and
$j$ whenever
\[
    j-i\equiv\pm2^r\pmod{2^m}
\]
for some $r=0,\ldots,m-1$.

We first compute the diameter. For $h\ge1$ and
$d\in\mathbb{Z}_{2^h}$, define
\[
    \ell_h(d)
    :=
    \min\left\{
        t\in\mathbb{Z}_{\ge0}:
        \begin{aligned}
            &d\equiv\sum_{q=1}^{t}\sigma_q2^{r_q}\pmod{2^h},\\
            &\sigma_q\in\{-1,1\},\
            r_q\in\crk{0,\ldots,h-1}
        \end{aligned}
    \right\}.
\]
Each term in the sum corresponds to one edge traversal. Hence
$\ell_m(d)$ is the distance from node $0$ to node $d$, and translation
invariance gives
\[
    D_{\cG_{\mathrm{ex}}}
    =
    \max_{d\in\mathbb{Z}_{2^m}}\ell_m(d).
\]

For $h\ge2$, pairing all terms of exponent zero in a minimum representation
gives the recurrence
\[
    \ell_h(d)
    =
    \begin{cases}
        \ell_{h-1}(d/2), & d\ \text{even},\\
        1+\displaystyle\min_{\sigma\in\{-1,1\}}
        \ell_{h-1}\prt{(d-\sigma)/2}, & d\ \text{odd},
    \end{cases}
\]
where the arguments on the right are taken modulo $2^{h-1}$.  Indeed, when
$d$ is even, all exponent-zero terms can be canceled in opposite-sign pairs
or combined in equal-sign pairs; when $d$ is odd, exactly one term
$\sigma\in\{-1,1\}$ remains.  Dividing the other terms by $2$ proves one
direction of the recurrence, and multiplying by $2$ and restoring $\sigma$
proves the converse.

Let $L_h:=\max_d\ell_h(d)$.  For $h\ge3$ and odd $d$, choose $\sigma$ so
that $d-\sigma$ is divisible by $4$.  The recurrence then gives
\[
    L_h\le\max\{L_{h-1},1+L_{h-2}\}.
\]
Since $L_1=L_2=1$, induction yields $L_h\le\ceil{h/2}$.  For the reverse
inequality, let $a_h\in\mathbb{Z}_{2^h}$ be defined, for $h\ge1$, by
\[
    a_h
    \equiv
    \sum_{j=0}^{\ceil{h/2}-1}2^{2j}
    \pmod{2^h}.
\]
For $h\ge3$, the identities $(a_h-1)/2\equiv2a_{h-2}$ and
$(a_h+1)/2\equiv-a_{h-1}$ modulo $2^{h-1}$, together with
$\ell_h(-d)=\ell_h(d)$, give
\[
    \ell_h(a_h)
    =
    1+\min\{\ell_{h-2}(a_{h-2}),\ell_{h-1}(a_{h-1})\}
    =
    \ceil{\frac{h}{2}}
\]
by induction from $h=1,2$.  Therefore
\[
    D_{\cG_{\mathrm{ex}}}
    =
    L_m
    =
    \ceil{\frac{m}{2}}
    =
    \Theta(\log n).
\]

It remains to estimate $\lambda_{\bW_{\mathrm{ex}}}$.  Let
$\omega:=e^{2\pi i/n}$ and define the Fourier vectors by
$\brk{v_k}_j:=n^{-1/2}\omega^{-kj}$.  Since
$\bP^qv_k=\omega^{kq}v_k$, they diagonalize $\bW_{\mathrm{ex}}$, with
\[
    \theta_k
    =
    \frac{1}{2m}\sum_{r=0}^{m-1}
    \prt{\omega^{k2^r}+\omega^{-k2^r}}
    =
    \frac{1}{m}\sum_{r=0}^{m-1}
    \cos\prt{\frac{2\pi k2^r}{n}}.
\]
Because the $v_k$ form an orthonormal basis and $\theta_0=1$, we have
$\lambda_{\bW_{\mathrm{ex}}}=\max_{1\le k<n}|\theta_k|$.

Set $\eta_k:=m^{-1}\sum_{r=0}^{m-1}\omega^{k2^r}$, so that
$|\theta_k|\le|\eta_k|$.  Write $k=2^su$, where $u$ is odd and $0\le s\le m-1$. If $s\ge1$,
consider the two summands in $\eta_k$ corresponding to
$r=m-s-1$ and $r=m-s$. They satisfy
\[
    \omega^{k2^{m-s-1}}=e^{\pi i u}=-1,
    \qquad
    \omega^{k2^{m-s}}=e^{2\pi i u}=1,
\]
and hence cancel.  If $s=0$,
the terms indexed by $r=m-1,m-2$ are $-1$ and $\pm i$, whose sum has
modulus $\sqrt{2}$.  The triangle inequality therefore gives
\[
    |\eta_k|
    \le
    \begin{cases}
        (m-2)/m, & s\ge1\\
        (m-2+\sqrt{2})/m, & s=0
    \end{cases}
    \le
    1-\frac{2-\sqrt{2}}{m}.
\]
For $k=2^{m-1}$, the term with $r=0$ equals $-1$ and all the others equal
$1$, so $\theta_k=1-2/m$.  Consequently,
\[
    1-\frac{2}{m}
    \le
    \lambda_{\bW_{\mathrm{ex}}}
    \le
    1-\frac{2-\sqrt{2}}{m},
    \qquad
    \frac{m}{2}
    \le
    \tau_{\bW_{\mathrm{ex}}}
    \le
    \frac{m}{2-\sqrt{2}}.
\]
Therefore
\[
    \tau_{\bW_{\mathrm{ex}}}
    =
    \Theta(m)
    =
    \Theta(\log n)
    =
    \Theta(D_{\cG_{\mathrm{ex}}}).
\]

\subsection{Proof of Lemma \ref{l:double_star}}
\label{pf:l:double_star}
The matrix $\bW_{\mathrm{ds}}$ is symmetric by construction. Each leaf has self-weight $b=(m+1)/(m+2)$ and one incident edge of weight $a=1/(m+2)$. Each center has self-weight $a$ and is incident to $m$ leaves and the other center, with all these $m+1$ graph edges having weight $a$. Hence every row sums to one, and symmetry implies that $\bW_{\mathrm{ds}}$ is doubly stochastic. The largest distance is attained between a leaf in $\cL_1$ and a leaf in $\cL_2$, so $D_{\cG_{\mathrm{ds}}}=3$.

It remains to compute $\lambda_{\bW_{\mathrm{ds}}}$. 
We first identify the eigenvectors associated with leaf differences. Define
\[
\begin{aligned}
\operatorname{supp}(\bz) & := \crk{i\in\cN: \brk{\bz}_i \ne 0}, \\
\cU_1
&:=
\crk{
\bz\in\reals^n:
\operatorname{supp}(\bz)\subseteq \cL_1,\ 
\sum_{\ell\in\cL_1} z_\ell=0
},\\
\cU_2
&:=
\crk{
\bz\in\reals^n:
\operatorname{supp}(\bz)\subseteq \cL_2,\ 
\sum_{\ell\in\cL_2} z_\ell=0
}.
\end{aligned}
\]
Both subspaces have dimension $m-1$. For any $\bz\in\cU_1$, each leaf $\ell\in\cL_1$ satisfies
\[
    \brk{\bW_{\mathrm{ds}}\bz}_\ell
    =
    bz_\ell + a z_{c_1}
    =
    bz_\ell,
\]
and the center coordinate satisfies
\[
    \brk{\bW_{\mathrm{ds}}\bz}_{c_1}
    =
    a\sum_{\ell\in\cL_1} z_\ell + a z_{c_2}
    =
    0.
\]
All remaining coordinates are also zero. Hence, $\bW_{\mathrm{ds}}\bz=b\bz$ for every $\bz\in\cU_1$. The same argument applies to $\cU_2$. Therefore, $\bW_{\mathrm{ds}}$ has the eigenvalue $b$ with multiplicity at least $2(m-1)$, carried by the disjoint subspaces $\cU_1$ and $\cU_2$.

We next compute all eigenvalues different from $b$. Let $(\lambda,\bz)$ be an eigenpair with $\lambda\neq b$. For any two leaves $\ell,\ell'\in\cL_1$, the leaf equations are
\[
    \lambda \brk{\bz}_\ell = b\brk{\bz}_\ell + a\brk{\bz}_{c_1},
    \qquad
    \lambda \brk{\bz}_{\ell'} = b\brk{\bz}_{\ell'} + a\brk{\bz}_{c_1}.
\]
Subtracting them gives
\[
    (\lambda-b)\prt{\brk{\bz}_\ell-\brk{\bz}_{\ell'}}=0.
\]
Since $\lambda\neq b$, all entries on $\cL_1$ are equal. The same argument shows that all entries on $\cL_2$ are equal. Hence every eigenvector associated with an eigenvalue $\lambda\neq b$ can be described by four scalars $x_1,y_1,y_2,x_2$, corresponding respectively to the common value on $\cL_1$, the value at $c_1$, the value at $c_2$, and the common value on $\cL_2$. The eigenvalue equations reduce to
\[
\begin{aligned}
\lambda x_1 &= b x_1 + a y_1, \\
\lambda y_1 &= am x_1 + a y_1 + a y_2, \\
\lambda y_2 &= a y_1 + a y_2 + am x_2, \\
\lambda x_2 &= b x_2 + a y_2 .
\end{aligned}
\]
We now solve this four-variable system by a change of variables. Define the sum variables and difference variables
\[
    s_x:=x_1+x_2,\qquad s_y:=y_1+y_2,
    \qquad
    d_x:=x_1-x_2,\qquad d_y:=y_1-y_2.
\]
This change of variables is invertible, since
\[
    x_1=\frac{s_x+d_x}{2},\quad
    x_2=\frac{s_x-d_x}{2},\quad
    y_1=\frac{s_y+d_y}{2},\quad
    y_2=\frac{s_y-d_y}{2}.
\]
Thus no eigenvector is lost by working with the variables $(s_x,s_y,d_x,d_y)$.

Adding the first and fourth equations, and adding the second and third equations, gives
\[
    \lambda s_x = bs_x + as_y,
    \qquad
    \lambda s_y = ams_x + 2as_y.
\]
Therefore, the sum variables satisfy the $2\times 2$ eigenvalue system
\[
    \lambda
    \begin{bmatrix}
    s_x\\ s_y
    \end{bmatrix}
    =
    \begin{bmatrix}
    b & a \\
    am & 2a
    \end{bmatrix}
    \begin{bmatrix}
    s_x\\ s_y
    \end{bmatrix}.
\]
The corresponding characteristic polynomial is
\[
    \lambda^2-(b+2a)\lambda+(2ab-a^2m)=0.
\]
Using $a=1/(m+2)$ and $b=(m+1)/(m+2)$, the two roots are $1$ and $a$.

Similarly, subtracting the fourth equation from the first, and subtracting the third equation from the second, gives
\[
    \lambda d_x = bd_x + ad_y,
    \qquad
    \lambda d_y = amd_x.
\]
Therefore, the difference variables satisfy
\[
    \lambda
    \begin{bmatrix}
    d_x\\ d_y
    \end{bmatrix}
    =
    \begin{bmatrix}
    b & a \\
    am & 0
    \end{bmatrix}
    \begin{bmatrix}
    d_x\\ d_y
    \end{bmatrix}.
\]
The corresponding characteristic polynomial is
\[
    \lambda^2-b\lambda-a^2m=0,
\]
and therefore its two eigenvalues are
\[
    \lambda_{+} = \frac{m+1+\sqrt{m^2+6m+1}}{2(m+2)},\text{ and } 
    \lambda_{-}
    =
    \frac{m+1-\sqrt{m^2+6m+1}}{2(m+2)}.
\]
The spectrum of $\bW_{\mathrm{ds}}$ therefore consists of the eigenvalue $b$, with multiplicity $2(m-1)$, together with the four eigenvalues $1$, $a$, $\lambda_{+}$, and $\lambda_{-}$. Since $1>\lambda_+>b>a$ and $\lambda_+\ge |\lambda_-|$, the largest eigenvalue in modulus on the orthogonal complement of $\mone$ is
\[
    \lambda_{\bW_{\mathrm{ds}}}
    =
    \lambda_+
    =
    \frac{m+1+\sqrt{m^2+6m+1}}{2(m+2)}.
\]
Finally,
\[
\begin{aligned}
\tau_{\bW_{\mathrm{ds}}}
&=
\frac{1}{1-\lambda_{\bW_{\mathrm{ds}}}}
=
\frac{(m+2)\prt{m+3+\sqrt{m^2+6m+1}}}{4}
=
\Theta(m^2)
=
\Theta(n^2).
\end{aligned}
\]

We now analyze the second weighting $\widetilde{\bW}_{\mathrm{ds}}$.
Each leaf row sums to $\beta+\alpha=1$, while each center row sums to
\[
    m\alpha+\frac14+\frac14=1.
\]
Thus $\widetilde{\bW}_{\mathrm{ds}}$ is symmetric and doubly stochastic, and all of its graph edges have positive weight. The two leaf-difference subspaces $\cU_1$ and $\cU_2$ considered above now carry the eigenvalue $\beta$, with total multiplicity $2(m-1)$.

On the remaining four-dimensional subspace, the same sum--difference decomposition gives
\[
    \lambda
    \begin{bmatrix}
        s_x\\s_y
    \end{bmatrix}
    =
    \begin{bmatrix}
        \beta & \alpha\\
        m\alpha & \frac12
    \end{bmatrix}
    \begin{bmatrix}
        s_x\\s_y
    \end{bmatrix},
    \qquad
    \lambda
    \begin{bmatrix}
        d_x\\d_y
    \end{bmatrix}
    =
    \begin{bmatrix}
        \beta & \alpha\\
        m\alpha & 0
    \end{bmatrix}
    \begin{bmatrix}
        d_x\\d_y
    \end{bmatrix}.
\]
The sum matrix has eigenvalues
\[
    1,
    \qquad
    \frac{m-1}{2m},
\]
while the difference matrix has eigenvalues
\[
    \rho_{\pm}
    =
    \frac{1-\frac{1}{2m}\pm\sqrt{1+\frac{1}{4m^2}}}{2}.
\]
Since $\rho_+>\beta$, $\rho_+>-\rho_-$, and $\rho_+>(m-1)/(2m)$, the largest eigenvalue magnitude on the orthogonal complement of $\mone$ is
\[
    \lambda_{\widetilde{\bW}_{\mathrm{ds}}}
    =
    \rho_+
    =
    \frac{1-\frac{1}{2m}+\sqrt{1+\frac{1}{4m^2}}}{2}.
\]
Rationalizing its gap yields
\[
\begin{aligned}
    \tau_{\widetilde{\bW}_{\mathrm{ds}}}
    &=
    \frac{1}{1-\rho_+}\\
    &=
    2m+1+\sqrt{4m^2+1}\\
    &=
    \Theta(m)
    =
    \Theta(n).
\end{aligned}
\]

\subsection{Proof of Theorem \ref{thm:arbitrary-diameter-gap-scaling}}
\label{pf:thm:arbitrary-diameter-gap-scaling}

Fix $p>0$ and $D_0\ge1$.  Let
$\{\cG_{\mathrm{cde},k}\}_{k\ge1}$ be an infinite constant-degree expander
family as in Lemma \ref{l:expander}, where the common degree $d\ge3$ is
fixed.  Let $n_k$ be the number of nodes in $\cG_{\mathrm{cde},k}$, indexed
so that $n_k\to\infty$, and let $\bA_{\mathrm{cde},k}$ be its adjacency
matrix.  Define
\[
    \bW_{\mathrm{cde},k}
    :=
    \frac{\bA_{\mathrm{cde},k}}{d},
    \qquad
    \widehat{\bW}_{\mathrm{cde},k}
    :=
    \frac12\prt{\bI+\bW_{\mathrm{cde},k}},
    \qquad
    D_k
    :=
    D_{\cG_{\mathrm{cde},k}}.
\]
All eigenvalues of $\bW_{\mathrm{cde},k}$ lie in $[-1,1]$, so
$\widehat{\bW}_{\mathrm{cde},k}$ is positive semidefinite.  It is also a
mixing matrix on $\cG_{\mathrm{cde},k}$.  Lemma \ref{l:expander} gives a
constant $c>0$, independent of $k$, such that
\[
    1
    \le
    \tau_{\widehat{\bW}_{\mathrm{cde},k}}
    =
    \frac{2}{1-\mu_2(\bW_{\mathrm{cde},k})}
    \le
    \frac{2}{c},
    \qquad
    D_k
    =
    \Theta(\log n_k)
    \to
    \infty.
\]

Choose $k$ sufficiently large that
$D_k\ge D_0$ and $D_k^p\ge 2/c$.  For this fixed $k$, define
\[
    \cG^*
    :=
    \cG_{\mathrm{cde},k},
    \qquad
    \delta
    :=
    \frac{\tau_{\widehat{\bW}_{\mathrm{cde},k}}}{D_k^p},
    \qquad
    \bW^*
    :=
    (1-\delta)\bI
    +
    \delta\widehat{\bW}_{\mathrm{cde},k}.
\]
The preceding bounds imply $\delta\in(0,1]$.  Hence Theorem
\ref{thm:matrix-dependent-gap} shows that $\bW^*$ is a mixing matrix on
$\cG^*$ and
\[
    \tau_{\bW^*}
    =
    \frac{\tau_{\widehat{\bW}_{\mathrm{cde},k}}}{\delta}
    =
    D_k^p
    =
    D_{\cG^*}^p.
\]
By construction, $D_{\cG^*}=D_k\ge D_0$, which completes the proof.

\section{Proofs for Section~\ref{sec:lower-bounds}}
\label{app:lower-complexity-proofs}

\subsection{Proof of Theorem~\ref{thm:weighted-path-inverse-gap}}
\label{pf:thm:weighted-path-inverse-gap}

Let $\cG_{\mathrm{path}}^{(n)}:=\operatorname{Path}_n$.  The construction
underlying \cite[Corollary~1]{lu2021optimal} pairs this graph with the
reflecting random-walk matrix whose nonzero entries are
\[
    \brk{\bW_{\mathrm{rw}}^{(n)}}_{i,i+1}
    =
    \brk{\bW_{\mathrm{rw}}^{(n)}}_{i+1,i}
    =
    \frac12
    \quad (1\le i<n),
    \qquad
    \brk{\bW_{\mathrm{rw}}^{(n)}}_{11}
    =
    \brk{\bW_{\mathrm{rw}}^{(n)}}_{nn}
    =
    \frac12.
\]
It is a mixing matrix on $\cG_{\mathrm{path}}^{(n)}$.  For the inverse-gap
budget in Theorem~\ref{thm:weighted-path-inverse-gap}, take
$\tau:=\tau_{\bW_{\mathrm{rw}}^{(n)}}$.  The standard cosine diagonalization
gives
\[
    D_{\cG_{\mathrm{path}}^{(n)}}=n-1,
    \qquad
    \lambda_{\bW_{\mathrm{rw}}^{(n)}}=\cos\prt{\frac{\pi}{n}},
    \qquad
    \tau
    =
    \frac{1}{1-\cos(\pi/n)}
    =
    \Theta\prt{(n-1)^2}.
\]

Every connected $n$-node graph of diameter $n-1$ is a relabeling of
$\operatorname{Path}_n$: a diametral path contains all $n$ nodes and cannot
have a chord.  Relabeling $\bW_{\mathrm{rw}}^{(n)}$ in the same way shows that
the graph components represented in $\widehat{\mathsf{P}}_n$ form exactly
$\mathsf{G}_{n,n-1}$.

Neither $\mathsf{A}_B$ nor $T_\varepsilon(A,f,O,\cG)$ depends on the matrix
paired with $\cG$.  Therefore
\[
    \inf_{A\in\mathsf{A}_B}
    \sup_{\substack{
        (\cG,\bW)\in\widehat{\mathsf{P}}_n,\\
        (\{f_i\},O)\in\mathsf{I}_n(\Delta,L,\sigma^2)
    }}
    T_\varepsilon(A,f,O,\cG)
    =
    \mathsf{T}_\varepsilon(\mathsf{G}_{n,n-1},B).
\]
More explicitly, apply \cite[Theorem~1]{lu2021optimal} with $\hat D=n-1$,
and let $\widehat{\cG}_{\mathrm{hard}}$ and the hard optimization--oracle
instance be those supplied by its construction.  Relabeling
$\bW_{\mathrm{rw}}^{(n)}$ in the same way as the path gives a matrix
$\bW_{\mathrm{hard}}$ such that
$(\widehat{\cG}_{\mathrm{hard}},\bW_{\mathrm{hard}})
\allowbreak\in\widehat{\mathsf{P}}_n$ and the same pair and instance are hard for every
$A\in\mathsf{A}_B$.  Finally, $n-1=\Theta(\sqrt{\tau})$ proves
\eqref{eq:lu-path-class-lower-bound}.

\subsection{Proof of Corollary~\ref{cor:reweighted-path-inverse-gap}}
\label{pf:cor:reweighted-path-inverse-gap}

Use the hard graph, matrix, and optimization--oracle instance constructed
above, and define
\[
    \bW_{\mathrm{harder}}
    :=
    \prt{1-\frac1n}\bI+\frac1n\bW_{\mathrm{hard}}.
\]
As a convex combination of $\bI$ and $\bW_{\mathrm{hard}}$, this matrix is
nonnegative, symmetric, and doubly stochastic.  It also preserves every
positive off-diagonal entry of $\bW_{\mathrm{hard}}$, and hence is a mixing
matrix on the same graph.
Because $\bW_{\mathrm{hard}}$ is permutation-similar to
$\bW_{\mathrm{rw}}^{(n)}$, its eigenvalues are
$\nu_k=\cos(k\pi/n)$ for $k=0,\ldots,n-1$.  The eigenvalues of
$\bW_{\mathrm{harder}}$ are therefore
\[
    \nu_k'
    =
    1-\frac{1-\cos(k\pi/n)}{n},
    \qquad k=0,\ldots,n-1.
\]
They are all positive for $n\ge2$, so the second-largest eigenvalue modulus is
$\nu_1'$.  Consequently,
\[
    1-\lambda_{\bW_{\mathrm{harder}}}
    =
    \frac{1-\cos(\pi/n)}{n},
    \qquad
    \tau_{\bW_{\mathrm{harder}}}
    =
    n\tau
    =
    \bar\tau
    =
    \Theta\prt{(n-1)^3}.
\]
Thus
$(\widehat{\cG}_{\mathrm{hard}},\bW_{\mathrm{harder}})
\in\overline{\mathsf{P}}_n$.  Since neither $\mathsf{A}_B$ nor
$T_\varepsilon(A,f,O,\cG)$ uses the designated matrix, replacing
$\bW_{\mathrm{hard}}$ by $\bW_{\mathrm{harder}}$ leaves the hard graph and
optimization--oracle instance unchanged.  Finally,
$n-1=\Theta(\bar\tau^{1/3})$ converts the diameter term in
\cite[Theorem~1]{lu2021optimal} into the second term of
\eqref{eq:lu-reweighted-path-class-lower-bound}.

\section{Proofs for Section~\ref{sec:tree-rgt}}
\label{app:tree-routing-matrix-proofs}

\subsection{Proof of Theorem \ref{thm:bfs-tree-optimal}}
\label{pf:thm:bfs-tree-optimal}

The breadth-first search discovers the root at level zero. Suppose that all
nodes at graph distance at most $\ell$ from the root have been discovered
with their correct depths. Every node at graph distance $\ell+1$ has a
neighbor at graph distance $\ell$, and hence receives a discovery message
at level $\ell$. Conversely, any node first reached at that level has a path
of length $\ell+1$ from the root and cannot have a shorter path, since it
was not discovered earlier. Therefore, every newly discovered node $i$
satisfies $h_i=\operatorname{dist}_{\cG}(1,i)$. Induction proves
\eqref{eq:bfs-pointwise-depth}.

For any other spanning tree $\cT'$ of $\cG$ rooted at node~$1$, its unique
root-to-$i$ path is also a path in $\cG$. Hence
\[
    \operatorname{dist}_{\cT'}(1,i)
    \ge
    \operatorname{dist}_{\cG}(1,i)
    =
    h_i
    \qquad\text{for every }i\in\cN.
\]
Taking the maximum over $i$ proves the minimization identity in
\eqref{eq:bfs-optimal-D}, and averaging over $i$ proves
\eqref{eq:bfs-optimal-Davg}. Since $D_{\cG}$ maximizes the distance over
all pairs of nodes, whereas $D$ maximizes only over pairs involving the
root, we have $D\le D_{\cG}$. Conversely, for any $i,j\in\cN$, the triangle
inequality and \eqref{eq:bfs-pointwise-depth} give
\[
    \operatorname{dist}_{\cG}(i,j)
    \le
    \operatorname{dist}_{\cG}(i,1)
    +
    \operatorname{dist}_{\cG}(1,j)
    =h_i+h_j
    \le 2D.
\]
Maximizing over $i,j$ yields $D_{\cG}\le2D$ and completes
\eqref{eq:bfs-optimal-D}.

\subsection{Proof of Lemma \ref{lem:Rk_norm}}
\label{pf:lem:Rk_norm}
We first prove the closed forms in \eqref{eq:RC-power-entrywise}. The identity
for $\bR^0=\bI$ is immediate. If it holds at $k$, then the unique
parent entry in row $i$ of $\bR$ gives
\[
    \brk{\bR^{k+1}}_{ij}
    =
    \sum_{\ell\in\cN}\brk{\bR}_{i\ell}\brk{\bR^k}_{\ell j}
    =
    \brk{\bR^k}_{p(i),j}
    =
    \mathbbm{1}_{\{p^{k+1}(i)=j\}}.
\]
Induction proves the first identity, and the second follows from
$\bC^k=(\bR^k)^\T$.

Since $p(1)=1$, a node reaches the root after repeated application of the
parent map precisely when $k\ge h_i$; that is,
$p^k(i)=1$ if and only if $h_i\le k$. Let $\bfe_j$ denote the $j$-th standard
basis vector. Equation~\eqref{eq:RC-power-entrywise} then gives that the $i$-th row of $\bR^k-\mone\bpi^\T$ satisfies:
\[
    \brk{\bR^k-\mone\bpi^\T}_{i,:}
    =
    \begin{cases}
        \mathbf{0}^\T, & h_i\le k,\\
        \bfe_{p^k(i)}^\T-\bpi^\T, & h_i>k.
    \end{cases}
\]
Each nonzero row has squared Euclidean norm $2$, and there are exactly
$n-n_k$ such rows. Hence
\begin{equation}
    \label{eq:Rk-residual-frobenius}
    \norm{\bR^k-\mone\bpi^\T}_F^2
    =
    2\prt{n-n_k}.
\end{equation}
Finally,
$\bC^k-\bpi\mone^\T=(\bR^k-\mone\bpi^\T)^\T$, so the corresponding spectral
and Frobenius norms are equal. The inequality
$\norm{\cdot}_2\le\norm{\cdot}_F$ completes the proof.

\subsection{Proof of Corollary \ref{lem:Rk_d}}
\label{pf:lem:Rk_d}

Since $D=\max_{i\in\cN}h_i$, we have $n_k=n$ if and only if $k\ge D$.
The result follows directly from Lemma~\ref{lem:Rk_norm}.

\subsection{Proof of Theorem \ref{thm:convergence}}
\label{pf:thm:convergence}

The parameter choices satisfy $0\le\beta\le1-1/\sqrt{T}$ and
$\gamma\le(1-\beta)/(60nL)$, so Lemma~\ref{lem:smooth} applies.
Using $\norm{\nabla\bF^{(0)}}_F^2=nH_0$ gives
\[
\begin{aligned}
    &\frac{1}{T}\sum_{t=0}^{T-1}
    \expect\norm{\nabla f(x_1^{(t)})}^2\le\frac{20\Delta_f}{\gamma nT}
    +220\gamma L\sigma^2+\frac{60D_{\avg}H_0}{T}
    +\frac{60H_0}{(1-\beta)T}.
\end{aligned}
\]
The stepsize and momentum choices imply
\[
\begin{aligned}
    \frac{1}{\gamma}
    &\le60nDL+
    L\sqrt{\frac{nT\sigma^2}{\Delta_f L+H_0}},
    \gamma L\sigma^2
    \le\frac{\sqrt{(\Delta_f L+H_0)\sigma^2}}{\sqrt{nT}},\\
    \frac{1}{1-\beta}
    &\le D+\frac{1}{60}
    \sqrt{\frac{T\sigma^2}{n(\Delta_f L+H_0)}}.
\end{aligned}
\]
Substituting these bounds and using
$20\Delta_f L+H_0\le20(\Delta_f L+H_0)$ yields
\[
\begin{aligned}
    &\frac{1}{T}\sum_{t=0}^{T-1}
    \expect\norm{\nabla f(x_1^{(t)})}^2\le
    \frac{240\sqrt{(\Delta_f L+H_0)\sigma^2}}{\sqrt{nT}}
    +\frac{1200D(\Delta_f L+H_0)}{T},
\end{aligned}
\]
where the inequality uses $D_{\avg}\le D$. This completes the proof.

\subsection{Proof of Corollary \ref{cor:tree-rgt-accuracy}}
\label{pf:cor:tree-rgt-accuracy}

Independence of the samples within each mini-batch and
Assumption~\ref{a.var} give
\[
    \expect\brk{g_{i,B}(x,\boldsymbol{\xi}_i)}=\nabla f_i(x),
    \qquad
    \expect\norm{
        g_{i,B}(x,\boldsymbol{\xi}_i)-\nabla f_i(x)
    }^2\le\frac{\sigma^2}{B}.
\]
Applying Theorem~\ref{thm:convergence} with $\sigma^2/B$ in place of
$\sigma^2$ proves the trajectory bound, with both the stepsize and the
momentum coefficient chosen for this reduced variance. The iteration budget
in \eqref{eq:tree-rgt-epsilon-trajectory} ensures
\[
    \frac{240\sqrt{(\Delta_f L+H_0)\sigma^2}}
    {\sqrt{nBT}}\le\frac{\varepsilon^2}{2},
    \qquad
    \frac{1200D(\Delta_f L+H_0)}{T}
    \le\frac{\varepsilon^2}{2}.
\]
Adding the inequalities proves the accuracy guarantee.

\section{Proofs for Section~\ref{sec:convergence-analysis}}
\label{app:convergence-proofs}

\subsection{Auxiliary Inequalities}

We first collect four elementary tools used in the convergence proofs.
They control weighted and unweighted finite matrix sums, matrix products,
and nonnegative finite-memory convolutions over the full time horizon.

\begin{lemma}[Weighted Cauchy--Schwarz bound]
    \label{lem:weighted_sum}
    Let $\bA_1,\ldots,\bA_m$ be matrices of the same size and let
    $w_1,\ldots,w_m\ge0$. Then
    \[
        \norm{\sum_{i=1}^{m}w_i\bA_i}_F^2
        \le\prt{\sum_{i=1}^{m}w_i}
        \sum_{i=1}^{m}w_i\norm{\bA_i}_F^2.
    \]
    In particular, if $\alpha_1,\ldots,\alpha_m>0$ and
    $\sum_{i=1}^{m}\alpha_i=1$, then
    \[
        \norm{\sum_{i=1}^{m}\bA_i}_F^2
        \le\sum_{i=1}^{m}\frac{\norm{\bA_i}_F^2}{\alpha_i}.
    \]
\end{lemma}
\begin{proof}
    The triangle inequality and scalar Cauchy--Schwarz give
    \[
    \begin{aligned}
        \norm{\sum_{i=1}^{m}w_i\bA_i}_F^2
        &\le\prt{\sum_{i=1}^{m}w_i\norm{\bA_i}_F}^2 \le\prt{\sum_{i=1}^{m}w_i}
        \sum_{i=1}^{m}w_i\norm{\bA_i}_F^2.
    \end{aligned}
    \]
    For the second inequality, apply the first with weights
    $w_i=\alpha_i$ and matrices $\bA_i/\alpha_i$.
\end{proof}

Taking $w_i=1$ for $i=1,\ldots,m$ in
Lemma~\ref{lem:weighted_sum} yields the following corollary.

\begin{corollary}[Finite-sum Frobenius bound]
    \label{lem:sum_matrix}
    Let $\bA_1,\ldots,\bA_m$ be matrices of the same size. Then
    \[
        \norm{\sum_{i=1}^{m}\bA_i}_F^2
        \le m \sum_{i=1}^{m} \norm{\bA_i}_F^2.
    \]
\end{corollary}

\begin{lemma}[Matrix-product Frobenius bound]
    \label{lem:matrix_norm}
    Let $\bA$ and $\bB$ be matrices of compatible dimensions. Then
    \[
        \norm{\bA\bB}_F \le \norm{\bA}_2\norm{\bB}_F.
    \]
\end{lemma}
\begin{proof}
    By the definition of the spectral norm,
    $\norm{\bA}_2^2\bI-\bA^\T\bA$ is positive semidefinite. Therefore, its
    congruence transformation
    \[
        \norm{\bA}_2^2\bB^\T\bB
        -\bB^\T\bA^\T\bA\bB
        =\bB^\T\prt{\norm{\bA}_2^2\bI-\bA^\T\bA}\bB
    \]
    is also positive semidefinite and thus has a nonnegative trace. It follows
    that
    \[
        \norm{\bA\bB}_F^2
        =\operatorname{tr}\prt{\bB^\T\bA^\T\bA\bB}
        \le \norm{\bA}_2^2\operatorname{tr}\prt{\bB^\T\bB}
        =\norm{\bA}_2^2\norm{\bB}_F^2.
    \]
    Taking square roots proves the claim.
\end{proof}

\begin{lemma}[Finite-memory convolution bound]
    \label{lem:sum_help}
    Let $T\ge1$ and $D\ge1$ be integers, and let
    $\crk{a_t}_{t=0}^{T-1}$ and $\crk{b_i}_{i=0}^{D-1}$ be nonnegative
    sequences. Then
    \[
        \sum_{t=0}^{T-1}
        \sum_{m=\max\crk{0,t-D}}^{t-1} b_{t-m-1}a_m
        \le
        \sum_{i=0}^{D-1}b_i\sum_{t=0}^{T-1}a_t,
    \]
    where a sum over an empty index set is understood to be zero.
\end{lemma}
\begin{proof}
    When $T=1$, the inner sum on the left-hand side is empty, so the claim is
    immediate. Suppose $T\ge2$. Set $i=t-m-1$ and $s=m$, or equivalently
    $t=s+i+1$. Reordering the summation gives
    \[
    \begin{aligned}
        \sum_{t=0}^{T-1}
        \sum_{m=\max\crk{0,t-D}}^{t-1}
        b_{t-m-1}a_m
        &=\sum_{i=0}^{\min\crk{D-1,T-2}}b_i
        \sum_{s=0}^{T-i-2}a_s \le \sum_{i=0}^{D-1}b_i\sum_{s=0}^{T-1}a_s,
    \end{aligned}
    \]
    where the last inequality follows from the nonnegativity of both
    sequences.
\end{proof}

\subsection{Proof of Lemma \ref{lem:Ak}}
\label{pf:lem:Ak}
For $k=0$, $\bA_0=\bI$ and $\bfe_{1,0}=\bpi$ give
\[
    \bpi^\T\bA_0=\bpi^\T=\bfe_{1,0}^\T.
\]
For $1\le k\le D$, \eqref{eq:RC-power-entrywise} implies
\[
    \bpi^\T\bC^k
    =
    \sum_{\ell=0}^{k}\bfe_{1,\ell}^\T .
\]
Consequently,
\[
    \bpi^\T\bA_k
    =
    \bpi^\T\prt{\bC^k-\bC^{k-1}}
    =
    \bfe_{1,k}^\T .
\]
If $k>D$, Corollary \ref{lem:Rk_d} gives $\bC^k=\bC^{k-1}=\bpi\mone^\T$, and
therefore $\bA_k=\mathbf{0}$.

\subsection{Proof of Lemma \ref{lem:RC-interaction}}
\label{pf:lem:RC-interaction}
Define the routing kernel without the momentum weights by
\[
    \bU_r:=\sum_{j=\max\crk{1,r-D}}^{\min\crk{r,D}}
    \bPi_{\bR}\bR^{j-1}\bS\bA_{r-j},\qquad r\ge1.
\]
For a node $v$ at depth $h_v$, $\bS=\bpi\bpi^\T$ gives
\begin{equation}
    \label{eq:RC-interaction-P-column}
    \bS\bA_k\bfe_v
    =\bpi\brk{\bpi^\T\bA_k\bfe_v}.
\end{equation}
Lemma~\ref{lem:Ak} implies, for $k\ge0$,
\begin{equation}
    \label{eq:RC-interaction-root-scalar}
    \bpi^\T\bA_k\bfe_v
    =\begin{cases}
        1,&k=h_v,\\
        0,&k\ne h_v.
    \end{cases}
\end{equation}
Hence
\[
    \bU_r\bfe_v
    =\begin{cases}
        \bR^{r-h_v-1}\bpi-\mone,&1\le r-h_v\le D,\\
        \mathbf{0},&\text{otherwise}.
    \end{cases}
\]
By \eqref{eq:RC-power-entrywise}, each entry belongs to
$\crk{-1,0}$. Thus $\norm{\bU_r}_F^2\le n^2$, and
$\bU_r=\mathbf{0}$ for $r>2D$. Substituting
\eqref{eq:momentum-kernel} into the definition of $\bJ_q$ gives
\[
    \begin{aligned}
        \bJ_q
        &=\sum_{j=1}^{\min\crk{q,D}}
        \bPi_{\bR}\bR^{j-1}\bK_{q-j}\\
        &=(1-\beta)\sum_{j=1}^{\min\crk{q,D}}
        \sum_{k=0}^{\min\crk{q-j,D}}
        \beta^{q-j-k}\bPi_{\bR}\bR^{j-1}\bS\bA_k.
    \end{aligned}
\]
Set $r=j+k$. The constraints $1\le j\le D$, $0\le k\le D$,
and $j+k\le q$ become $1\le r\le\min\crk{q,2D}$ and
$\max\crk{1,r-D}\le j\le\min\crk{r,D}$. Since
$\beta^{q-j-k}=\beta^{q-r}$, grouping terms with the same $r$ yields
\[
    \begin{aligned}
        \bJ_q
        &=(1-\beta)\sum_{r=1}^{\min\crk{q,2D}}\beta^{q-r}
        \sum_{j=\max\crk{1,r-D}}^{\min\crk{r,D}}
        \bPi_{\bR}\bR^{j-1}\bS\bA_{r-j}\\
        &=(1-\beta)\sum_{r=1}^{\min\crk{q,2D}}
        \beta^{q-r}\bU_r,
    \end{aligned}
\]
where the inner sum is precisely $\bU_r$.
Apply Lemma~\ref{lem:weighted_sum} with matrices $\bU_r$ and weights
$(1-\beta)\beta^{q-r}$, whose sum is $\rho_q\le1$. Then
\begin{equation}
    \label{eq:RC-interaction-uniform}
    \begin{aligned}
        \norm{\bJ_q}_F^2
        &\le\rho_q(1-\beta)\sum_{r=1}^{\min\crk{q,2D}}
        \beta^{q-r}\norm{\bU_r}_F^2
        \le n^2\rho_q^2\le n^2.
    \end{aligned}
\end{equation}
Reordering the nonnegative sums yields
\[
    \sum_{q=1}^{\infty}\rho_q
    =(1-\beta)\sum_{r=1}^{2D}\sum_{q=r}^{\infty}\beta^{q-r}
    =2D.
\]
Since $\rho_q\le1$, $\sum_{q=1}^{\infty}\rho_q^2\le2D$.
Corollary~\ref{lem:sum_matrix}, applied to the $2D$ scalar terms in
$\rho_q$ with zeros for $r>q$, also gives
\[
\begin{aligned}
    \sum_{q=1}^{\infty}\rho_q^2
    &\le2D(1-\beta)^2\sum_{r=1}^{2D}
    \sum_{q=r}^{\infty}\beta^{2(q-r)}
    =4D^2\frac{1-\beta}{1+\beta}.
\end{aligned}
\]
This proves the lemma.

\subsection{Proof of Lemma \ref{lem:Ak-barF}}
\label{pf:lem:Ak-barF}

Let
\[
    U_t
    :=
    \norm{\sum_{m=0}^{\min\crk{t,D}}\bpi^\T\bA_m
    \nabla\bar{\bF}^{(t-m)}}^2 .
\]
By Lemma~\ref{lem:Ak},
$\bpi^\T\bA_m=\bfe_{1,m}^\T$ for every $0\le m\le D$. Moreover, the $i$-th
row of $\nabla\bar{\bF}^{(s)}$ is
$\nabla f_i(x_i^{(s)})-\nabla f_i(x_1^{(s)})$. Hence, the sum defining $U_t$
collects exactly the agents whose information has reached the root by time
$t$:
\[
    U_t
    =
    \norm{
    \sum_{i\in\cN:\,h_i\le t}
    \prt{
    \nabla f_i(x_i^{(t-h_i)})
    -\nabla f_i(x_1^{(t-h_i)})}
    }^2 .
\]
There are at most $n$ summands. Corollary~\ref{lem:sum_matrix} and
Assumption~\ref{a.smooth} therefore give the pointwise bound
\[
    U_t
    \le
    nL^2
    \sum_{i\in\cN:\,h_i\le t}
    \norm{x_i^{(t-h_i)}-x_1^{(t-h_i)}}^2 .
\]
Taking expectations, summing over $t$, and setting $s=t-h_i$ for each agent
yield
\[
\begin{aligned}
    \sum_{t=0}^{T-1}\expect U_t
    &\le
    nL^2
    \sum_{i\in\cN:\,h_i\le T-1}
    \sum_{s=0}^{T-1-h_i}
    \expect\norm{x_i^{(s)}-x_1^{(s)}}^2\\
    &\le
    nL^2
    \sum_{s=0}^{T-1}\sum_{i=1}^{n}
    \expect\norm{x_i^{(s)}-x_1^{(s)}}^2=
    nL^2\sum_{s=0}^{T-1}
    \expect\norm{\hat{\bX}^{(s)}}_F^2.
\end{aligned}
\]
Here, the last identity follows from
$\sum_{i=1}^{n}\norm{x_i^{(s)}-x_1^{(s)}}^2
=\norm{\bPi_{\bR}\bX^{(s)}}_F^2$.
This proves the claimed estimate.

\subsection{Proof of Lemma \ref{lem:X_diff}}
\label{pf:lem:X_diff}

By \eqref{eq:root-diff}, \eqref{eq:momentum-kernel}, and
$\bpi^\T\bS=\bpi^\T$,
\begin{equation}
    \label{eq:lem:X_diff-2}
    \begin{aligned}
        \Delta\bar{\bX}^{(t)}
        &=-\gamma\beta^{t+1}\mone\bpi^\T\bM^{(0)}
        -\gamma\mone\sum_{s=0}^{t}\bpi^\T\bK_{t-s}\bG^{(s)}\\
        &=-\gamma\beta^{t+1}\mone\bpi^\T\bM^{(0)}
        -\gamma(1-\beta)\mone\sum_{s=0}^{t}\beta^{t-s}
        \sum_{m=0}^{\min\crk{s,D}}\bpi^\T\bA_m\bG^{(s-m)}.
    \end{aligned}
\end{equation}
The second equality reorders the time and layer indices. Decompose
$\bG^{(r)}=\bTh^{(r)}+\nabla\bar{\bF}^{(r)}
+\nabla\bF(\bar{\bX}^{(r)})$. Summation by parts, the definition of
$\bA_m$, and $\bPi_{\bC}\bC^D=\mathbf{0}$ give, for $s\ge0$,
\[
\begin{aligned}
&\sum_{m=0}^{\min\crk{s,D}}\bpi^\T\bA_m
\nabla\bF(\bar{\bX}^{(s-m)})\\
&\quad=
\bpi^\T\sum_{m=0}^{\min\crk{s,D}-1}\bPi_{\bC}\bC^m
\prt{\nabla\bF(\bar{\bX}^{(s-m)})
-\nabla\bF(\bar{\bX}^{(s-m-1)})}\\
&\qquad
+\bpi^\T\bPi_{\bC}\bC^s\nabla\bF^{(0)}
+\mone^\T\nabla\bF(\bar{\bX}^{(s)}).
\end{aligned}
\]
Here $\nabla\bF(\bar{\bX}^{(0)})=\nabla\bF^{(0)}$ by the common
initialization. Using $\mone^\T\nabla\bF(\bar{\bX}^{(s)})
=n\prt{\nabla f(x_1^{(s)})}^\T$ in \eqref{eq:lem:X_diff-2} yields
\begin{equation}
    \label{eq:lem:X_diff-2-1}
    \Delta\bar{\bX}^{(t)}=\sum_{j=1}^{6}\bH_j^{(t)},
\end{equation}
where sums over empty index sets are understood to be zero and
\[
\begin{aligned}
\bH_1^{(t)}
&:=-\gamma(1-\beta)\mone\sum_{s=0}^{t}\beta^{t-s}
\sum_{m=0}^{\min\crk{s,D}}\bpi^\T\bA_m\bTh^{(s-m)},\\
\bH_2^{(t)}
&:=-\gamma(1-\beta)\mone\sum_{s=0}^{t}\beta^{t-s}
\sum_{m=0}^{\min\crk{s,D}}\bpi^\T\bA_m\nabla\bar{\bF}^{(s-m)},\\
\bH_3^{(t)}
&:=-\gamma(1-\beta)\mone\bpi^\T
\sum_{s=0}^{t}\beta^{t-s}
\sum_{m=0}^{\min\crk{s,D}-1}\bPi_{\bC}\bC^m\\
&\qquad\qquad\cdot
\prt{\nabla\bF(\bar{\bX}^{(s-m)})
-\nabla\bF(\bar{\bX}^{(s-m-1)})},\\
\bH_4^{(t)}
&:=-\gamma(1-\beta)\mone\bpi^\T
\sum_{s=0}^{t}\beta^{t-s}\bPi_{\bC}\bC^s\nabla\bF^{(0)},\\
\bH_5^{(t)}
&:=-\gamma n(1-\beta)\mone
\sum_{s=0}^{t}\beta^{t-s}\prt{\nabla f(x_1^{(s)})}^\T, \bH_6^{(t)}
:=-\gamma\beta^{t+1}\mone\bpi^\T\bM^{(0)}.
\end{aligned}
\]
The six terms correspond to noise, disagreement, root-gradient
variation, routing initialization, the global gradient, and initial
momentum, respectively.

For $\bH_1^{(t)}$, use the kernel form in \eqref{eq:lem:X_diff-2}:
\[
    \bH_1^{(t)}
    =-\gamma\mone\sum_{r=0}^{t}\bpi^\T\bK_{t-r}\bTh^{(r)}.
\]
For $0\le r<q\le t$, $\bTh^{(r)}$ is $\cF_q$-measurable.
Conditional centering in \eqref{eq:noise-conditional-properties} and the
deterministic kernels therefore give
\[
    \expect\left\langle
    \bpi^\T\bK_{t-r}\bTh^{(r)},
    \bpi^\T\bK_{t-q}\bTh^{(q)}
    \right\rangle=0.
\]
Conditional independence of the rows,
\eqref{eq:noise-conditional-properties}, and $\norm{\mone}^2=n$ then yield
\[
\begin{aligned}
    \expect\norm{\bH_1^{(t)}}_F^2
    &=\gamma^2n\sum_{r=0}^{t}
    \expect\norm{\bpi^\T\bK_{t-r}\bTh^{(r)}}^2\le\gamma^2n\sigma^2\sum_{k=0}^{t}
    \norm{\bpi^\T\bK_k}_2^2.
\end{aligned}
\]
By \eqref{eq:momentum-kernel} and Lemma~\ref{lem:Ak},
\[
    \bpi^\T\bK_k
    =(1-\beta)\sum_{m=0}^{\min\crk{k,D}}
    \beta^{k-m}\bfe_{1,m}^\T.
\]
The disjoint layer supports and the geometric-series bound give
\[
\begin{aligned}
    \sum_{k=0}^{t}\norm{\bpi^\T\bK_k}_2^2
    &=(1-\beta)^2\sum_{m=0}^{\min\crk{t,D}}
    |\cI_{1,m}|\sum_{u=0}^{t-m}\beta^{2u}\le\frac{1-\beta}{1+\beta}
    \sum_{m=0}^{\min\crk{t,D}}|\cI_{1,m}|
    \le n\frac{1-\beta}{1+\beta}.
\end{aligned}
\]
Summing over $t=0,\ldots,T-1$ yields
\begin{equation}
    \label{eq:lem:X_diff-H1}
    \sum_{t=0}^{T-1}\expect\norm{\bH_1^{(t)}}_F^2
    \le\gamma^2n^2\frac{1-\beta}{1+\beta}\sigma^2T.
\end{equation}

For $\bH_2^{(t)}$, set
\[
    \bV^{(s)}
    :=\sum_{m=0}^{\min\crk{s,D}}
    \bpi^\T\bA_m\nabla\bar{\bF}^{(s-m)},\qquad s\ge0.
\]
Set $w_{t,s}:=(1-\beta)\beta^{t-s}$ for $0\le s\le t$.
Since $\sum_{s=0}^{t}w_{t,s}=1-\beta^{t+1}\le1$,
Lemma~\ref{lem:weighted_sum} with matrices $\bV^{(s)}$ and weights
$w_{t,s}$ gives
\[
\begin{aligned}
    \expect\norm{\bH_2^{(t)}}_F^2
    &=\gamma^2n\expect\norm{
    (1-\beta)\sum_{s=0}^{t}\beta^{t-s}\bV^{(s)}}^2 \le\gamma^2n(1-\beta)\sum_{s=0}^{t}\beta^{t-s}
    \expect\norm{\bV^{(s)}}^2.
\end{aligned}
\]
Reordering the time sums and applying Lemma~\ref{lem:Ak-barF} yield
\begin{equation}
    \label{eq:lem:X_diff-H2}
\begin{aligned}
    \sum_{t=0}^{T-1}\expect\norm{\bH_2^{(t)}}_F^2
    &\le\gamma^2n\sum_{s=0}^{T-1}
    \brk{(1-\beta)\sum_{t=s}^{T-1}\beta^{t-s}}
    \expect\norm{\bV^{(s)}}^2\\
    &\le\gamma^2n\sum_{s=0}^{T-1}\expect\norm{\bV^{(s)}}^2 \le\gamma^2n^2L^2\sum_{s=0}^{T-1}
    \expect\norm{\hat{\bX}^{(s)}}_F^2.
\end{aligned}
\end{equation}

For $\bH_3^{(t)}$, set
\[
\begin{aligned}
    \bW^{(s)}
    &:=\bpi^\T\sum_{m=0}^{\min\crk{s,D}-1}
    \bPi_{\bC}\bC^m\prt{\nabla\bF(\bar{\bX}^{(s-m)})
    -\nabla\bF(\bar{\bX}^{(s-m-1)})},\qquad s\ge0.
\end{aligned}
\]
Corollary~\ref{lem:sum_matrix}, Assumption~\ref{a.smooth}, and
Lemma~\ref{lem:Rk_norm} imply
\[
\begin{aligned}
    \expect\norm{\bW^{(s)}}^2
    &\le D\sum_{m=0}^{\min\crk{s,D}-1}
    \expect\norm{\bpi^\T\bPi_{\bC}\bC^m
    \prt{\nabla\bF(\bar{\bX}^{(s-m)})
    -\nabla\bF(\bar{\bX}^{(s-m-1)})}}^2\\
    &\le2DL^2\sum_{m=0}^{\min\crk{s,D}-1}
    (n-n_m)\expect\norm{\Delta\bar{\bX}^{(s-m-1)}}_F^2.
\end{aligned}
\]
Applying Lemma~\ref{lem:weighted_sum} with matrices $\bW^{(s)}$ and
weights $w_{t,s}$ and summing over $t$ gives
\[
    \sum_{t=0}^{T-1}\expect\norm{\bH_3^{(t)}}_F^2
    \le\gamma^2n\sum_{s=0}^{T-1}\expect\norm{\bW^{(s)}}^2.
\]
Lemma~\ref{lem:sum_help} and $\sum_{m=0}^{D-1}(n-n_m)=nD_{\avg}$ then yield
\begin{equation}
    \label{eq:lem:X_diff-H3}
\begin{aligned}
    \sum_{t=0}^{T-1}\expect\norm{\bH_3^{(t)}}_F^2
    &\le2\gamma^2nDL^2\sum_{m=0}^{D-1}(n-n_m)
    \sum_{s=0}^{T-1}\expect\norm{\Delta\bar{\bX}^{(s)}}_F^2\\
    &=2\gamma^2n^2DD_{\avg}L^2
    \sum_{s=0}^{T-1}\expect\norm{\Delta\bar{\bX}^{(s)}}_F^2.
\end{aligned}
\end{equation}

For $\bH_4^{(t)}$, apply Lemma~\ref{lem:weighted_sum} with matrices
$\bpi^\T\bPi_{\bC}\bC^s\nabla\bF^{(0)}$ and weights $w_{t,s}$.
Summing over $t$ and using Lemma~\ref{lem:Rk_norm},
$\bPi_{\bC}\bC^s=\bC^s-\bpi\mone^\T$, and
$\bPi_{\bC}\bC^s=\mathbf{0}$ for $s\ge D$ give
\begin{equation}
    \label{eq:lem:X_diff-H4}
\begin{aligned}
    \sum_{t=0}^{T-1}\expect\norm{\bH_4^{(t)}}_F^2
    &\le\gamma^2n\sum_{s=0}^{T-1}
    \norm{\bpi^\T\bPi_{\bC}\bC^s\nabla\bF^{(0)}}^2 \le2\gamma^2n\sum_{s=0}^{D-1}(n-n_s)
    \norm{\nabla\bF^{(0)}}_F^2\\
    &=2\gamma^2n^2D_{\avg}\norm{\nabla\bF^{(0)}}_F^2.
\end{aligned}
\end{equation}

For $\bH_5^{(t)}$, apply Lemma~\ref{lem:weighted_sum} with matrices
$\prt{\nabla f(x_1^{(s)})}^\T$ and weights $w_{t,s}$, then sum over $t$
and use Lemma~\ref{lem:sum_help}:
\begin{equation}
    \label{eq:lem:X_diff-H5}
\begin{aligned}
    \sum_{t=0}^{T-1}\expect\norm{\bH_5^{(t)}}_F^2
    &\le\gamma^2n^3(1-\beta)
    \sum_{t=0}^{T-1}\sum_{s=0}^{t}\beta^{t-s}
    \expect\norm{\nabla f(x_1^{(s)})}^2\\
    &\le\gamma^2n^3\sum_{s=0}^{T-1}
    \expect\norm{\nabla f(x_1^{(s)})}^2.
\end{aligned}
\end{equation}

For $\bH_6^{(t)}$, $\bM^{(0)}=\bS\bG^{(0)}$ implies
$\bpi^\T\bM^{(0)}=\prt{g_1(x^{(0)},\xi_1^{(0)})}^\T$.
Assumption~\ref{a.var} gives
\[
\begin{aligned}
    \expect\norm{\bpi^\T\bM^{(0)}}^2
    &=\norm{\nabla f_1(x^{(0)})}^2
    +\expect\norm{\brk{\bTh^{(0)}}_{1:}}^2\le\norm{\nabla f_1(x^{(0)})}^2+\sigma^2.
\end{aligned}
\]
Summing the geometric series gives
\begin{equation}
    \label{eq:lem:X_diff-H6}
\begin{aligned}
    \sum_{t=0}^{T-1}\expect\norm{\bH_6^{(t)}}_F^2
    &=\gamma^2n\sum_{t=0}^{T-1}\beta^{2(t+1)}
    \expect\norm{\bpi^\T\bM^{(0)}}^2 \le\gamma^2n\frac{\beta^2}{1-\beta}
    \brk{\norm{\nabla \bF^{(0)}}^2_F+\sigma^2}.
\end{aligned}
\end{equation}

By Corollary~\ref{lem:sum_matrix} and
\eqref{eq:lem:X_diff-2-1}--\eqref{eq:lem:X_diff-H6}, we obtain
\begin{equation}
    \label{eq:lem:X_diff-4}
    \begin{aligned}
        & \sum_{t=0}^{T-1}\expect\norm{\Delta\bar{\bX}^{(t)}}_F^2
        \le6\sum_{i=1}^{6}\sum_{t=0}^{T-1}
        \expect\norm{\bH_i^{(t)}}_F^2\\
        \le & 6\gamma^2n^2\frac{1-\beta}{1+\beta}\sigma^2T
        +6\gamma^2n^2L^2
        \sum_{t=0}^{T-1}\expect\norm{\hat{\bX}^{(t)}}_F^2+12\gamma^2n^2DD_{\avg}L^2
        \sum_{t=0}^{T-1}\expect\norm{\Delta\bar{\bX}^{(t)}}_F^2\\
        &+12\gamma^2n \prt{n D_{\avg} + \frac{\beta^2}{1-\beta}}\norm{\nabla\bF^{(0)}}_F^2
        +6\gamma^2n^3
        \sum_{t=0}^{T-1}\expect\norm{\nabla f(x_1^{(t)})}^2 +6\gamma^2n\frac{\beta^2}{1-\beta}
        \sigma^2.
    \end{aligned}
\end{equation}
For $\gamma\le1/(6n\sqrt{DD_{\avg}}L)$, the coefficient
$12\gamma^2n^2DD_{\avg}L^2$ is at most $1/3$.
Absorbing this term into the left-hand side of
\eqref{eq:lem:X_diff-4} gives
\[
\begin{aligned}
    & \sum_{t=0}^{T-1}\expect\norm{\Delta\bar{\bX}^{(t)}}_F^2
    \le9\gamma^2n^2\frac{1-\beta}{1+\beta}\sigma^2T
    +9\gamma^2n^2L^2
    \sum_{t=0}^{T-1}\expect\norm{\hat{\bX}^{(t)}}_F^2\\
    &\quad+18\gamma^2n \prt{n D_{\avg} + \frac{\beta^2}{1-\beta}}\norm{\nabla\bF^{(0)}}_F^2
    +9\gamma^2n^3
    \sum_{t=0}^{T-1}\expect\norm{\nabla f(x_1^{(t)})}^2 + 9\gamma^2n\frac{\beta^2}{1-\beta}
        \sigma^2.
\end{aligned}
\]
This proves the lemma.

\subsection{Proof of Lemma \ref{lem:PiX}}
\label{pf:lem:PiX}

Substituting
$\bG^{(m)}=\bTh^{(m)}+\nabla\bar{\bF}^{(m)}
+\nabla\bF(\bar{\bX}^{(m)})$ into \eqref{eq:hatX-Y} gives
\[
\begin{aligned}
    \hat{\bX}^{(t)}
    &=-\gamma\sum_{m=0}^{t-1}\sum_{j=1}^{t-m}
    \bPi_{\bR}\bR^{j-1}\bK_{t-m-j}
    \prt{\bTh^{(m)}+\nabla\bar{\bF}^{(m)}
    +\nabla\bF(\bar{\bX}^{(m)})}\\
    &\quad-\gamma\beta\sum_{j=1}^{t}\beta^{t-j}
    \bPi_{\bR}\bR^{j-1}\bM^{(0)}\\
    &=:\bM_1^{(t)}+\bM_2^{(t)}+\bM_3^{(t)}+\bM_4^{(t)}.
\end{aligned}
\]
The four terms collect noise, disagreement gradients, root gradients,
and initial momentum, respectively.

We use the notation $\bJ_q$ and $\rho_q$ from
Lemma~\ref{lem:RC-interaction}.

For $\bM_1^{(t)}$, conditional centering and row independence in
\eqref{eq:noise-conditional-properties}, together with
Lemma~\ref{lem:RC-interaction}, give
\[
\begin{aligned}
    \expect\norm{\bM_1^{(t)}}_F^2
    &=\gamma^2\sum_{m=0}^{t-1}
    \expect\norm{\bJ_{t-m}\bTh^{(m)}}_F^2
    &\le\gamma^2\sigma^2\sum_{q=1}^{t}\norm{\bJ_q}_F^2
    \le\gamma^2n^2\sigma^2\sum_{q=1}^{t}\rho_q^2.
\end{aligned}
\]
Summing over $t$ and applying Lemma~\ref{lem:RC-interaction} yields
\begin{equation}
    \label{eq:lem:PiX-1}
\begin{aligned}
    \sum_{t=0}^{T-1}\expect\norm{\bM_1^{(t)}}_F^2
    &\le\gamma^2n^2\sigma^2
    \sum_{q=1}^{T-1}(T-q)\rho_q^2
    &\le\gamma^2n^2\sigma^2T
    \min\crk{2D,\,4D^2\frac{1-\beta}{1+\beta}}.
\end{aligned}
\end{equation}

For $\bM_2^{(t)}$, apply Lemmas~\ref{lem:RC-interaction}
and~\ref{lem:weighted_sum}, using weights $\rho_{t-m}$ for the scalar
norms $\norm{\nabla\bar{\bF}^{(m)}}_F$.
Assumption~\ref{a.smooth} then gives
\[
\begin{aligned}
    \expect\norm{\bM_2^{(t)}}_F^2
    &\le\gamma^2n^2\prt{\sum_{m=0}^{t-1}\rho_{t-m}}
    \sum_{m=0}^{t-1}\rho_{t-m}
    \expect\norm{\nabla\bar{\bF}^{(m)}}_F^2\\
    &\le2\gamma^2n^2DL^2\sum_{m=0}^{t-1}\rho_{t-m}
    \expect\norm{\hat{\bX}^{(m)}}_F^2.
\end{aligned}
\]
Reordering the time sums and using $\sum_{q=1}^{\infty}\rho_q=2D$ yields
\begin{equation}
    \label{eq:lem:PiX-2}
\begin{aligned}
    \sum_{t=0}^{T-1}\expect\norm{\bM_2^{(t)}}_F^2
    &\le2\gamma^2n^2DL^2\sum_{m=0}^{T-1}
    \prt{\sum_{q=1}^{T-1-m}\rho_q}
    \expect\norm{\hat{\bX}^{(m)}}_F^2\\
    &\le4\gamma^2n^2D^2L^2\sum_{m=0}^{T-1}
    \expect\norm{\hat{\bX}^{(m)}}_F^2.
\end{aligned}
\end{equation}

For $\bM_3^{(t)}$, applying \eqref{eq:momentum-kernel} and setting
$r=m+\ell+j$ yields
\begin{equation}
    \label{eq:lem:PiX-M3-expansion}
\begin{aligned}
    \bM_3^{(t)}
    &=-\gamma\sum_{m=0}^{t-1}\sum_{j=1}^{t-m}
    \bPi_{\bR}\bR^{j-1}\bK_{t-m-j}
    \nabla\bF(\bar{\bX}^{(m)})\\
    &=-\gamma(1-\beta)\sum_{m=0}^{t-1}\sum_{j=1}^{t-m}
    \sum_{\ell=0}^{\min\crk{t-m-j,D}}\beta^{t-m-j-\ell}
    \bPi_{\bR}\bR^{j-1}\bS\bA_\ell
    \nabla\bF(\bar{\bX}^{(m)})\\
    &=-\gamma(1-\beta)\sum_{r=1}^{t}\beta^{t-r}
    \sum_{j=1}^{\min\crk{r,D}}\bPi_{\bR}\bR^{j-1}\bpi
    \sum_{\ell=0}^{\min\crk{r-j,D}}
    \bpi^\T\bA_\ell\nabla\bF(\bar{\bX}^{(r-j-\ell)}).
\end{aligned}
\end{equation}
The last equality uses $\bS=\bpi\bpi^\T$ and
$\bPi_{\bR}\bR^{j-1}=\mathbf{0}$ for $j>D$.

For $s\ge0$, define
\[
\begin{aligned}
    \bM_{3,1}^{(s)}
    &:=\sum_{\ell=0}^{\min\crk{s,D}}\bpi^\T\bA_\ell
    \brk{\nabla\bF(\bar{\bX}^{(s-\ell)})
    -\nabla\bF(\bar{\bX}^{(s)})},\\
    \bM_{3,2}^{(s)}
    &:=n\prt{\nabla f(x_1^{(s)})}^\T,\\
    \bM_{3,3}^{(s)}
    &:=-\sum_{\ell=\min\crk{s,D}+1}^{D}\bpi^\T\bA_\ell
    \nabla\bF(\bar{\bX}^{(s)}).
\end{aligned}
\]
Then, by \eqref{eq:lem:PiX-M3-expansion} and Lemma~\ref{lem:Ak}, we have
\[
\begin{aligned}
    \bM_3^{(t)}
    &=-\gamma(1-\beta)\sum_{r=1}^{t}\beta^{t-r}
    \sum_{j=1}^{\min\crk{r,D}}\bPi_{\bR}\bR^{j-1}\bpi\brk{\bM_{3,1}^{(r-j)}
    +\bM_{3,2}^{(r-j)}+\bM_{3,3}^{(r-j)}}.
\end{aligned}
\]
Thus, the three components share the same outer geometric weights.
Although $\bM_{3,3}^{(s)}$ has finite support, its contribution to
$\bM_3^{(t)}$ retains a geometric tail when $\beta>0$.
Lemma~\ref{lem:Ak} and \eqref{eq:tree-rgt-decision-delay} also give
\[
\begin{aligned}
    \bM_{3,1}^{(s)}
    &=\sum_{\ell=0}^{\min\crk{s,D}}\sum_{i\in\cI_{1,\ell}}
    \prt{\nabla f_i(x_1^{(s-\ell)})-\nabla f_i(x_1^{(s)})}^\T\\
    &=\sum_{i\in[n]:\,h_i\le s}
    \prt{\nabla f_i(x_i^{(s)})-\nabla f_i(x_1^{(s)})}^\T.
\end{aligned}
\]
Assumption~\ref{a.smooth} and Corollary~\ref{lem:sum_matrix} imply
\[
\begin{aligned}
    \norm{\bM_{3,1}^{(s)}}^2
    &\le nL^2\sum_{i\in[n]:\,h_i\le s}
    \norm{x_i^{(s)}-x_1^{(s)}}^2
    \le nL^2\norm{\hat{\bX}^{(s)}}_F^2.
\end{aligned}
\]
For $\bM_{3,1}^{(s)}$, apply Lemma~\ref{lem:weighted_sum} to the outer sum
with weights $(1-\beta)\beta^{t-r}$, whose sum is $1-\beta^t\le1$,
and Corollary~\ref{lem:sum_matrix} to the inner sum, which has at most $D$
terms. Using $\norm{\bPi_{\bR}\bR^{j-1}\bpi}^2=n-n_{j-1}$,
summing over $t$, and reordering gives
\begin{equation}
    \label{eq:lem:PiX-m31-1}
\begin{aligned}
    &\sum_{t=0}^{T-1}\expect\norm{(1-\beta)
    \sum_{r=1}^{t}\beta^{t-r}\sum_{j=1}^{\min\crk{r,D}}
    \bPi_{\bR}\bR^{j-1}\bpi\bM_{3,1}^{(r-j)}}_F^2\\
    \le & D\sum_{r=1}^{T-1}\prt{1-\beta^{T-r}}
    \sum_{j=1}^{\min\crk{r,D}}\prt{n-n_{j-1}}
    \expect\norm{\bM_{3,1}^{(r-j)}}^2\\
    \le & nDL^2\sum_{r=1}^{T-1}
    \sum_{j=1}^{\min\crk{r,D}}\prt{n-n_{j-1}}
    \expect\norm{\hat{\bX}^{(r-j)}}_F^2 \le n^2DD_{\avg}L^2\sum_{t=0}^{T-1}
    \expect\norm{\hat{\bX}^{(t)}}_F^2.
\end{aligned}
\end{equation}
Here $(1-\beta)\sum_{t=r}^{T-1}\beta^{t-r}=1-\beta^{T-r}\le1$,
and the last inequality follows from Lemma~\ref{lem:sum_help} and
$\sum_{j=1}^{D}(n-n_{j-1})=nD_{\avg}$.

For $\bM_{3,2}^{(s)}$, its definition gives
\[
    \norm{\bM_{3,2}^{(s)}}^2
    =n^2\norm{\nabla f(x_1^{(s)})}^2.
\]
Applying the same argument as in \eqref{eq:lem:PiX-m31-1} yields
\begin{equation}
    \label{eq:lem:PiX-m32-1}
\begin{aligned}
    &\sum_{t=0}^{T-1}\expect\norm{(1-\beta)
    \sum_{r=1}^{t}\beta^{t-r}\sum_{j=1}^{\min\crk{r,D}}
    \bPi_{\bR}\bR^{j-1}\bpi\bM_{3,2}^{(r-j)}}_F^2\\
    \le & n^2D\sum_{r=1}^{T-1}\prt{1-\beta^{T-r}}
    \sum_{j=1}^{\min\crk{r,D}}\prt{n-n_{j-1}}
    \expect\norm{\nabla f(x_1^{(r-j)})}^2\\
    \le & n^3DD_{\avg}\sum_{t=0}^{T-1}
    \expect\norm{\nabla f(x_1^{(t)})}^2.
\end{aligned}
\end{equation}
For $\bM_{3,3}^{(s)}$, Lemma~\ref{lem:Ak} gives
\[
    \bM_{3,3}^{(s)}
    =-\sum_{i\in[n]:\,h_i>s}\prt{\nabla f_i(x_1^{(s)})}^\T,
\]
which vanishes for $s\ge D$. For $h_i>s$,
\eqref{eq:tree-rgt-decision-delay} gives $x_i^{(s)}=x^{(0)}$.
Assumption~\ref{a.smooth} and Corollary~\ref{lem:sum_matrix} therefore imply,
for $0\le s<D$,
\[
\begin{aligned}
    \norm{\bM_{3,3}^{(s)}}^2
    &\le2n\sum_{i\in[n]:\,h_i>s}
    \brk{L^2\norm{x_i^{(s)}-x_1^{(s)}}^2
    +\norm{\nabla f_i(x^{(0)})}^2}\\
    &\le2nL^2\norm{\hat{\bX}^{(s)}}_F^2
    +2n\norm{\nabla\bF^{(0)}}_F^2.
\end{aligned}
\]
Applying Lemma~\ref{lem:weighted_sum} with weights
$(1-\beta)\beta^{t-r}$ and Corollary~\ref{lem:sum_matrix} as in
\eqref{eq:lem:PiX-m31-1}, then summing over $t$, yields
\begin{equation}
    \label{eq:lem:PiX-m33}
\begin{aligned}
    &\sum_{t=0}^{T-1}\expect\norm{(1-\beta)
    \sum_{r=1}^{t}\beta^{t-r}\sum_{j=1}^{\min\crk{r,D}}
    \bPi_{\bR}\bR^{j-1}\bpi\bM_{3,3}^{(r-j)}}_F^2\\
    \le{}& D\sum_{r=1}^{T-1}\prt{1-\beta^{T-r}}
    \sum_{j=\max\crk{1,r-D+1}}^{\min\crk{r,D}}\prt{n-n_{j-1}}
    \expect\norm{\bM_{3,3}^{(r-j)}}^2\\
    \le{}& 2nD\sum_{j=1}^{D}\prt{n-n_{j-1}}
    \sum_{s=0}^{\min\crk{T-j-1,D-1}}
    \brk{L^2\expect\norm{\hat{\bX}^{(s)}}_F^2
    +\norm{\nabla\bF^{(0)}}_F^2}\\
    \le{}& 2n^2DD_{\avg}L^2\sum_{t=0}^{T-1}
    \expect\norm{\hat{\bX}^{(t)}}_F^2
    +2n^2D^2D_{\avg}\norm{\nabla\bF^{(0)}}_F^2.
\end{aligned}
\end{equation}
The second inequality sets $s=r-j$ and uses $1-\beta^{T-r}\le1$.
The last uses at most $D$ boundary times for each $j$ and
$\sum_{j=1}^{D}(n-n_{j-1})=nD_{\avg}$.

Applying Corollary~\ref{lem:sum_matrix} to the three contributions in
$\bM_3^{(t)}$ and using
\eqref{eq:lem:PiX-m31-1}--\eqref{eq:lem:PiX-m33} yields
\begin{equation}
    \label{eq:lem:PiX-3-1}
\begin{aligned}
    &\sum_{t=0}^{T-1}\expect\norm{\bM_3^{(t)}}_F^2 \le{}9\gamma^2n^2DD_{\avg}L^2\sum_{t=0}^{T-1}
    \expect\norm{\hat{\bX}^{(t)}}_F^2  \\
    & \qquad +3\gamma^2n^3DD_{\avg}\sum_{t=0}^{T-1}
    \expect\norm{\nabla f(x_1^{(t)})}^2 +6\gamma^2n^2D^2D_{\avg}\norm{\nabla\bF^{(0)}}_F^2.
\end{aligned}
\end{equation}

For $\bM_4^{(t)}$, $\bM^{(0)}=\bS\bG^{(0)}$ and
$\bPi_{\bR}\bR^{j-1}=\mathbf{0}$ for $j>D$ give
\[
    \bM_4^{(t)}=-\gamma\beta
    \sum_{j=1}^{\min\crk{t,D}}\beta^{t-j}
    \bPi_{\bR}\bR^{j-1}\bpi\prt{\bpi^\T\bM^{(0)}}.
\]
By \eqref{eq:noise-conditional-properties},
\[
\begin{aligned}
    \expect\norm{\bpi^\T\bM^{(0)}}^2
    &=\norm{\bpi^\T\nabla\bF^{(0)}}^2
    +\expect\norm{\bpi^\T\bTh^{(0)}}^2
    \le\norm{\nabla\bF^{(0)}}_F^2+\sigma^2.
\end{aligned}
\]
Apply Corollary~\ref{lem:sum_matrix} to the sum over $j$, which has at
most $D$ terms. Summing the geometric series and using
$\norm{\bPi_{\bR}\bR^{j-1}\bpi}^2=n-n_{j-1}$ and
$\sum_{j=1}^{D}(n-n_{j-1})=nD_{\avg}$ yields
\begin{equation}
    \label{eq:lem:PiX-4}
\begin{aligned}
    &\sum_{t=0}^{T-1}\expect\norm{\bM_4^{(t)}}_F^2
    \le{}\gamma^2D\frac{\beta^2}{1-\beta^2}
    \expect\norm{\bpi^\T\bM^{(0)}}^2
    \sum_{j=1}^{\min\crk{T-1,D}}\prt{n-n_{j-1}}
    \prt{1-\beta^{2(T-j)}}\\
    & \le{}\gamma^2nDD_{\avg}\frac{\beta^2}{1-\beta^2}
    \brk{\norm{\nabla\bF^{(0)}}_F^2+\sigma^2}.
\end{aligned}
\end{equation}

By Corollary~\ref{lem:sum_matrix} and
\eqref{eq:lem:PiX-1}, \eqref{eq:lem:PiX-2},
\eqref{eq:lem:PiX-3-1}, and \eqref{eq:lem:PiX-4}, we obtain
\[
\begin{aligned}
    &\sum_{t=0}^{T-1}\expect\norm{\hat{\bX}^{(t)}}_F^2
    \le4\sum_{k=1}^{4}\sum_{t=0}^{T-1}
    \expect\norm{\bM_k^{(t)}}_F^2\\
    \le{}&4\gamma^2n^2\sigma^2T
    \min\crk{2D,\,4D^2\frac{1-\beta}{1+\beta}}+\gamma^2n^2L^2\prt{16D^2+36DD_{\avg}}
    \sum_{t=0}^{T-1}\expect\norm{\hat{\bX}^{(t)}}_F^2\\
    &\quad+12\gamma^2n^3DD_{\avg}\sum_{t=0}^{T-1}
    \expect\norm{\nabla f(x_1^{(t)})}^2
    +24\gamma^2n^2D^2D_{\avg}\norm{\nabla\bF^{(0)}}_F^2\\
    &\quad+4\gamma^2nDD_{\avg}\frac{\beta^2}{1-\beta^2}
    \brk{\norm{\nabla\bF^{(0)}}_F^2+\sigma^2}.
\end{aligned}
\]
For $\gamma\le1/(20nDL)$, $D_{\avg}\le D$ implies
$\gamma^2n^2L^2(16D^2+36DD_{\avg})
\le52\gamma^2n^2D^2L^2\le1/5$.
Absorbing this term gives
\[
\begin{aligned}
    \sum_{t=0}^{T-1}\expect\norm{\hat{\bX}^{(t)}}_F^2
    &\le5\gamma^2n^2\sigma^2T
    \min\crk{2D,\,4D^2\frac{1-\beta}{1+\beta}}+15\gamma^2n^3DD_{\avg}\sum_{t=0}^{T-1}
    \expect\norm{\nabla f(x_1^{(t)})}^2\\
    &\quad+30\gamma^2n^2D^2D_{\avg}\norm{\nabla\bF^{(0)}}_F^2 +5\gamma^2nDD_{\avg}\frac{\beta^2}{1-\beta^2}
    \brk{\norm{\nabla\bF^{(0)}}_F^2+\sigma^2}.
\end{aligned}
\]
This proves the lemma.

\subsection{Proof of Lemma \ref{lem:smooth}}
\label{pf:lem:smooth}

Let $f^\star:=\inf_x f(x)$ and
$\Delta_f:=f(x^{(0)})-f^\star$. Assumption~\ref{a.smooth} makes $f$
$L$-smooth. Since $\norm{\Delta\bar{\bX}^{(t)}}_F^2
=n\norm{x_1^{(t+1)}-x_1^{(t)}}^2$, the descent inequality gives
\begin{equation}
    \label{eq:lem:smooth-1}
\begin{aligned}
    \expect f(x_1^{(t+1)})
    &\le\expect f(x_1^{(t)})
    +\expect\ip{\nabla f(x_1^{(t)})}{x_1^{(t+1)}-x_1^{(t)}} +\frac{L}{2n}\expect\norm{\Delta\bar{\bX}^{(t)}}_F^2.
\end{aligned}
\end{equation}
For $t\ge0$, define
\[
\omega_{t,s}:=\begin{cases}
    \beta^t,&s=0,\\
    (1-\beta)\beta^{t-s},&1\le s\le t,
\end{cases}
\qquad\sum_{s=0}^{t}\omega_{t,s}=1.
\]
By \eqref{eq:momentum-closed-form} and
\eqref{eq:momentum-kernel},
\begin{equation}
    \label{eq:lem:smooth-2}
\begin{aligned}
    \prt{x_1^{(t+1)}-x_1^{(t)}}^\T
    &=-\gamma\bpi^\T\bM^{(t)}=-\gamma\beta^{t+1}\bpi^\T\bM^{(0)}
    -\gamma\sum_{r=0}^{t}\bpi^\T\bK_{t-r}\bG^{(r)}\\
    &=-\gamma\beta^{t+1}\bpi^\T\bM^{(0)}
    -\gamma(1-\beta)\sum_{s=0}^{t}\beta^{t-s}
    \sum_{m=0}^{\min\crk{s,D}}\bpi^\T\bA_m\bG^{(s-m)}\\
    &=-\gamma\sum_{s=0}^{t}\omega_{t,s}
    \sum_{m=0}^{\min\crk{s,D}}\bpi^\T\bA_m\bG^{(s-m)}.
\end{aligned}
\end{equation}
The third equality sets $s=r+m$ after expanding $\bK_{t-r}$ and
uses $\bpi^\T\bS=\bpi^\T$. The last equality uses
$\bpi^\T\bM^{(0)}=\bpi^\T\bA_0\bG^{(0)}$ and combines the
time-zero coefficients as
$\beta^{t+1}+(1-\beta)\beta^t=\beta^t=\omega_{t,0}$.
Decompose $\bG^{(s)}=\bTh^{(s)}+\nabla\bar{\bF}^{(s)}
+\nabla\bF(\bar{\bX}^{(s)})$. Applying summation by parts to the delayed
root gradients, as in Appendix~\ref{pf:lem:X_diff}, and using
$\sum_{s=0}^{t}\omega_{t,s}=1$ yields
\[
    \bpi^\T\bM^{(t)}
    =n\prt{\nabla f(x_1^{(t)})}^\T+\sum_{k=1}^{4}\bQ_k^{(t)},
\]
where
\[
\begin{aligned}
    \bQ_1^{(t)}
    &:=\sum_{s=0}^{t}\omega_{t,s}
    \sum_{m=0}^{\min\crk{s,D}}\bpi^\T\bA_m\bTh^{(s-m)},\bQ_2^{(t)}
    :=\sum_{s=0}^{t}\omega_{t,s}
    \sum_{m=0}^{\min\crk{s,D}}\bpi^\T\bA_m\nabla\bar{\bF}^{(s-m)},\\
    \bQ_3^{(t)}
    &:=\sum_{s=0}^{t}\omega_{t,s}\bpi^\T
    \sum_{m=0}^{\min\crk{s,D}-1}\bPi_{\bC}\bC^m
    \prt{\nabla\bF(\bar{\bX}^{(s-m)})
    -\nabla\bF(\bar{\bX}^{(s-m-1)})}\\
    &\quad+\sum_{s=0}^{t}\omega_{t,s}\bpi^\T
    \bPi_{\bC}\bC^s\nabla\bF^{(0)}, \bQ_4^{(t)}
    :=n\sum_{s=0}^{t}\omega_{t,s}
    \prt{\nabla f(x_1^{(s)})-\nabla f(x_1^{(t)})}^\T.
\end{aligned}
\]
Define the four scalar correction terms by
\[
\begin{aligned}
    \bP_1^{(t)}&:=\ip{\nabla f(x_1^{(t)})}{\prt{\bQ_1^{(t)}}^\T},
    &\bP_2^{(t)}&:=\ip{\nabla f(x_1^{(t)})}{\prt{\bQ_2^{(t)}}^\T},\\
    \bP_3^{(t)}&:=\ip{\nabla f(x_1^{(t)})}{\prt{\bQ_3^{(t)}}^\T},
    &\bP_4^{(t)}&:=\ip{\nabla f(x_1^{(t)})}{\prt{\bQ_4^{(t)}}^\T}.
\end{aligned}
\]
They correspond to noise-gradient correlations, disagreement gradients,
routing variation and initialization, and momentum gradient variation,
respectively.
Substituting the four corrections into \eqref{eq:lem:smooth-1} yields
\begin{equation}
    \label{eq:lem:smooth-2-2}
\begin{aligned}
    \expect f(x_1^{(t+1)})
    &\le\expect f(x_1^{(t)})
    -\gamma n\expect\norm{\nabla f(x_1^{(t)})}^2-\gamma\sum_{k=1}^{4}\expect\bP_k^{(t)}
    +\frac{L}{2n}\expect\norm{\Delta\bar{\bX}^{(t)}}_F^2.
\end{aligned}
\end{equation}
We bound the four corrections separately. For $\bP_1^{(t)}$, set
\[
\mathbf{N}^{(s)}
:=\sum_{m=0}^{\min\crk{s,D}}\bpi^\T\bA_m\bTh^{(s-m)},
\qquad s\ge0.
\]
By Lemma~\ref{lem:Ak}, different arrival times use distinct node-time
sample pairs. Conditional centering and row independence in
\eqref{eq:noise-conditional-properties} make their noises orthogonal in
expectation. Thus, for $s\ne q$,
\[
\expect\ip{\mathbf{N}^{(s)}}{\mathbf{N}^{(q)}}=0,
\qquad\expect\norm{\mathbf{N}^{(s)}}^2\le n\sigma^2,
\qquad\expect\norm{\mathbf{N}^{(0)}}^2\le\sigma^2.
\]
Since $\bQ_1^{(r)}=\sum_{s=0}^{r}\omega_{r,s}\mathbf{N}^{(s)}$, the
squared momentum weights give
\[
\begin{aligned}
    \expect\norm{\bQ_1^{(r)}}^2
    &=\beta^{2r}\expect\norm{\mathbf{N}^{(0)}}^2
    +(1-\beta)^2\sum_{s=1}^{r}\beta^{2(r-s)}
    \expect\norm{\mathbf{N}^{(s)}}^2\le\beta^{2r}\sigma^2
    +n\frac{1-\beta}{1+\beta}\sigma^2.
\end{aligned}
\]
Lemma~\ref{lem:independ1} gives
$\expect\ip{\nabla f(x_1^{(s)})}{\prt{\mathbf{N}^{(s)}}^\T}=0$.
Expanding the gradient differences and using
$\omega_{t,s}=\beta^{t-r}\omega_{r,s}$ for $s\le r\le t$ yields
\[
\begin{aligned}
    \expect\bP_1^{(t)}
    &=\sum_{s=0}^{t}\omega_{t,s}\expect\ip{
    \nabla f(x_1^{(t)})-\nabla f(x_1^{(s)})}{
    \prt{\mathbf{N}^{(s)}}^\T}\\
    &=\sum_{r=0}^{t-1}\beta^{t-r}\expect\ip{
    \nabla f(x_1^{(r+1)})-\nabla f(x_1^{(r)})}{
    \prt{\bQ_1^{(r)}}^\T}.
\end{aligned}
\]
Smoothness and Young's inequality, together with
$\norm{x_1^{(r+1)}-x_1^{(r)}}^2
=\norm{\Delta\bar{\bX}^{(r)}}_F^2/n$, imply
\[
\begin{aligned}
    -\gamma\expect\bP_1^{(t)}
    &\le\frac{L}{2n}\sum_{r=0}^{t-1}\beta^{t-r}
    \expect\norm{\Delta\bar{\bX}^{(r)}}_F^2
    +\frac{\gamma^2L}{2}\sum_{r=0}^{t-1}\beta^{t-r}
    \expect\norm{\bQ_1^{(r)}}^2.
\end{aligned}
\]
Summing over $t=0,\ldots,T-1$, using
$\sum_{t=r+1}^{T-1}\beta^{t-r}\le\beta/(1-\beta)$ and
$\sum_{r=0}^{T-2}\beta^{2r}\le1/(1-\beta^2)$, gives
\begin{equation}
    \label{eq:lem:smooth-P1}
\begin{aligned}
    -\gamma\sum_{t=0}^{T-1}\expect\bP_1^{(t)}
    &\le\frac{L\beta}{2n(1-\beta)}
    \sum_{t=0}^{T-1}\expect\norm{\Delta\bar{\bX}^{(t)}}_F^2 + \frac{\gamma^2nL\beta}{2(1+\beta)}\sigma^2T
    +\frac{\gamma^2L\beta}{2(1-\beta)^2}\sigma^2.
\end{aligned}
\end{equation}

For $\bP_2^{(t)}$, set
\[
\bV^{(s)}:=\sum_{m=0}^{\min\crk{s,D}}
\bpi^\T\bA_m\nabla\bar{\bF}^{(s-m)},\qquad s\ge0.
\]
The common initialization gives $\bV^{(0)}=\mathbf{0}$.
Apply Lemma~\ref{lem:weighted_sum} with weights $\omega_{t,s}$, whose sum
is one. Since $\sum_{t=s}^{T-1}\omega_{t,s}=1-\beta^{T-s}\le1$ for
$s\ge1$, Lemma~\ref{lem:Ak-barF} yields
\[
\begin{aligned}
    \sum_{t=0}^{T-1}\expect\norm{\bQ_2^{(t)}}^2
    &\le\sum_{s=1}^{T-1}(1-\beta^{T-s})
    \expect\norm{\bV^{(s)}}^2
    \le nL^2\sum_{t=0}^{T-1}\expect\norm{\hat{\bX}^{(t)}}_F^2.
\end{aligned}
\]
Young's inequality gives
$-\gamma\expect\bP_2^{(t)}
\le\gamma n\expect\norm{\nabla f(x_1^{(t)})}^2/4
+\gamma\expect\norm{\bQ_2^{(t)}}^2/n$.
Summing over $t$ therefore gives
\begin{equation}
    \label{eq:lem:smooth-3-1}
\begin{aligned}
    -\gamma\sum_{t=0}^{T-1}\expect\bP_2^{(t)}
    &\le\frac{\gamma n}{4}\sum_{t=0}^{T-1}
    \expect\norm{\nabla f(x_1^{(t)})}^2
    +\gamma L^2\sum_{t=0}^{T-1}\expect\norm{\hat{\bX}^{(t)}}_F^2.
\end{aligned}
\end{equation}

For $\bP_3^{(t)}$, telescoping over the tree layers and using
\eqref{eq:tree-rgt-decision-delay} give
\[
\begin{aligned}
    \bW^{(s)}
    &:=\bpi^\T\sum_{m=0}^{\min\crk{s,D}-1}\bPi_{\bC}\bC^m
    \prt{\nabla\bF(\bar{\bX}^{(s-m)})
    -\nabla\bF(\bar{\bX}^{(s-m-1)})}\\
    &=-\sum_{i=1}^{n}\prt{\nabla f_i(x_1^{(s)})
    -\nabla f_i(x_i^{(s)})}^\T
    =\mone^\T\nabla\bar{\bF}^{(s)},\\
    \bB^{(s)}
    &:=\bpi^\T\bPi_{\bC}\bC^s\nabla\bF^{(0)}
    =-\sum_{i:\,h_i>s}\prt{\nabla f_i(x^{(0)})}^\T.
\end{aligned}
\]
Write $\bQ_3^{(t)}=\widetilde{\bW}^{(t)}+\widetilde{\bB}^{(t)}$, where
\[
\widetilde{\bW}^{(t)}:=\sum_{s=0}^{t}\omega_{t,s}\bW^{(s)},
\qquad\widetilde{\bB}^{(t)}:=\sum_{s=0}^{t}\omega_{t,s}\bB^{(s)}.
\]
As for $\bQ_2^{(t)}$, $\bW^{(0)}=\mathbf{0}$ and
Lemma~\ref{lem:weighted_sum} with weights $\omega_{t,s}$ give
\[
\sum_{t=0}^{T-1}\expect\norm{\widetilde{\bW}^{(t)}}^2
\le\sum_{s=1}^{T-1}\expect\norm{\bW^{(s)}}^2
\le nL^2\sum_{t=0}^{T-1}\expect\norm{\hat{\bX}^{(t)}}_F^2.
\]
Reordering the initialization term by node yields
\[
\widetilde{\bB}^{(t)}
=-\sum_{i=2}^{n}a_{t,i}\prt{\nabla f_i(x^{(0)})}^\T,
\qquad
 a_{t,i}:=\begin{cases}
    1,&0\le t<h_i,\\
    \beta^{t-h_i+1},&t\ge h_i.
\end{cases}
\]
Lemma~\ref{lem:matrix_norm}, Lemma~\ref{lem:weighted_sum}, and
$\sum_{i=2}^{n}h_i=nD_{\avg}$ imply
\[
\begin{aligned}
    \sum_{t=0}^{T-1}\norm{\widetilde{\bB}^{(t)}}^2
    &\le\sum_{i=2}^{n}\sum_{t=0}^{T-1}a_{t,i}^2
    \norm{\nabla\bF^{(0)}}_F^2\le\brk{nD_{\avg}+ n\frac{\beta^2}{1-\beta}}
    \norm{\nabla\bF^{(0)}}_F^2.
\end{aligned}
\]
Young's inequality and Corollary~\ref{lem:sum_matrix} with two terms
give
\[
-\gamma\expect\bP_3^{(t)}
\le\frac{\gamma n}{4}\expect\norm{\nabla f(x_1^{(t)})}^2
+\frac{2\gamma}{n}\brk{
\expect\norm{\widetilde{\bW}^{(t)}}^2
+\norm{\widetilde{\bB}^{(t)}}^2}.
\]
Summing over $t$ gives
\begin{equation}
    \label{eq:lem:smooth-4-1}
\begin{aligned}
    -\gamma\sum_{t=0}^{T-1}\expect\bP_3^{(t)}
    &\le\frac{\gamma n}{4}\sum_{t=0}^{T-1}
    \expect\norm{\nabla f(x_1^{(t)})}^2
    +2\gamma L^2\sum_{t=0}^{T-1}\expect\norm{\hat{\bX}^{(t)}}_F^2\\
    &\quad+2\gamma\brk{D_{\avg}
    +\frac{\beta^2}{1-\beta}}
    \norm{\nabla\bF^{(0)}}_F^2.
\end{aligned}
\end{equation}

For $\bP_4^{(t)}$, set
$\bd^{(r)}:=\nabla f(x_1^{(r+1)})-\nabla f(x_1^{(r)})$.
Since $\sum_{s=0}^{r}\omega_{t,s}=\beta^{t-r}$,
\[
\begin{aligned}
    \bQ_4^{(t)}
    &=-n\sum_{r=0}^{t-1}\brk{\sum_{s=0}^{r}\omega_{t,s}}
    \prt{\bd^{(r)}}^\T
    =-n\beta\sum_{r=0}^{t-1}\beta^{t-r-1}\prt{\bd^{(r)}}^\T.
\end{aligned}
\]
Apply Lemma~\ref{lem:weighted_sum} with weights $\beta^{t-r-1}$,
whose sum is at most $1/(1-\beta)$. Smoothness gives
$\norm{\bd^{(r)}}^2\le L^2\norm{\Delta\bar{\bX}^{(r)}}_F^2/n$.
Young's inequality therefore yields
\[
\begin{aligned}
    -\gamma\expect\bP_4^{(t)}
    &\le\frac{\gamma n\beta}{4}\expect\norm{\nabla f(x_1^{(t)})}^2
    +\gamma n\beta\expect\norm{
    \sum_{r=0}^{t-1}\beta^{t-r-1}\bd^{(r)}}^2\\
    &\le\frac{\gamma n}{4}\expect\norm{\nabla f(x_1^{(t)})}^2
    +\frac{\gamma L^2}{1-\beta}
    \sum_{r=0}^{t-1}\beta^{t-r-1}
    \expect\norm{\Delta\bar{\bX}^{(r)}}_F^2.
\end{aligned}
\]
Summing over $t$ and using
$\sum_{t=r+1}^{T-1}\beta^{t-r-1}\le1/(1-\beta)$ gives
\begin{equation}
    \label{eq:lem:smooth-P4}
\begin{aligned}
    -\gamma\sum_{t=0}^{T-1}\expect\bP_4^{(t)}
    &\le\frac{\gamma n}{4}\sum_{t=0}^{T-1}
    \expect\norm{\nabla f(x_1^{(t)})}^2
    +\frac{\gamma L^2}{(1-\beta)^2}
    \sum_{t=0}^{T-1}\expect\norm{\Delta\bar{\bX}^{(t)}}_F^2.
\end{aligned}
\end{equation}
Sum \eqref{eq:lem:smooth-2-2} over $t=0,\ldots,T-1$ and then
substitute \eqref{eq:lem:smooth-P1}--\eqref{eq:lem:smooth-P4}.
Using $f(x_1^{(T)})\ge f^\star$ and setting
\[
\kappa_\beta:=\frac{1}{2(1-\beta)}
+\frac{\gamma nL}{(1-\beta)^2},
\]
we obtain
\begin{equation}
    \label{eq:lem:smooth-6}
\begin{aligned}
    -\Delta_f
    &\le-\frac{\gamma n}{4}
    \sum_{t=0}^{T-1}\expect\norm{\nabla f(x_1^{(t)})}^2
    +3\gamma L^2\sum_{t=0}^{T-1}\expect\norm{\hat{\bX}^{(t)}}_F^2+\frac{\gamma^2nL\beta}{2(1+\beta)}\sigma^2T\\
    &\quad+\frac{L\kappa_\beta}{n}
    \sum_{t=0}^{T-1}\expect\norm{\Delta\bar{\bX}^{(t)}}_F^2
    +2\gamma\brk{D_{\avg}+\frac{\beta^2}{1-\beta}}
    \norm{\nabla\bF^{(0)}}_F^2
    +\frac{\gamma^2L\beta}{2(1-\beta)^2}\sigma^2.
\end{aligned}
\end{equation}
Assume
\begin{equation}
    \label{eq:lem:smooth-step}
    0<\gamma\le\min\crk{\frac{1}{60nDL},\,
    \frac{1-\beta}{60nL}},
    \qquad 0\le\beta\le1-\frac{1}{\sqrt{T}}.
\end{equation}
Since $D_{\avg}\le D$, the stepsize bound
\eqref{eq:lem:smooth-step} ensures the stepsize conditions of
Lemmas~\ref{lem:X_diff} and~\ref{lem:PiX}.
It also yields $\kappa_\beta\le1/(1-\beta)$.
Applying Lemma~\ref{lem:X_diff} to \eqref{eq:lem:smooth-6} gives
\begin{equation}
    \label{eq:lem:smooth-7}
\begin{aligned}
    -\Delta_f
    &\le-\gamma n\brk{\frac{1}{4}-\frac{9\gamma n L}{1-\beta}}
    \sum_{t=0}^{T-1}\expect\norm{\nabla f(x_1^{(t)})}^2 + 3\gamma L^2\brk{1+\frac{3\gamma n L}{1-\beta}}
    \sum_{t=0}^{T-1}\expect\norm{\hat{\bX}^{(t)}}_F^2\\
    &\quad+\gamma^2nL\brk{\frac{9}{1+\beta}
    +\frac{\beta}{2(1+\beta)}}\sigma^2T + \frac{10\gamma^2 L \beta\sigma^2}{(1-\beta)^2}\\
    &\quad+\brk{2\gamma + \frac{18\gamma^2nL}{(1-\beta)}}\prt{D_{\avg}
    +\frac{\beta^2}{1-\beta}}\norm{\nabla\bF^{(0)}}_F^2.
\end{aligned}
\end{equation}
The stepsize bound \eqref{eq:lem:smooth-step} also implies
$1+3\gamma n L/\prt{1-\beta}\le 2$.
Using this bound and applying Lemma~\ref{lem:PiX} to
\eqref{eq:lem:smooth-7} gives
\begin{equation}
    \label{eq:lem:smooth-8}
    \begin{aligned}
    -\Delta_f
    &\le-\gamma n\brk{\frac{1}{4}-\frac{9\gamma n L}{1-\beta} - 90\gamma^2 n^2 D D_{\avg} L^2}\
    \sum_{t=0}^{T-1}\expect\norm{\nabla f(x_1^{(t)})}^2 \\
    &\quad+\gamma^2nL\brk{\frac{9}{1+\beta}
    +\frac{\beta}{2(1+\beta)} + 30\gamma nL\min\crk{2D,4D^2\frac{1-\beta}{1+\beta}}}\sigma^2T \\
    &\quad+\brk{2\gamma + \frac{18\gamma^2nL}{(1-\beta)} + 180\gamma^3 n^2 D^2L^2}\prt{D_{\avg}
    +\frac{\beta^2}{1-\beta}}\norm{\nabla\bF^{(0)}}_F^2\\
    & \quad + \frac{10\gamma^2 L \beta\sigma^2}{(1-\beta)^2} + 30 \gamma^3 n D D_{\avg}\frac{\beta^2}{1-\beta}\sigma^2 L^2.
    \end{aligned}
\end{equation}
The parameter bounds \eqref{eq:lem:smooth-step} further give
\[
\begin{aligned}
    \frac{1}{4}-\frac{9\gamma nL}{1-\beta}
    -90\gamma^2n^2DD_{\avg}L^2&\ge\frac{1}{20},\\
    2\gamma+\frac{18\gamma^2nL}{1-\beta}
    +180\gamma^3n^2D^2L^2&\le3\gamma.
\end{aligned}
\]
Since $T(1-\beta)^2\ge1$, $\gamma nDL\le1/60$, $n\ge2$, and
$D_{\avg}\le D\le n-1$, the total coefficient of $\gamma L\sigma^2$
in \eqref{eq:lem:smooth-8}, after normalization by $\gamma nT/20$,
is bounded by
\[
\begin{aligned}
    &30+\frac{170}{1+\beta}
    +\frac{200\beta}{nT(1-\beta)^2}
    +\frac{600\gamma LDD_{\avg}\beta^2}{T(1-\beta)}\\
    &\quad\le30+\frac{170}{1+\beta}
    +100\beta+10\beta^2(1-\beta)
    \le215\le220.
\end{aligned}
\]
Thus the initial momentum noise is absorbed into
$220\gamma L\sigma^2$. Rearranging \eqref{eq:lem:smooth-8} and
using $\beta^2\le1$ gives
\begin{equation}
    \label{eq:lem:smooth-9}
    \begin{aligned}
        &\frac{1}{T}\sum_{t=0}^{T-1}
        \expect\norm{\nabla f(x_1^{(t)})}^2
        \le\frac{20\Delta_f}{\gamma nT}+220\gamma L\sigma^2+\frac{60}{nT}
        \prt{D_{\avg}+\frac{1}{1-\beta}}
        \norm{\nabla\bF^{(0)}}_F^2.
    \end{aligned}
\end{equation}
This completes the proof.

\bibliographystyle{siamplain}
\bibliography{references_all}

@article{nedic2009distributed,
  title={Distributed subgradient methods for multi-agent optimization},
  author={Nedic, Angelia and Ozdaglar, Asuman},
  journal={IEEE Transactions on Automatic Control},
  volume={54},
  number={1},
  pages={48--61},
  year={2009},
  publisher={IEEE}
}

@article{pu2020asymptotic,
  title={Asymptotic network independence in distributed stochastic optimization for machine learning: Examining distributed and centralized stochastic gradient descent},
  author={Pu, Shi and Olshevsky, Alex and Paschalidis, Ioannis Ch},
  journal={IEEE signal processing magazine},
  volume={37},
  number={3},
  pages={114--122},
  year={2020},
  publisher={IEEE}
}

@article{xi2017add,
	title={ADD-OPT: Accelerated distributed directed optimization},
	author={Xi, Chenguang and Xin, Ran and Khan, Usman A},
	journal={IEEE Transactions on Automatic Control},
	volume={63},
	number={5},
	pages={1329--1339},
	year={2017},
	publisher={IEEE}
}

@article{nedic2015distributed,
	title={Distributed optimization over time-varying directed graphs},
	author={Nedi{\'c}, Angelia and Olshevsky, Alex},
	journal={IEEE Transactions on Automatic Control},
	volume={60},
	number={3},
	pages={601--615},
	year={2015},
	publisher={IEEE},
	doi={10.1109/TAC.2014.2364096}
}

@article{yuan2023removing,
  title={Removing data heterogeneity influence enhances network topology dependence of decentralized sgd},
  author={Yuan, Kun and Alghunaim, Sulaiman A and Huang, Xinmeng},
  journal={Journal of Machine Learning Research},
  volume={24},
  number={280},
  pages={1--53},
  year={2023}
}

@article{pu2021sharp,
  title={A sharp estimate on the transient time of distributed stochastic gradient descent},
  author={Pu, Shi and Olshevsky, Alex and Paschalidis, Ioannis Ch},
  journal={IEEE Transactions on Automatic Control},
  volume={67},
  number={11},
  pages={5900--5915},
  year={2021},
  publisher={IEEE}
}

@article{huang2022improving,
  title={Improving the transient times for distributed stochastic gradient methods},
  author={Huang, Kun and Pu, Shi},
  journal={IEEE Transactions on Automatic Control},
  volume={68},
  number={7},
  pages={4127--4142},
  year={2022},
  publisher={IEEE},
  doi={10.1109/TAC.2022.3201141}
}

@article{xin2020general,
  title={A general framework for decentralized optimization with first-order methods},
  author={Xin, Ran and Pu, Shi and Nedi{\'c}, Angelia and Khan, Usman A},
  journal={Proceedings of the IEEE},
  volume={108},
  number={11},
  pages={1869--1889},
  year={2020},
  publisher={IEEE}
}

@article{koloskova2021improved,
  title={An improved analysis of gradient tracking for decentralized machine learning},
  author={Koloskova, Anastasiia and Lin, Tao and Stich, Sebastian U},
  journal={Advances in Neural Information Processing Systems},
  volume={34},
  pages={11422--11435},
  year={2021}
}

@article{xiao2004fast,
  title={Fast linear iterations for distributed averaging},
  author={Xiao, Lin and Boyd, Stephen},
  journal={Systems \& Control Letters},
  volume={53},
  number={1},
  pages={65--78},
  year={2004},
  publisher={Elsevier}
}

@article{nedic2017achieving,
	title={Achieving geometric convergence for distributed optimization over time-varying graphs},
	author={Nedic, Angelia and Olshevsky, Alex and Shi, Wei},
	journal={SIAM Journal on Optimization},
	volume={27},
	number={4},
	pages={2597--2633},
	year={2017},
	publisher={SIAM}
}

@article{nedic2018network,
  title={Network topology and communication-computation tradeoffs in decentralized optimization},
  author={Nedi{\'c}, Angelia and Olshevsky, Alex and Rabbat, Michael G},
  journal={Proceedings of the IEEE},
  volume={106},
  number={5},
  pages={953--976},
  year={2018},
  publisher={IEEE}
}

@article{pu2020push,
  title={Push--pull gradient methods for distributed optimization in networks},
  author={Pu, Shi and Shi, Wei and Xu, Jinming and Nedi{\'c}, Angelia},
  journal={IEEE Transactions on Automatic Control},
  volume={66},
  number={1},
  pages={1--16},
  year={2021},
  publisher={IEEE},
  doi={10.1109/TAC.2020.2972824}
}

@article{ying2021exponential,
  title={Exponential graph is provably efficient for decentralized deep training},
  author={Ying, Bicheng and Yuan, Kun and Chen, Yiming and Hu, Hanbin and Pan, Pan and Yin, Wotao},
  journal={Advances in Neural Information Processing Systems},
  volume={34},
  pages={13975--13987},
  year={2021}
}

@article{hoory2006expander,
  title={Expander Graphs and Their Applications},
  author={Hoory, Shlomo and Linial, Nathan and Wigderson, Avi},
  journal={Bulletin of the American Mathematical Society},
  volume={43},
  number={4},
  pages={439--561},
  year={2006},
  doi={10.1090/S0273-0979-06-01126-8}
}

@article{vogels2021relaysum,
  title={Relaysum for decentralized deep learning on heterogeneous data},
  author={Vogels, Thijs and He, Lie and Koloskova, Anastasiia and Karimireddy, Sai Praneeth and Lin, Tao and Stich, Sebastian U and Jaggi, Martin},
  journal={Advances in Neural Information Processing Systems},
  volume={34},
  pages={28004--28015},
  year={2021}
}

@article{pu2021distributed,
  title={Distributed stochastic gradient tracking methods},
  author={Pu, Shi and Nedi{\'c}, Angelia},
  journal={Mathematical Programming},
  volume={187},
  pages={409--457},
  year={2021},
  publisher={Springer}
}

@article{lian2017can,
  title={Can decentralized algorithms outperform centralized algorithms? a case study for decentralized parallel stochastic gradient descent},
  author={Lian, Xiangru and Zhang, Ce and Zhang, Huan and Hsieh, Cho-Jui and Zhang, Wei and Liu, Ji},
  journal={Advances in neural information processing systems},
  volume={30},
  year={2017}
}

@inproceedings{ding2023dsgd,
  title={DSGD-CECA: Decentralized SGD with Communication-Optimal Exact Consensus Algorithm},
  author={Ding, Lisang and Jin, Kexin and Ying, Bicheng and Yuan, Kun and Yin, Wotao},
  booktitle={International Conference on Machine Learning},
  pages={8067--8089},
  year={2023},
  organization={PMLR}
}

@article{lin2021quasi,
  title={Quasi-global momentum: Accelerating decentralized deep learning on heterogeneous data},
  author={Lin, Tao and Karimireddy, Sai Praneeth and Stich, Sebastian U and Jaggi, Martin},
  journal={arXiv preprint arXiv:2102.04761},
  year={2021}
}

@article{shi2015extra,
  title={Extra: An exact first-order algorithm for decentralized consensus optimization},
  author={Shi, Wei and Ling, Qing and Wu, Gang and Yin, Wotao},
  journal={SIAM Journal on Optimization},
  volume={25},
  number={2},
  pages={944--966},
  year={2015},
  publisher={SIAM}
}

@inproceedings{tang2018d,
  title={$D^2$: Decentralized training over decentralized data},
  author={Tang, Hanlin and Lian, Xiangru and Yan, Ming and Zhang, Ce and Liu, Ji},
  booktitle={International Conference on Machine Learning},
  pages={4848--4856},
  year={2018},
  organization={PMLR}
}

@article{nguyen2025graphs,
  title={On graphs with finite-time consensus and their use in gradient tracking},
  author={Nguyen, Edward Duc Hien and Jiang, Xin and Ying, Bicheng and Uribe, C{\'e}sar A},
  journal={SIAM Journal on Optimization},
  volume={35},
  number={2},
  pages={872--898},
  year={2025},
  publisher={SIAM}
}

@article{you2024b,
  title={B-ary Tree Push-Pull Method is Provably Efficient for Distributed Learning on Heterogeneous Data},
  author={You, Runze and Pu, Shi},
  journal={Advances in Neural Information Processing Systems},
  volume={37},
  pages={97523--97561},
  year={2024}
}

@article{alghunaim2022unified,
  title={A unified and refined convergence analysis for non-convex decentralized learning},
  author={Alghunaim, Sulaiman A and Yuan, Kun},
  journal={IEEE Transactions on Signal Processing},
  volume={70},
  pages={3264--3279},
  year={2022},
  publisher={IEEE}
}

@article{song2024optimal,
  title={Optimal gradient tracking for decentralized optimization},
  author={Song, Zhuoqing and Shi, Lei and Pu, Shi and Yan, Ming},
  journal={Mathematical Programming},
  volume={207},
  number={1},
  pages={1--53},
  year={2024},
  publisher={Springer}
}

@article{kovalev2020optimal,
  title={Optimal and Practical Algorithms for Smooth and Strongly Convex Decentralized Optimization},
  author={Kovalev, Dmitry and Salim, Adil and Richt{\'a}rik, Peter},
  journal={Advances in Neural Information Processing Systems},
  volume={33},
  pages={18342--18352},
  year={2020}
}

@article{liang2025linear,
  title={On the Linear Speedup of the Push-Pull Method for Decentralized Optimization over Digraphs},
  author={Liang, Liyuan and Luo, Gan and Yuan, Kun},
  journal={arXiv preprint arXiv:2506.18075},
  year={2025}
}

@inproceedings{lu2021optimal,
  title={Optimal complexity in decentralized training},
  author={Lu, Yucheng and De Sa, Christopher},
  booktitle={International conference on machine learning},
  pages={7111--7123},
  year={2021},
  organization={PMLR}
}

@article{carmon2021lower,
  title={Lower bounds for finding stationary points {II}: First-order methods},
  author={Carmon, Yair and Duchi, John C. and Hinder, Oliver and Sidford, Aaron},
  journal={Mathematical Programming},
  volume={185},
  number={1--2},
  pages={315--355},
  year={2021},
  publisher={Springer},
  doi={10.1007/s10107-019-01431-x}
}

@article{ghadimi2013stochastic,
  title={Stochastic first- and zeroth-order methods for nonconvex stochastic programming},
  author={Ghadimi, Saeed and Lan, Guanghui},
  journal={SIAM Journal on Optimization},
  volume={23},
  number={4},
  pages={2341--2368},
  year={2013},
  publisher={SIAM}
}

@article{nesterov1983method,
  title={A method for unconstrained convex minimization problem with the rate of convergence $O(1/k^2)$},
  author={Nesterov, Yurii E.},
  journal={Doklady Akademii Nauk SSSR},
  volume={269},
  number={3},
  pages={543--547},
  year={1983}
}

@article{kingma2014adam,
  title={{Adam}: A method for stochastic optimization},
  author={Kingma, Diederik P. and Ba, Jimmy},
  journal={arXiv preprint arXiv:1412.6980},
  year={2014}
}

@article{tieleman2012lecture,
  title={Lecture 6.5---{RMSProp}: Divide the gradient by a running average of its recent magnitude},
  author={Tieleman, Tijmen and Hinton, Geoffrey},
  journal={COURSERA: Neural Networks for Machine Learning},
  volume={4},
  number={2},
  pages={26--31},
  year={2012}
}

@article{ward2018adagrad,
  title={{AdaGrad} stepsizes: Sharp convergence over nonconvex landscapes, from any initialization},
  author={Ward, Rachel and Wu, Xiaoxia and Bottou, Leon},
  journal={arXiv preprint arXiv:1806.01811},
  year={2018}
}

@article{zeiler2012adadelta,
  title={{AdaDelta}: An adaptive learning rate method},
  author={Zeiler, Matthew D.},
  journal={arXiv preprint arXiv:1212.5701},
  year={2012}
}

@inproceedings{scaman2017optimal,
  title={Optimal algorithms for smooth and strongly convex distributed optimization in networks},
  author={Scaman, Kevin and Bach, Francis and Bubeck, S{\'e}bastien and Lee, Yin Tat and Massouli{\'e}, Laurent},
  booktitle={International Conference on Machine Learning},
  pages={3027--3036},
  year={2017},
  organization={PMLR}
}

@article{boyd2004fastest,
  title={Fastest mixing Markov chain on a graph},
  author={Boyd, Stephen and Diaconis, Persi and Xiao, Lin},
  journal={SIAM Review},
  volume={46},
  number={4},
  pages={667--689},
  year={2004},
  publisher={SIAM}
}

@article{yuan2022decentralized,
  title={Decentralized training of foundation models in heterogeneous environments},
  author={Yuan, Binhang and He, Yongjun and Davis, Jared and Zhang, Tianyi and Dao, Tri and Chen, Beidi and Liang, Percy S and Re, Christopher and Zhang, Ce},
  journal={Advances in Neural Information Processing Systems},
  volume={35},
  pages={25464--25477},
  year={2022}
}

@article{yuan2022revisiting,
  title={Revisiting optimal convergence rate for smooth and non-convex stochastic decentralized optimization},
  author={Yuan, Kun and Huang, Xinmeng and Chen, Yiming and Zhang, Xiaohan and Zhang, Yingya and Pan, Pan},
  journal={Advances in Neural Information Processing Systems},
  volume={35},
  pages={36382--36395},
  year={2022}
}

@article{huang2026accelerated,
  title={An accelerated distributed stochastic gradient method with momentum},
  author={Huang, Kun and Pu, Shi and Nedi{\'c}, Angelia},
  journal={Mathematical Programming},
  volume={215},
  number={1--2},
  pages={193--236},
  year={2026},
  publisher={Springer},
  doi={10.1007/s10107-025-02217-0}
}

@article{erdos1959random,
  title={On Random Graphs. I},
  author={Erd{\H{o}}s, Paul and R{\'e}nyi, Alfr{\'e}d},
  journal={Publicationes Mathematicae Debrecen},
  volume={6},
  pages={290--297},
  year={1959}
}

@article{you2025distributed,
  title={Distributed learning over arbitrary topology: Linear speed-up with polynomial transient time},
  author={You, Runze and Pu, Shi},
  journal={arXiv preprint arXiv:2503.16123},
  year={2025}
}

@article{liang2025understanding,
  title={Understanding the influence of digraphs on decentralized optimization: Effective metrics, lower bound, and optimal algorithm},
  author={Liang, Liyuan and Huang, Xinmeng and Xin, Ran and Yuan, Kun},
  journal={SIAM Journal on Optimization},
  volume={35},
  number={3},
  pages={1570--1600},
  year={2025},
  publisher={SIAM},
  doi={10.1137/24M1657341}
}

@article{nedic2025ab,
  title={{AB/Push-Pull} method for distributed optimization in time-varying directed networks},
  author={Nedi{\'c}, Angelia and Nguyen, Duong Thuy Anh and Nguyen, Duong Tung},
  journal={Optimization Methods and Software},
  volume={40},
  number={5},
  pages={1044--1071},
  year={2025},
  publisher={Taylor \& Francis},
  doi={10.1080/10556788.2023.2261602}
}

@inproceedings{liang2025row,
  title={Achieving linear speedup and near-optimal complexity for decentralized optimization over row-stochastic networks},
  author={Liang, Liyuan and Chen, Xinyi and Luo, Gan and Yuan, Kun},
  booktitle={Proceedings of the 42nd International Conference on Machine Learning},
  series={Proceedings of Machine Learning Research},
  volume={267},
  pages={37100--37130},
  year={2025},
  publisher={PMLR},
  url={https://proceedings.mlr.press/v267/liang25c.html}
}

@article{kluger2026nearoptimal,
  title={Near-optimal decentralized stochastic convex optimization over networks},
  author={Kluger, Nitai and Attia, Amit and Koren, Tomer},
  journal={arXiv preprint arXiv:2606.04757},
  year={2026},
  doi={10.48550/arXiv.2606.04757}
}

@article{sun2026accelerated,
  title={Accelerated decentralized stochastic gradient descent for strongly convex optimization},
  author={Sun, Ming and Yuan, Kun},
  journal={arXiv preprint arXiv:2606.07496},
  year={2026},
  doi={10.48550/arXiv.2606.07496}
}

@article{you2026stochastic,
  title={Stochastic Push--Pull for Decentralized Nonconvex Optimization},
  author={You, Runze and Pu, Shi},
  journal={IEEE Transactions on Signal Processing},
  volume={74},
  pages={1383--1398},
  year={2026},
  publisher={IEEE},
  doi={10.1109/TSP.2026.3675119}
}

@techreport{krizhevsky2009learning,
  title={Learning multiple layers of features from tiny images},
  author={Krizhevsky, Alex and Hinton, Geoffrey E.},
  year={2009},
  institution={University of Toronto}
}

@article{diaconis1991geometric,
  title={Geometric Bounds for Eigenvalues of {Markov} Chains},
  author={Diaconis, Persi and Stroock, Daniel},
  journal={The Annals of Applied Probability},
  volume={1},
  number={1},
  pages={36--61},
  year={1991},
  doi={10.1214/aoap/1177005980}
}

@article{chung1989diameters,
  title={Diameters and eigenvalues},
  author={Chung, Fan R. K.},
  journal={Journal of the American Mathematical Society},
  volume={2},
  number={2},
  pages={187--196},
  year={1989},
  doi={10.1090/S0894-0347-1989-0965008-X}
}

@book{chung1997spectral,
  title={Spectral Graph Theory},
  author={Chung, Fan R. K.},
  series={CBMS Regional Conference Series in Mathematics},
  volume={92},
  publisher={American Mathematical Society},
  address={Providence, RI},
  year={1997},
  doi={10.1090/cbms/092}
}

@article{chung1994upper,
  title={An Upper Bound on the Diameter of a Graph from Eigenvalues Associated with Its {Laplacian}},
  author={Chung, F. R. K. and Faber, V. and Manteuffel, Thomas A.},
  journal={SIAM Journal on Discrete Mathematics},
  volume={7},
  number={3},
  pages={443--457},
  year={1994},
  doi={10.1137/S0895480191217776}
}

\end{document}